\documentclass[11pt]{article}
\usepackage{amsmath,amsthm,amsfonts,amssymb}
\usepackage{graphicx}
\usepackage{mathtools}
\usepackage{enumerate}
\usepackage{xcolor}
\usepackage[colorlinks=true,linkcolor=blue,citecolor=teal]{hyperref}
\usepackage{bbm}

\newcommand{\Rat}{\mathrm{Rat}}
\newcommand{\ring}{\mathrm{ring}}
\newcommand{\scoop}{\mathrm{Scoop}}

\newcommand{\Trunk}{\mathrm{Trunk}}
\renewcommand{\root}{\emptyset}

\newcommand{\out}{\mathrm{out}}
\newcommand{\per}{\mathrm{per}}
\newcommand{\TPerm}{\mathrm{TPerm}}
\newcommand{\CRT}{\mathrm{CRT}}

\newcommand{\BGW}{\mathrm{BGW}}
\newcommand{\CLE}{\mathrm{CLE}}

\newcommand{\Vol}{\mathrm{Vol}}

\newcommand{\ttF}{\mathtt{F}}
\newcommand{\cU}{\mathcal{U}}

\newcommand{\Loop}{\mathrm{Loop}}

\newcommand{\RR}{\mathbb{R}}

\newcommand{\CC}{\mathbb{C}}
\newcommand{\PP}{\mathbb{P}}
\newcommand{\EE}{\mathbb{E}}

\newcommand{\bchi}{\pmb{\chi}}

\newcommand{\NN}{\mathbb{N}}
\newcommand{\inn}{\mathrm{in}}
\newcommand{\inner}{_{\infty}^\inn}

\renewcommand{\epsilon}{\varepsilon}
\newcommand{\cL}{\mathbf{c}_{\mathrm{L}}}
\newcommand{\bbe}{\mathbbm{e}}

\usepackage[letterpaper,hmargin=1in,vmargin=1in]{geometry}
\numberwithin{equation}{section}

\newtheorem{theorem}{Theorem}[section]
\newtheorem{lemma}[theorem]{Lemma}

\newtheorem{corollary}[theorem]{Corollary}
\newtheorem{proposition}[theorem]{Proposition}

\theoremstyle{definition}
\newtheorem{definition}[theorem]{Definition}
\newtheorem{remark}[theorem]{Remark}

\newcommand{\ind}{\mathbf{1}}

\newcommand{\cR}{\mathcal{R}}

\newcommand{\fF}{\mathfrak{F}}

\newcommand{\cA}{\mathcal{A}}

\newcommand{\ZZ}{\mathbb{Z}}
\newcommand{\ext}{\mathrm{ext}}
\newcommand{\dis}{\mathrm{dis}}

\newcommand{\cT}{\mathcal{T}}

\newcommand{\bfq}{\mathbf{q}}
\newcommand{\unif}{\mathrm{unif}}

\begin{document}

\title{Area measures and branched polymers in \\supercritical Liouville quantum gravity}
 \date{ }
 \author{
\begin{tabular}{c} Manan Bhatia\\[-5pt]\small MIT \end{tabular}
\begin{tabular}{c} Ewain Gwynne\\[-5pt]\small University of Chicago \end{tabular}  
\begin{tabular}{c} Jinwoo Sung\\[-5pt]\small University of Chicago \end{tabular}  
}
\maketitle

\begin{abstract}
  We study Liouville quantum gravity (LQG) in the supercritical (a.k.a.\ strongly coupled) phase, which has background charge $Q \in (0,2)$ and central charge $\mathbf c_{\mathrm{L}} = 1+6Q^2 \in (1,25)$. Recent works have shown how to define LQG in this phase as a planar random geometry associated with a variant of the Gaussian free field, which exhibits ``infinite spikes." In contrast, a number of results from physics, dating back to the 1980s, suggest that supercritical LQG surfaces should behave like ``branched polymers": i.e., they should look like the continuum random tree.

  We prove a result which reconciles these two descriptions of supercritical LQG. More precisely, we show that for a family of random planar maps with boundary in the universality class of supercritical LQG, if we condition on the (small probability) event that the planar map is finite, then the scaling limit is the continuum random tree. 
  
  We also show that there does not exist any locally finite measure associated with supercritical LQG which is locally determined by the field and satisfies the LQG coordinate change formula. Our proofs are based on a branching process description of supercritical LQG which comes from its coupling with CLE$_4$ (Ang and Gwynne, 2023).
\end{abstract}
\tableofcontents
\section{Introduction}
\label{sec:introduction}

Liouville quantum gravity (LQG) is a canonical one-parameter family of random geometries on an orientable surface with given topology. Conceived by Polyakov \cite{Pol81}, LQG was studied initially in the physics literature in the context of string theory and two-dimensional quantum gravity and since has become an active area of mathematical research. 

LQG is indexed by the parameter $\mathbf{c}_{\mathrm L} > 1$ called the \textbf{(Liouville) central charge}. (Some works on LQG alternatively consider the \textbf{matter central charge} $\mathbf c_{\mathrm M} = 26 - \mathbf c_{\mathrm L}$.) Heuristically speaking, an LQG surface of given topology is described by a random metric $g$ sampled from the ``uniform measure on the space of Riemannian metric tensors weighted by $(\det \Delta_g)^{-(26-\mathbf c_{\mathrm L})/2}$," where $\Delta_g$ is the Laplace--Beltrami operator of the surface with the metric $g$. While this definition does not make literal sense, there are various ways to define LQG rigorously that we discuss below. These definitions utilize alternative parameters, with common choices being the \textbf{background charge} $Q>0$ and the \textbf{coupling constant} $\gamma \in (0,2]\cup \{z\in \CC: |z|=2\}$ related by the formulae
\begin{equation}\label{eq:parameters}
    \mathbf{c}_{\mathrm L} = 1+6Q^2 \quad \text{and} \quad Q = \frac{2}{\gamma} + \frac{\gamma}{2}.
\end{equation}
The geometric properties of an LQG surface depend heavily on the value of its central charge, which can be classified into subcritical, critical, and supercritical phases as summarized in Table~\ref{table:1}. 
\begin{table} 
\renewcommand\baselinestretch{1.3}\selectfont
\centering
\begin{tabular}{ |c||c|c|c| } 
\hline 
& Central charge $\mathbf c_{\mathrm L}$ & Background charge $Q$ & Coupling constant $\gamma$\\
 \hline
 Subcritical &  $(25,\infty)$ & $(2,\infty)$ & $(0,2)$ \\ 
 \hline
 Critical & 25 & 2 & 2 \\ 
 \hline
 Supercritical & $(1,25)$ & $(0,2)$ & $\gamma \in \mathbb C\setminus \mathbb R$ with $|\gamma|=2$ \\ 
 \hline 
\end{tabular}
\label{table:1}
\bigskip
\caption{The phases of LQG and their parameter ranges.}
\end{table}
In the physics literature, LQG is said to be \textit{weakly coupled} in the subcritical and critical phases ($\mathbf c_{\mathrm L} \geq 25)$ and \textit{strongly coupled} in the supercritical phase ($\mathbf c_{\mathrm L} \in (1,25)$). In this paper, we are primarily interested in the latter case.

Over the past two decades, the subcritical and critical phases of LQG have been studied at length in the mathematical community. In these cases, the so-called  \textit{DDK ansatz} \cite{Dav88,DK89} states that, when defined on a domain $U\subseteq \CC$, the metric tensor associated with LQG should take the form
\begin{equation}\label{eq:ddk}
    g = e^{\gamma h}(dx^2 + dy^2)
\end{equation}
where $dx^2+dy^2$ is the Euclidean metric tensor and $h$ is a variant of the Gaussian free field (GFF) on $U$. (We assume that the reader is familiar with the GFF and we refer to the surveys \cite{She07, WP21, BP24} for detailed introductions to it.)
A GFF is rigorously defined only as a random generalized function and thus the metric tensor \eqref{eq:ddk} does not make a priori sense. Nevertheless, the geometry associated with \eqref{eq:ddk} can be made rigorous through a renormalization procedure approximating $h$ by a sequence of continuous functions. Indeed, as the limit of such approximations, the LQG area measure was constructed in \cite{DS11} as a Gaussian multiplicative chaos measure (see also \cite{Kah,RV08}). Furthermore, the distance function associated with LQG (or simply, the LQG metric) was constructed in a series of works culminating in \cite{DDDF20,GM21} again by a renormalization procedure involving smooth approximations to the GFF $h$. 
We also note the vast literature on the deep connections between LQG in the subcritical and critical phases and various mathematical objects including Schramm--Loewner evolution (SLE), random planar maps, conformal field theory, and random permutations; we direct the reader to the surveys \cite{Gwy-notices,GHS19,She22,BP24,GKR24} as starting points.

\subsection{Physics predictions on the behavior of supercritical LQG and the main contributions of this paper}
\label{sec:physics}
In the context of string theory, it is the supercritical phase of LQG which is expected to be physically relevant: indeed, the regime $\mathbf c_{\mathrm L} \in \{1,\dots,25\}$ is the one that was originally considered by Polyakov \cite{Pol81} to model the worldsheets swept out by strings moving in $\mathbb R^{26-\mathbf c_{\mathrm L}}$. The reader may refer to \cite[Remark~1.4]{AG23} for further explanation on the above-mentioned string theory connection. Moreover, given the recent works \cite{CJ16,Jaf16,Cha19a,Cha19b,MP19,MP24,CPS23} which have shown that Wilson loop observables in lattice gauge theory can be represented as sums over surfaces in $\RR^d$, it is plausible that supercritical LQG may be relevant to Yang--Mills theory. To be specific, the surfaces appearing in the surface-sum representation for $U(N)$ lattice gauge theory in the $N \rightarrow \infty$ limit might be related to supercritical LQG; we note such connections are still speculative and refer the reader to \cite[Section~7]{CPS23} for more on this.

However, when compared to the subcritical and critical phases, the supercritical phase of LQG has stayed much more mysterious. There are a number of works in the physics literature painting the underlying geometry of supercritical LQG in different lights.

\begin{enumerate}[(a)]
\item \label{item:infinite-spikes} One line of thought (e.g., \cite{Sei90}, see also \cite[Figure 10]{GM93}) suggests that when extending the conformal field theory to the regime $\mathbf c_{\mathrm L} \in (1,25)$, the geometry of the surface is torn apart by \textbf{``infinite spikes"} at a dense set of points. In fact, this picture of supercritical LQG is consistent with the mathematical literature \cite{GHPR20,DG23-tightness,DG23,Pfe21,AG23}, where it is represented as a random geometry constructed from a variant of the GFF with a dense set of singular points which are at infinite distance from all other points.
\item \label{item:branched-polymer}  A drastically different viewpoint is taken in the works \cite{DFJ84,ADFO86,Cat88,BH92,BJ92,CKR92,ADJT93,Dav97}, where, by often using numerical and heuristic analyses of random planar maps weighted by $(\det \Delta_{\mathrm{graph}})^{(\mathrm c_{\mathrm L}-26)/2}$ as an approximation of LQG, they suggest that LQG in the supercritical regime should look like a continuum random tree, or as it is often called in the physics literature, a \textbf{branched polymer}.
\item \label{item:Gervais} Yet another picture arises in the works of Gervais et al.\ \cite{GN84,GN85,GR94,GR96,CGS97,Ger94,Ger95,Ger97}, where, based on an algebraic study of the path integral formulation of LQG, it is claimed that there are some special integrable algebraic structures associated with supercritical LQG at the specific values $\cL=7,13,19$. Moreover, for these values of $\cL$, these authors define finite local operators which gives rise to a so-called ``area element'' (e.g., (4.3) in \cite{GR94}), which could conceivably be interpreted as a measure associated with LQG.

\item There are also purely algebraic approaches to supercritical LQG based on conformal field theory. See, e.g., \cite{ribault-cft} for an introduction to this approach.
\end{enumerate}

While the branched polymer picture seems to have the most support in the physics literature, there have been no mathematically rigorous results in support of this view prior to this work. 
In the works \cite{GHPR20,APPS22}, a possible reconciliation of the branched polymer and the infinite spikes pictures was suggested. Namely, it was suggested that, while a supercritical LQG surface a priori has infinite spikes at a dense set of points, it degenerates to a branched polymer when ``conditioned'' on the zero-probability event of having no infinite spikes. However, previous to our work, it was unclear whether this heuristic explanation is correct. Furthermore, it was not clear how to make this idea rigorous, especially given that we need to condition on the zero-probability event of a supercritical LQG surface having no infinite spikes at all. 

Our goal in this work is to present a rigorous formulation of a supercritical LQG surface reconciling the two viewpoints (\ref{item:infinite-spikes}) and (\ref{item:branched-polymer}). Moreover, in disagreement with the claim in (\ref{item:Gervais}), we establish that it is not possible to construct a measure associated with supercritical LQG for any value of $\mathbf c_\mathrm{L}\in (1,25)$: this was unclear from the previous mathematical literature (see \cite[Questions~6.1--2]{GHPR20}). We now give a quick overview of the main results of this paper. 

\begin{enumerate}
\item We establish a link between the branched polymer and infinite spikes pictures via a discrete model of supercritical LQG. We consider a model of random planar maps with boundary (see Section~\ref{sec:discrete-model}) which can be expected to converge to supercritical LQG in the scaling limit (see Theorem~\ref{thm:loop-convergence}). While a map in this model has a finite perimeter $p\in \NN$, it has infinitely many vertices (and thereby ``infinite spikes") with positive probability. In fact, the probability of the map with perimeter $p$ being infinite increases to $1$ as $p\rightarrow \infty$ (Proposition~\ref{thm:2}). We establish that upon \textit{conditioning} this planar map on the rare event that it contains finitely many vertices, we obtain a continuum random tree in the scaling limit as $p\rightarrow \infty$ (Theorem~\ref{thm:1}). 
The physics papers referenced in (\ref{item:branched-polymer}) considered \textit{finite} planar maps (e.g., with a fixed number of edges), so they were implicitly conditioning the surface to be finite.
\item One interpretation of the ``area element'' described in (\ref{item:Gervais}) above is that for the values $\cL=7,13,19$, a supercritical LQG surface has a locally finite volume form associated to it. In other words, given a GFF $h$, one would expect to find a locally finite Borel measure $\mathfrak{m}_h$ which is locally determined by and compatible with the geometry of a supercritical LQG surface associated with $h$. Nonetheless, in Theorem \ref{thm:3}, we show that there does not exist any natural notion of a locally finite and locally determined volume measure for the entire range $\cL \in (1,25)$ of supercritical LQG.
\end{enumerate}

The central idea of our proofs is to exploit the multiplicative cascade structure of supercritical LQG, which comes from its coupling with nested CLE$_4$ \cite{AG23}. This is closely related to the scaling limit of the analogous structure for the loop $O(n)$ model identified in \cite{CCM20} based on the gasket decomposition of \cite{BBG11}. The crucial difference is that in our model, even in the discrete version introduced in the next subsection, there are infinitely many loops of macroscopic length. This is the first paper to use multiplicative cascades to prove properties of objects related to supercritical LQG. We expect that this idea will have many additional applications in the future.

We now introduce our results in further detail, first giving a precise description of our random planar map model of supercritical LQG. This is followed by statements of the main results of this paper and the ideas of their proofs.

\subsection{A discrete model of supercritical LQG}
\label{sec:discrete-model}
As mentioned above, we give a rigorous meaning to ``conditioning a supercritical LQG surface to have no infinite spikes" by considering a discrete model in the universality class of supercritical LQG, for which there is a natural notion of conditioning it to have finite size. In particular, we consider a random loop-decorated planar map which is a discrete analog of the supercritical LQG disk decorated with a conformal loop ensemble of parameter $\kappa=4$ ($\mathrm{CLE}_4$) as constructed recently in \cite{AG23}; see Section~\ref{sec:2} for an introduction to this continuum model. This discrete model is a variant of planar maps decorated by the loop $O(2)$ model where we force the loops to have discrepancies between their outer and inner lengths. This is a slight modification of the model suggested in \cite{AG23}, which is constructed iteratively from the gasket decomposition of the loop $O(n)$ model \cite{BBG11}. We refer the reader to Remark \ref{rem:model} for a further discussion on the relations between these discrete models.

We begin with a few basic definitions.
A \textbf{planar map} is a planar graph embedded into the Riemann sphere considered up to orientation-preserving homeomorphisms. Our planar map $M$ will have a designated \textbf{root face}, which we view as the ``outside'' of a planar map which has the ``topology'' of a disk. The \textbf{boundary} $\partial M$ of the map $M$ is the subgraph of $M$ consisting of vertices and edges incident to the root face. The non-root faces of $M$ are its \textbf{interior faces}, denoted as $\fF(M)$. Given a face $f$, we use $\per(f)$ to denote the \textbf{half-perimeter} of $f$. That is, $2\,\per(f)$ is equal to the number of edges incident to $f$, where an edge is counted twice if it is incident to $f$ on both of its sides (see, e.g., the root face $f$ in Figure \ref{fig:model}(a) with $\per(f)=13$).

We shall construct the model inductively, with two types of random planar maps inserted into the existing planar map one after another. 
\begin{itemize}
\item The first kind is a rooted bipartite\footnote{A graph is called bipartite if all of its faces have even degrees. We choose to work with bipartite Boltzmann maps so as to utilize the random walk bijections of \cite{BDG04} and \cite{JS15} (see Proposition~\ref{prop:10}). We expect our results to hold in a greater generality without assuming the planar maps to be bipartite.} \textbf{Boltzmann map} conditioned on a fixed boundary length, which should be understood as the discrete counterpart of the CLE$_4$ gasket. 
We fix a critical non-generic weight sequence $\bfq=(q_i)_{i\in \NN}$ of type ${a=2}$, and choose our maps from the probability measure $\PP_\bfq^{(p)}$ on rooted bipartite planar maps with boundary perimeter $2p$. For example, we can take $\bfq$ to be a weight sequence corresponding to the gasket of a version of the critical $O(2)$ loop model-decorated quadrangulation (see the case $g=h/2$ in \cite[Appendix~B]{ADH24}).

We define these precisely in Section~\ref{sec:maps-def}, and at this point we simply note that a Boltzmann map sampled from the law $\PP_{\bfq}^{(p)}$ is expected to converge in the $p\to\infty$ scaling limit to the gasket of a CLE$_4$ on an independent critical LQG disk. (See Lemma \ref{lem:13} for a rigorous result on the convergence of perimeters of faces to length of CLE$_4$ loops in support of this). We also emphasize that these Boltzmann maps are not required to have simple boundaries; for example, we may sample from $\PP_\bfq^{(p)}$ a tree with $p$ edges, which, we note, has no interior faces ($\fF(M)=\varnothing$).

\item The second kind of random planar map that we need is a \textbf{``ring"} with two distinguished faces --- an outer face $f_{\mathrm{out}}$ and an inner face $f_{\mathrm{in}}$ --- such that both $f_\out$ and $f_\inn$ have even degrees. 
See Figure~\ref{fig:ring} for an example of a ring. These rings are inserted into each interior face of a Boltzmann map. They are analogous to the loops in the loop $O(n)$ model (see \cite{BBG11}) on the dual map, where the two faces of a ring correspond to the inside and outside of a loop. 

Let $Q\in (0,2)$ be the background charge associated with the given universality class of supercritical LQG and define \begin{equation}\label{eq:beta-Q}
      \beta_Q :=\frac{\pi \sqrt{4-Q^2}}{Q}.
    \end{equation}
We fix a sequence of probability measures $\big(\PP^{(p)}_{\ring}\big)_{p \in \NN}$ of rings such that under $\PP_\ring^{(p)}$, we have $\per(f_\out) = p$ a.s.\ and the ratio $\per(f_\inn)/\per(f_\out)$ is approximately equal in distribution to $\exp(\beta_Q Y)$ where $Y$ is a Rademacher random variable. This is the law of the ratio of inner and outer LQG lengths of CLE$_4$ loops on supercritical LQG in the setting of \cite{AG23}  (see \eqref{eq:length-ratio}).

 For concreteness, we can choose the law of $\per(f_\inn)$ under $\PP_\ring^{(p)}$ to be equal to the law of $\lfloor p\exp(\beta_Q Y)\rfloor $, and let $\PP_\ring^{(p)}$ to be supported on rings where all vertices are on the boundary of either the inner or outer face and all faces other than $f_\out$ and $f_\inn$ are quadrilaterals.  However, we can allow a more general class of distributions of rings, and this is specified later in Definition~\ref{def:ring}.

\end{itemize}

\begin{figure}
\centering
\includegraphics[width=0.3\linewidth]{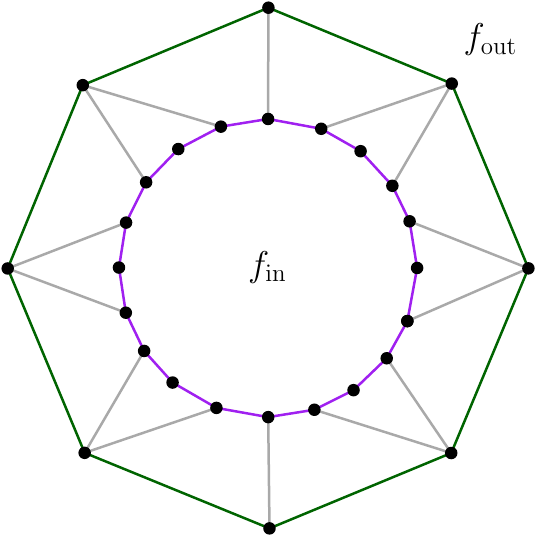}
\caption{An example of a ring sampled from $\mathbb P_\ring^{(4)}$ with the outer boundary colored in green and the inner boundary in purple. For this ring, $\per(f_\out)=4$ and $\per(f_\inn)=10$.}
\label{fig:ring}
\end{figure}

With these building blocks at hand, we now iteratively construct a discrete model of supercritical LQG. See Figure~\ref{fig:model} for an illustration of the following construction.

\begin{figure}
\centering
\includegraphics[width=0.8\linewidth]{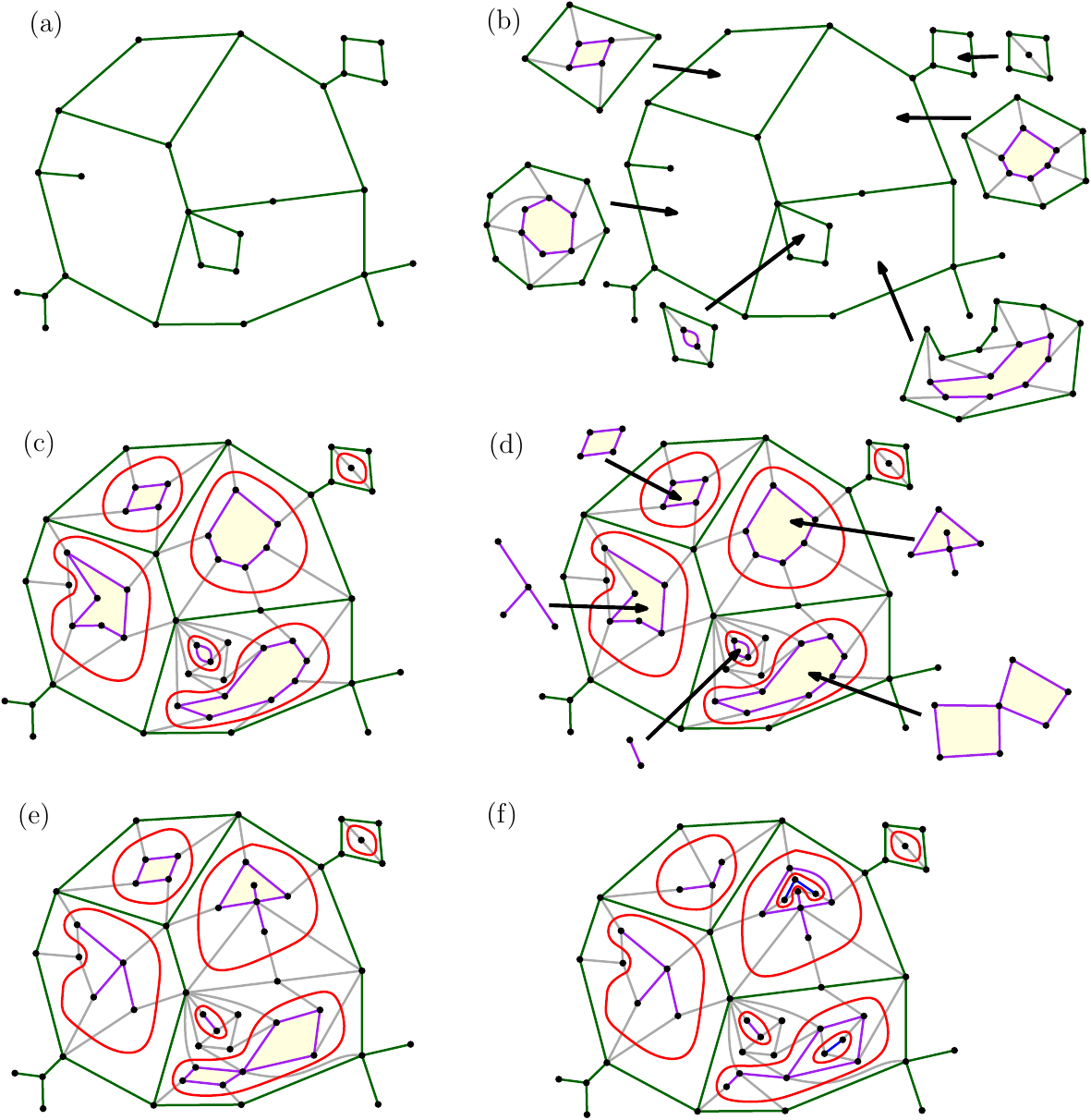}
\caption{An illustration of the iterative construction in Definition~\ref{def:model} with $p=13$. (a) The outermost gasket $M_0$ is sampled from $\mathbb P_\bfq^{(p)}$. The collection $\fF_0$ comprises the inner faces of $M_0$. (b) Given $M_0$, a ring $R(f)$ is sampled conditionally independently for each face $f\in \fF_0$. The outer/inner boundaries of the rings are colored in green/purple, respectively. The ring in the top-right of the figure has inner perimeter zero. (c) The rings are attached to the corresponding faces, with the possible rotations chosen uniformly. The rings are identified with red loops that separate inner and outer boundaries. These loops are the discrete analogs of CLE$_4$ loops in the coupling of supercritical LQG disk with CLE$_4$. (d) Given the previous figure, for each ring $R(f)$,  sample conditionally independent Boltzmann maps from $\PP_\bfq^{\per_\inn(f)}$ where $\per_\inn(f)$ is the inner half-perimeter of the ring $R(f)$. (e) The Boltzmann maps are glued to the inner boundaries of the rings, with the possible rotations chosen uniformly. These comprise the map $M_1$ colored purple and their inner faces $\fF_1$ colored yellow. (f) The map after another iteration, with $M_0$ colored in green, $M_1$ in purple, and $M_2$ in blue. In this case, $\fF_2=\varnothing$ as the Boltzmann maps comprising $M_2$ do not have non-root (inner) faces.  The construction terminates at this stage, giving us a finite map.  }
\label{fig:model}
\end{figure}

\begin{definition}
  \label{def:model}
Fix $p\in\NN$. Let $\PP_\infty^{(p)}$ denote the law of the Markov chain $\{(M_i,\fF_i)\}_{i\in\NN\cup \{0\}}$ defined below, where $M_i$ is a rooted bipartite planar map whose root face has degree $2p$ and $\fF_i\subseteq\fF(M_i)$. As can be seen below, $\fF_i$ denotes the set of interior faces of $M_i$ where Boltzmann maps are to be inserted in the subsequent step. 

\begin{itemize}
\item We start with $M_0$ sampled from the law $\PP^{(p)}_{\bfq}$ and set $\fF_0=\fF(M_0)$.
\item The induction step consists of two phases.
\begin{enumerate}[(i)]
    \item Given $(M_i,\fF_i)$, sample rings $\{R(f)\}_{f\in \fF_i}$ conditionally independently from $\PP_{\ring}^{(\per(f))}$. For each $f\in \fF_i$, glue the outer boundary of the ring $R(f)$ to the edges encircling $f$ with a uniform choice among the possible rotations of $R(f)$ relative to $f$. Define the inner perimeter $\per_{\inn}(f)$ of $f$ to be that of $R(f)$ (i.e., $\per(f_{\inn})$ with $f_\inn$ being the inner face of $R(f)$).
    \item Conditioned on the rings $\{R(f)\}_{f\in \fF_i}$ glued to the marked faces $\fF_i \subseteq \fF(M_i)$ in the previous phase, for each $f\in \fF_i$, sample an independent Boltzmann map $M(f)$ from $\PP_{\bfq}^{(\per_{\inn}(f))}$. (If $\per_\inn(f) = 0$, then we declare $M(f)$ to be a single point.) Glue its outer boundary to the inner boundary of the ring $R(f)$, again choosing uniformly among the possible rotations of $M(f)$ relative to the inner boundary of $R(f)$. 
\end{enumerate}

 \noindent The new map $M_{i+1}$ is the one resulting from gluing in all of the rings $\{R(f)\}_{f\in \fF_i}$ and the Boltzmann maps $\{M(f)\}_{f\in \fF_i}$ into each $f\in \fF_i$ as described above. The new marked faces are $\fF_{i+1}:=\bigcup_{f\in \fF_i}\fF(M(f))$.

\item There is a positive probability that the above iterative construction terminates after a finite number of steps. For example, it could be that $M_0$ is a tree, in which case $\fF_0 = \varnothing$. Also, given $\fF_i\neq \varnothing$, we may have $\fF_{i+1} = \varnothing$ if for every $f \in \fF_{i}$, either the inner boundary of the ring $R(f)$ is a single vertex (that is, $\per(f_{\inn})=0$) or $\fF(M(f)) = \varnothing$ (i.e., the Boltzmann map $M(f)$ is a tree). These events can happen with positive probabilities since for every $p\in \NN$, the Boltzmann map with perimeter $2p$ has a positive chance of having no interior faces.

If $\fF_i = \varnothing$ for some $i \in \NN\cup \{0\}$, we define $(M_j,\fF_j) := (M_{i},\varnothing)$ for all $j\geq i$ and set $M :=M_{i}$. In this case, we say that the resulting map $M$ is \textbf{finite}, as it has finitely many vertices. 

 \item Otherwise, we define $M=\bigcup_{i=0}^\infty M_0$, where we note that we have the natural inclusions $M_0\subseteq M_1\subseteq \cdots$. More formally, we define $M$ to be the projective limit of $M_i$ as $i\to\infty$. That is, $M$ is the unique (infinite) planar map with an embedding of the whole chain $(M_i, \fF_i)_{i\in\NN\cup \{0\}}$ into $M$ such that every vertex and edge of $M$ appears in some $M_i$ under this embedding. We say that the resulting map $M$ is \textbf{infinite}, as it has infinitely many vertices. 
\end{itemize}
We shall refer to the map $M_0$ in the above process as the \textbf{outermost gasket} of the discrete supercritical map $M$. Also, for any $f\in \bigcup_{i\in\NN\cup \{0\}}\fF_i$, we will use $M\lvert_f$ to denote the submap of $M$ inside the inner boundary of the ring $R(f)$. Finally, let $\PP_\ttF^{(p)}$ be the probability measure obtained by conditioning $\PP_\infty^{(p)}$ to output a finite map $M$. 
\end{definition}

\subsection{Main result 1: supercritical LQG disk conditioned to be finite is a branched polymer}\label{sec:main-discrete}

Our first main result (Theorem~\ref{thm:1}) is that our discrete model of supercritical LQG disk given in Definition~\ref{def:model}, conditioned on the event that it is finite, degenerates into a continuum tree as the perimeter of its boundary tends to infinity. 
We expect that a planar map sampled from $\PP_{\infty}^{(p)}$, when scaled appropriately, would converge to a supercritical LQG surface with central charge $\mathbf c_{\mathrm L} = 1 + 6Q^2$. We establish this scaling limit for the loops appearing in our discrete and continuum models, observing that the joint law of the lengths of the discrete loops scaled by $1/p$ converges to the joint law of the supercritical LQG lengths of the CLE$_4$ loops. We refer the reader to \cite[Section 3.2]{AG23} for a further discussion on the relation between our discrete and continuum models of supercritical LQG.

\begin{theorem}[See Proposition~\ref{prop:26} for a precise statement] \label{thm:loop-convergence}
    Consider the loop-decorated planar map sampled from the unconditioned law $\PP_\infty^{(p)}$ defined in Definition \ref{def:model}. For each $n\in \NN$, let $(\chi_i^n)_{i\in \NN}$ be the inner half-perimeters of the $n$th generation rings $\{\per_\inn(f): f\in \fF_{n-1}\}$ arranged in a decreasing order and padded by zeros. Let $(Z_i^n)_{i\in \NN}$ be the inner supercritical LQG lengths of the $n$th generation CLE$_4$ loops in the coupling of \cite{AG23}, arranged in a decreasing order. Then, as $p\to\infty$, we have
    \begin{equation}
        (p^{-1}\chi_i^n)_{i,n\in \NN} \stackrel{d}{\to} (Z_i^n)_{i,n\in\NN}
    \end{equation}
    jointly with respect to the product topology on $\RR^{\NN\times \NN}$.
\end{theorem}

When we condition on the event that the map sampled from $\PP_\infty^{(p)}$ is finite, we are looking at an event that becomes increasingly rarer as the outermost half-perimeter $p$ tends to $\infty$. Using Theorem~\ref{thm:loop-convergence} as a key input, we show that this probability decays exponentially in $p$.

\begin{proposition}
  \label{thm:2}
  For $p\in \NN$, define 
  \begin{equation}
    \label{eq:83}
    F(p)=\PP_\infty^{(p)}(M \textnormal{ is finite}).
  \end{equation}
  Then, there exists a constant $\alpha\in (0,\infty)$ depending only on $Q \in (0,2)$ such that as $p\rightarrow \infty$,
  \begin{equation}
    \label{eq:84}
    \lim_{p\to\infty} \frac{-\log F(p)}{p}= \alpha.
  \end{equation}
\end{proposition}

Note that Proposition~\ref{thm:2} implies that a map sampled from $\PP_\infty^{(p)}$ is actually infinite with positive probability (i.e., $F(p)<1$) for large $p$, which is not a priori obvious. In fact, we show that $F(p)<1$ for all $p\in \NN$ (Lemma~\ref{lem:1}). The value of the constant $\alpha$ is unknown as it is shown to exist via subadditivity, nor do we know whether this $\alpha$ depends on the particular random planar map model or just on $\mathbf c_{\mathrm L}$.

We now give the first main result of this paper after stating a few preliminaries. We use the abbreviation CRT to refer to the law of Aldous's continuum random tree \cite{Ald91}, which is a canonical random tree constructed out of a Brownian excursion. We describe the scaling limit using the Gromov--Hausdorff distance, which is a canonical metric on the space of compact metric spaces defined up to isometry. (These concepts are reviewed in Section~\ref{sec:crt}.) Also, given a planar map $M$ with vertex set $V_M$ and graph distance $d_M$, for $r\in \RR$, let $rM$ denote the metric space $(V_M, rd_M)$. 
\begin{theorem}
  \label{thm:1}
 For each $p\in \NN$, let $M^{(p)}$ be the planar map sampled from $\PP^{(p)}_{\ttF}$ --- i.e., by conditioning a map drawn from $\PP^{(p)}_\infty$ to be finite. Then, there is a constant $\theta_\ttF$ depending on $\bfq$ such that 
  \begin{equation}
    \frac{\theta_\ttF}{ \sqrt{p}} M^{(p)} \stackrel{d}{\rightarrow} \CRT
  \end{equation}
  with respect to the Gromov--Hausdorff distance.
\end{theorem}

As discussed earlier in Section \ref{sec:physics}, this result provides a rigorous reconciliation of the ``infinite spikes'' picture (\ref{item:infinite-spikes}) and the ``branched polymer" description (\ref{item:branched-polymer}) of supercritical LQG. 
There are various settings in which one could attempt to prove versions of the statement that "supercritical LQG surfaces conditioned to be finite converge to the CRT". For instance, one could look at the approximation scheme used to construct the supercritical LQG metric in \cite{DG23-tightness,DG23} and condition on the event that the approximating metrics have bounded diameter. Or, one could consider the dyadic tiling model studied in \cite{GHPR20,APPS22} and condition on the (rare event) that there are only finitely many dyadic squares in some bounded region. These approaches have been tried unsuccessfully in the past. A key difficulty is that it is hard to describe what effect the conditioning has on the underlying GFF. The advantage of the random planar map approximations considered in Theorem \ref{thm:1} is that the operation of conditioning to be finite is in some sense compatible with the solvability of the random planar map model.

The heuristic reasoning behind Theorem \ref{thm:1} is that the iterative construction of the map in Definition~\ref{def:model} can be thought of, in some sense, as a supercritical branching process. As is the case for the Bienaym\'e--Galton--Watson (BGW) processes, it turns out that our planar map becomes ``subcritical" when conditioned to be finite, and we use this subcriticality to show that the planar map sampled from $\PP_\ttF^{(p)}$ degenerates to a branched polymer in the scaling limit. Implementing this heuristic requires a combination of several ideas, including results from \cite{Bet15,JS15,KR20,Mar22} on the convergence of rescaled subcritical Boltzmann maps to the CRT, combinatorial encodings of Boltzmann maps in terms of heavy-tailed random walks \cite{BDG04,JS15,CCM20}, and tricky estimates to show that the submaps surrounded by the outermost loops on $M^{(p)}$ are all microscopic (Section~\ref{sec:tail-bounds}). We refer the reader to Section \ref{sec:outline} for a more detailed overview of the proof. 

Finally, we note that it was shown in \cite{Feng23} that random planar maps decorated by the Fortuin--Kasteleyn (FK) cluster model with parameter $q > 4$ converge to the CRT in the Gromov--Hausdorff sense (using completely different techniques from the ones in this paper). FK-decorated random planar maps for $q \in[0,4]$ are expected to converge to $\gamma$-LQG where $\gamma \in [\sqrt 2,2]$ satisfies $q = 2 + 2 \cos(\pi \gamma^2/2)$ \cite{She11}, but we currently have no reason to believe that such maps for $q > 4$ are related to supercritical LQG.

\subsection{Main result 2: non-existence of supercritical LQG volume measure}

In the subcritical and critical phases, an LQG surface with the disk topology is associated with two natural measures: a length measure on the boundary and an area measure in the bulk. In saying that these measures are ``natural," we mean that they are intrinsic to the LQG surface in the following sense: (i) they are locally determined by the background GFF $h$ and (ii) they transform appropriately under a conformal change of coordinates (see \eqref{eq:109} and \eqref{eq:170} for the precise condition). In the recent work \cite{AG23}, it was shown that for a supercritical LQG disk, there exists a natural notion of a boundary length measure. 

Given this result, one may wonder if it is also possible to construct a natural volume measure in the bulk of a supercritical LQG surface. In this work, we establish that there does not exist any such locally finite, nontrivial measure. 
It is not at all obvious how to prove this directly from the above two axioms. One reason for this is that there are a huge number of possible ways to construct random measures (see the discussion just after Theorem \ref{thm:3}). The key insight in our proof is that the exact self-similarity provided by the coupling of supercritical LQG with CLE$_4$ from \cite{AG23} prevents a measure satisfying the conditions of Theorem \ref{thm:3} from being locally finite.

To state the result, we require the following notion of local absolute continuity with respect to the GFF. Let $\tilde h$ be a Dirichlet GFF on an open set $U\subseteq \CC$. A random generalized function $h$ on $U$ is said to be \textbf{locally absolutely continuous} to a GFF if for every $z\in U$, there exists a open neighborhood $V\subsetneq U$ containing $z$ with a compact closure in $U$ such that the law of $h\lvert_V$ is absolutely continuous to that of $\tilde h\lvert_V$. We are now ready to state the second main result of this article.

\begin{theorem}
  \label{thm:3}
Fix the background charge $Q\in (0,2)$. Suppose that for each domain $U \subset \mathbb C$, we have a measurable mapping $h \mapsto \mathfrak m_h$ from generalized functions on $U$ to Borel measures on $U$. Assume further that the following conditions hold whenever $h$ is a random generalized function on a domain $U\subset \CC$ which is locally absolutely continuous to a GFF.
  \begin{itemize}
  \item \textbf{Locality:} Almost surely, for any fixed $V\subseteq U$, we have $(\mathfrak m_h)|_V = \mathfrak m_{(h|_V)}$.
  \item \textbf{Coordinate change:} For any fixed conformal map $f : \widetilde U\to U$, we almost surely have
    \begin{equation}
      \label{eq:109}
      \mathfrak m_{h \circ f + Q \log |f'|}(A) = \mathfrak m_h(f(A))
    \end{equation}
    for every Borel subset $A\subseteq \widetilde U$.
  \end{itemize}
  Then, any such mapping $\mathfrak m$ must be trivial in the following sense: for any random generalized function $h$ on a domain $U\subseteq \CC$ which is locally absolutely continuous to a GFF, the measure $\mathfrak{m}_h$ is either a.s.\ equal to the zero measure or a.s.\ assigning infinite mass to every Euclidean open subset of $U$.
\end{theorem}

In the above statement, we do not care how the mapping $h \mapsto \mathfrak m_h$ is defined when $h$ is not absolutely continuous with respect to a GFF. For instance, we could require it to be identically zero for concreteness. (A similar setup appears in the definition of LQG metric in \cite{DFGPS20,GM21,Pfe21,DG23}.)

While the above theorem is stated with $\mathfrak{m}_h$ being a positive measure, we note that by the usual decomposition of signed measures into positive and negative parts, it follows immediately that there does not exist any natural signed (or even complex) bulk measure in the supercritical phase of LQG. This rules out some possible constructions for the volume measure of supercritical LQG, including some versions of Gaussian multiplicative chaos (GMC): e.g., supercritical GMC in either the dual phase \cite{BJRV13} or the glassy phase \cite{MRV16}, as well as some version of complex GMC \cite{LRV15}. It also rules out a construction based on finding a natural operator other than the exponential of the field $h$ such as in the works of Gervais et al.\ reviewed in Section~\ref{sec:physics}. 

Finally, we note that all of the three conditions for $\mathfrak{m}_h$ --- a random measure, locally determined by $h$, and satisfying the coordinate change rule --- are necessary in Theorem~\ref{thm:3}. Here are examples of functionals of the field $h$ which satisfy a strict subset of these conditions.

\begin{enumerate}[(a)]
    \item A clear candidate for a random measure locally determined by a GFF $h$ is the $\beta$-LQG measure $\mu_h$ where $\beta \in (0,2]$. However, this does not satisfy the coordinate-change rule \eqref{eq:109} for any $Q\in (0,2)$.
    \item Since a GFF $h$ in $U$ is an element of the Sobolev space $H^{-\epsilon}(U)$ for $\epsilon>0$ (see, e.g., \cite{WP21}), $\Delta h$ can be defined as a generalized function in $H^{-(2+\epsilon)}(U)$ for $\epsilon>0$. While $\Delta h$ is not a random measure, it is a local functional of $h$ and satisfies 
    \begin{equation}
        \Delta( h \circ f + Q\log|f'|) = (\Delta h) \circ f
    \end{equation}
trivially for any conformal $f$ and $Q\in\RR$ since $\Delta \log |f'| = 0$.
    
    \item Let $h$ be a Gaussian field on $U\subsetneq \CC$. Fix $\beta\in (0,2]$ and let $\alpha = Q-(\frac{2}{\beta}+\frac{\beta}{2})$. Recall that the conformal radius of a domain $U$ viewed from $z\in U$ is defined as $\mathrm{CR}(z;U) := |g'(0)|$ where $g:\mathbb{D}\to U$ is a conformal map sending 0 to $z$. Consider
    \begin{equation}\label{eq:measure-nonlocal}
        \mathfrak{m}_h(dz) := \mathrm{CR}(z;U)^\alpha\mu_h^{(\beta)}(dz)
    \end{equation}
    where $\mu_h^{(\beta)}$ is the $\beta$-LQG measure on $U$. Then, for a conformal map $f:U\to \widetilde U$ and a GFF $\tilde h$ on $\widetilde U$, the pullback of $\mathfrak{m}_{\tilde h}$ under $f$ is given by
    \begin{equation}
    \begin{split}
        f^*\mathfrak{m}_{\tilde h}(dz) &= \mathrm{CR}(f(z);\widetilde U)^\alpha \mu_{\tilde h}^{(\beta)}(df(z))\\
        &= |f'(z)|^\alpha \mathrm{CR}(z;U)^\alpha \mu_{\tilde h \circ f + (\frac{2}{\beta}+\frac{\beta}{2})\log|f'|}^{(\beta)}(dz) = \mathfrak{m}_{\tilde h \circ f + Q\log|f'|}(dz).
    \end{split}
    \end{equation}
    Hence, this $\mathfrak{m}_h$ is a random Borel measure satisfying the coordinate change rule $\eqref{eq:109}$ with $Q\in(0,2)$. Nevertheless, $\mathfrak{m}_h$ does not satisfy the locality condition of Theorem~\ref{thm:3} since the measure $\mathfrak{m}_{(h\lvert_V)}$ on a domain $V\subseteq U$ also depends on the background domain $U$ via the conformal radius factor appearing in \eqref{eq:measure-nonlocal}.

    \item In Section~\ref{sec:cascade-measure}, we give a natural family of random measures on supercritical LQG disks constructed via the multiplicative cascade associated to the coupling of supercritical LQG with nested CLE$_4$. These measures satisfy the coordinate change rule \eqref{eq:109}, but they are not local.
\end{enumerate}

\begin{remark}
  For $\mathbf{c}_{\mathrm L} \in (1,25)$, the LQG metric does not induce the Euclidean topology. In particular, the $\alpha$-thick points for $\alpha >Q$ lie at infinite distance from every other point (see \cite[Proposition 1.11]{Pfe21}), and consequently LQG metric balls have empty Euclidean interior. Could there be a measure satisfying the hypotheses of Theorem \ref{thm:3} which assigns finite mass to LQG metric balls, but infinite mass to Euclidean open sets? 
   
  We expect that no such measure exists. It is shown in \cite[Proposition~1.14]{Pfe21} that almost surely for every $s < r$, a supercritical LQG metric ball of radius $r$ cannot be covered by finitely many supercritical metric balls of radius $s$.
    This means that if we have a measure $\mathfrak m_h$ associated with a supercritical LQG surface with field $h$, then either $\mathfrak m_h$ assigns very small mass to most supercritical LQG metric balls, or $\mathfrak m_h$ assigns infinite mass to every supercritical LQG metric ball.  In particular, there cannot be any measure $\mathfrak m_h$ which is ``compatible with the supercritical LQG metric $D_h$" in the same sense that the LQG measure is compatible with the metric in the subcritical case. 
\end{remark}

\subsection{Proof outline}
\label{sec:outline}
This article is divided into two main parts: continuous (Sections \ref{sec:2}--\ref{sec:3}) and discrete (Sections \ref{sec:boltzmann}--\ref{sec:conv}). Throughout both parts, the central idea is to consider the multiplicative cascade of lengths of loops on a supercritical LQG disk or its discrete analog. The linchpin between the two parts is Theorem~\ref{thm:loop-convergence}, in which we establish the convergence of the discrete cascade of loop lengths in Definition~\ref{def:model} to the continuum cascade of loop lengths in the CLE$_4$-coupled supercritical LQG disk of \cite{AG23}. This is the supercritical analog of \cite{CCM20}, which considered the convergence of the discrete cascade originating from the gasket decomposition \cite{BBG11} of loop $O(n)$ model decorated planar maps with $n\in (0,2)$ to the continuum cascade of subcritical $\gamma$-LQG disks coupled with independent CLE$_\kappa$ with $\kappa = \gamma^2 \in (8/3,4)$ in \cite{MSW22}.

The first part of this work concentrates on the proof of Theorem~\ref{thm:3}. In Section \ref{sec:2}, we recall from \cite{AG23} the definition of a supercritical LQG disk and its coupling with (nested) CLE$_4$. We identify the exact law of the multiplicative cascade on the boundary lengths of CLE$_4$ loops in this coupling (Corollary~\ref{cor:supercritical-cascade}). In Section \ref{sec:macro-loops}, we draw classical facts from the branching random walk literature to analyze the law of the multiplicative cascade and show, in particular, that the number of CLE$_4$ loops of any given size tends to infinity as we look at further generations of loops (Proposition~\ref{prop:many-macroscopic-loops}). Given these results on the lengths of CLE$_4$ loops, the proof of Theorem~\ref{thm:3} in Section~\ref{sec:proof-nonexistence} follows straightforwardly from the fact the supercritical LQG disks enclosed by CLE$_4$ loops of the same generation are conditionally independent given their boundary lengths (Proposition~\ref{prop:nested-coupling}). Section~\ref{sec:cascade-measure} gives the construction of a natural family of (non-local) measures on supercritical LQG disks using its coupling with nested CLE$_4$.

In the second part, we analyze the behavior of our random planar map model of supercritical LQG and prove Theorem~\ref{thm:1}. 
Section~\ref{sec:boltzmann} recalls from the random planar maps literature the two key tools for our analysis: the connection between Boltzmann maps and random walks (Proposition~\ref{prop:10}) and the convergence of the  boundaries of large subcritical Boltzmann maps to the CRT (Proposition~\ref{prop:4}). 

In Section~\ref{sec:prob-dis}, we consider the probability $F(p)$ that a map sampled from our discrete model $\PP_\infty^{(p)}$ of supercritical LQG with boundary length $2p$ is finite. 
Recalling the convergence (Theorem \ref{thm:loop-convergence}) of the discrete perimeter cascade of our model to the continuum perimeter cascade of the CLE$_4$-coupled supercritical LQG disk, we use the existence of large loops in the latter (Proposition~\ref{prop:many-macroscopic-loops} again) to conclude that $F(p)\to 0$ as $p\to\infty$ (Lemma~\ref{lem:7}). 
We then prove Proposition~\ref{thm:2} based on a slightly modified form of the standard subadditivity argument which allows for logarithmic errors, which we establish using the random walk description of Boltzmann maps. 

In Section~\ref{sec:estimates}, we analyze the behavior of our discrete model $\PP_\ttF^{(p)}$ of a supercritical LQG disk conditioned to be finite. This is the most difficult part of the paper. We first identify that the marginal law on the outermost gasket $M_0^{(p)}$ under this conditioning is that of a subcritical Boltzmann map (Proposition~\ref{prop:2}). From here, we extrapolate the subcriticality of the entire map $M^{(p)}$ sampled from $\PP_\ttF^{(p)}$ in the sense that the expected size of maps added in subsequent generations decays exponentially (Corollary~\ref{lem:21}). In particular, the size of the submaps of $M^{(p)}$ inside each face of the outermost gasket $M_0^{(p)}$ is of smaller order than the diameter of the boundary $\partial M^{(p)}$ (Proposition~\ref{lem:23}). Establishing this requires intricate estimates for the branching process which describes the lengths of the loops in each generation. Throughout this analysis, we find another use for Proposition~\ref{thm:2} in that it allows us to use $p$ and $-\alpha^{-1}\log F(p)$ interchangeably for large values of $p$ in many of our estimates. This is useful due to the relation
\begin{equation}
    F(p) = \EE\Bigg[\prod_{f\in \fF_0} F(\per_\inn(f))\Bigg],
\end{equation}
where $\fF_0$ is the collection of interior faces of the outermost gasket $M_0^{(p)}$.

We finally prove Theorem~\ref{thm:1} in Section~\ref{sec:conv}. Applying the results of Section~\ref{sec:estimates} and the convergence of the entire subcritical Boltzmann map to the CRT (Proposition~\ref{prop:4}), we see that a map $M^{(p)}$ sampled from $\PP_\ttF^{(p)}$ is ``thin" in the sense that the maximum graph distance from an interior vertex of $M^{(p)}$ to the boundary $\partial M^{(p)}$ is of smaller order than the diameter of $\partial M^{(p)}$. Thus, the problem at hand reduces to showing the convergence to the CRT when we equip $\partial M^{(p)}$ with the restriction of the graph distance $d_{M^{(p)}}$ on the entire map (Proposition~\ref{prop:8}). We adapt the proof of the convergence of $\partial M^{(p)}$ with the graph distance $d_{\partial M^{(p)}}$ in \cite{KR20} by showing that the two distances $d_{M^{(p)}}$ and $d_{\partial M^{(p)}}$ are equivalent on $\partial M^{(p)}$ (Proposition~\ref{prop:129}). This follows from a law of large numbers argument based on the spine decomposition of the critical looptree structure of $\partial M^{(p)}$ identified in \cite{CK15,Ric18}.

\paragraph{\textbf{Notational comments}} Throughout the paper, we shall have multiple occasions to use the Bienaym\'e-Galton-Watson (or Galton--Watson) tree, which which refer to as a $\BGW$ tree. For any planar map $M$, we denote its vertex set as $V_M$ and use $d_{M}(x,y)$ to denote the graph distance between vertices $x$ and $y$ in $M$. Also, we will use $\Vol(M)$ to denote the total number of vertices present in a planar map $M$. Let $\NN$ denote the set $\{1,2\dots\}$ of positive integers and $\NN^{\#}=\NN\cup \{0\}$ denote the nonnegative integers. For $a<b\in \RR$, the double interval $[\![a,b]\!]$ refers to $[a,b]\cap \ZZ$.

\paragraph{\textbf{Acknowledgements}} We thank Morris Ang, Bruno Balthazar, Nicolas Curien, William Da Silva, Antti Kupiainen, David Kutasov, Eveliina Peltola, Josuha Pfeffer, Scott Sheffield, Xin Sun, and Paul Wiegmann for helpful discussions. M.B.\ acknowledges the partial support of the NSF grant DMS-2153742. E.G.\ was partially supported by a Clay research fellowship and the NSF grant DMS-2245832. J.S.\ was partially supported by a Kwanjeong Educational Foundation scholarship. This work was completed in part during visits by M.B.\ and J.S.\ to the Thematic Program on Randomness and Geometry at the Fields Institute, whose hospitality is gratefully acknowledged.

\section{Background on the supercritical LQG disk}
\label{sec:2}

\subsection{Supercritical LQG disk}
\label{sec:lqg}
We begin with a general definition of an LQG surface, which is a two-dimensional domain equipped with a generalized function considered up to a conformal change of parameterization. The following is a generalization of the definitions in \cite{DS11,She16,DMS14}, etc., to the supercritical case $Q\in (0,2)$.

\begin{definition}\label{def:lqg-surface}
    Let $\cL > 1$ and let $Q = \sqrt{(\cL - 1)/6} > 0$ be the corresponding background charge. Consider the set of tuples $(U, h, A)$ where $U \subset \CC$ is open, $h$ is a generalized function on $U$, and $\cA$ a sequence of compact subsets of $U$ (this is some sort of ``decoration" which could be a collection of curves, points, etc.). Let $\sim_Q$ be an equivalence relation on such tuples where we define 
    \begin{equation}
        (U, h, \cA) \sim_Q (\widetilde U, \tilde h, \widetilde \cA)
    \end{equation}
    if there is a conformal map $f:\widetilde U \to U$ such that 
    \begin{equation}\label{eq:170}
        \tilde h = h \circ f + Q\log |f'| \quad \text{and} \quad f(\widetilde \cA) = \cA.
    \end{equation}
    A \textbf{Liouville quantum gravity (LQG) surface} of central charge $\cL$, which we often call a \textbf{$Q$-LQG surface}, is an equivalence class under $\sim_Q$. A representative of an LQG surface is called an \textbf{embedding}.
\end{definition}

In the theory of LQG, the generalized function $h$ is always a Gaussian free field (GFF) or some variant of it, and we now give a quick introduction to the GFF. A GFF on a domain $U\subset \CC$ is a centered Gaussian generalized function $h$ with the covariance kernel $G$ given by a Green's function\footnote{We assume that the Green's function $G$ is normalized so that $G(z,w) \sim -\log|z-w|$ as $w\to z$.} corresponding to the Laplacian on $U$. That is, if we denote the paring of $h$ with a smooth function $f\in C^\infty(U)$ as $(h,f)$, then this is a Gaussian random variable with
\begin{equation}
  \label{eq:110}
  \EE[(h,f_1)(h,f_2)]=\int_{U\times U} f_1(z_1) G(z_1,z_2) f_2(z_2)\,dz_1dz_2.
\end{equation}
for any smooth functions $f_1,f_2\in C^\infty(U)$.
We note that the GFF is said to have \textbf{Dirichlet/zero boundary conditions} if the Green's function $G$ is defined with zero boundary conditions. The same is true with \textbf{Neumann/free boundary conditions}.

We note that the definition of the equivalence relation $\sim_Q$ in \eqref{eq:170} is given to reflect the conformal covariance of the LQG measure \cite{DS11,SW16} and the LQG metric \cite{GM21b} in the subcritical regime ($\cL>25$, $Q>2$). Thus, if $h$ is a GFF on $U$ and $\mu_h$ and $D_h$ are the LQG measure and the metric corresponding to the field $h$ with central charge $\cL$, then for a conformal map $f:\widetilde U\to U$ with $\tilde h$ given by \eqref{eq:170}, we have $f_* \mu_{\tilde h} = \mu_h$ and $f_* D_{\tilde h} = D_h$ almost surely, where, as usual, $f_*$ is used to denote the push-forward with respect to $f$. This result extends to the critical case ($\cL=25$, $Q=2$) for the LQG measure \cite[Theorem~13]{DRSV14b} and is expected to hold for the LQG metric in the critical/supercritical regime ($\cL\in(1,25]$, $Q\in(0,2]$) but is so far known only for complex affine $f$ \cite[Propositon~1.9]{DG23} (see also \cite[Theorem~1.4]{Dev23}). This is the context in which the analogous conformal covariance condition \eqref{eq:109} is required in our definition of a supercritical LQG measure in Theorem~\ref{thm:3}.

We now present the definition of the supercritical LQG disk in \cite{AG23}, which is based on the critical LQG disk defined in \cite{AHPS23}. Let $\nu_h$ be the \textit{critical} LQG boundary measure corresponding to the field $h$; this was constructed in \cite{DRSV14b} as 
\begin{equation}
    \nu_h:= \lim_{\epsilon \to 0} \sqrt{\log(1/\epsilon)}\epsilon e^{h_\epsilon(z)}|dz|,
\end{equation}
where $h_\epsilon$ is a mollification of $h$ of size $\epsilon$ and $|dz|$ denotes the Lebesgue measure on the boundary of the domain. Also see \cite[Section~4.1.2]{APS19} for its construction as a limit of subcritical LQG boundary measures. As observed in \cite[Lemma~2.1]{AG23-critical}, the conformal covariance relation $f_* \nu_{\tilde h} = \nu_h$ given \eqref{eq:170} follows from this limiting construction thanks to the analogous result for subcritical boundary measures \cite[Section~6]{DS11}.

\begin{definition}[Critical LQG disk]\label{def:2-disk}
    Let $\mathcal S = \RR \times (0,2\pi)$.
    Let $\mathcal B:\RR \to (-\infty,0]$ be a random function such that $(\mathcal B_s/\sqrt 2)_{s\geq 0}$ and $(\mathcal B_{-s}/\sqrt 2)_{s\geq 0}$ are independent 3-dimensional Bessel processes started at 0, and let $h^|: \mathcal S \to \RR$ be the random function which is identically equal to $\mathcal B_s$ on each vertical segment $\{s\} \times (0,2\pi)$. Let $h^\dagger$ be the lateral part of a free-boundary GFF on $\mathcal S$  (i.e., the orthogonal projection to the subspace of generalized functions on $\mathcal S$ which has zero mean on every vertical segment) sampled independently from $h^|$. Define $h := h^| + h^\dagger$ and $\hat h:= h - \log \nu_h(\partial \mathcal S)$. For $L>0$, the \textbf{critical LQG disk with boundary length $L$} is the 2-LQG surface $(\mathcal S, \hat h + \log L)/\!\sim_2$.
\end{definition}

\begin{definition}[Supercritical LQG disk]\label{def:Q-disk}
    Fix $Q \in (0,2)$ and $L>0$. Let $(U,h^C)$ be an embedding of the critical LQG disk with boundary length $L$ given in Definition~\ref{def:2-disk}. Let $h^D$ be an independent zero-boundary GFF on $U$. Then, the \textbf{supercritical LQG disk with central charge $\cL = 1 + 6Q^2$ and boundary length $L$} (also called the \textbf{$Q$-LQG disk with boundary length $L$}) is the $Q$-LQG surface $(U, h)/\!\sim_Q$ with 
    \begin{equation}\label{eq:supercritical-field}
        h := \frac{Q}{2} h^C + \frac{\sqrt{4-Q^2}}{2} h^D.
    \end{equation}
\end{definition}

We note that for any $\tilde L >0$, if $(U,h)$ is an embedding of a $Q$-LQG disk with boundary length $L$, then $(U,h + \frac{2}{Q} \log (\tilde L /L))/\!\sim_Q$ is a $Q$-LQG disk with boundary length $\tilde L$. The \textbf{supercritical LQG length measure} on the boundary of a $Q$-LQG disk is defined as 
\begin{equation}\label{eq:supercritical-length}
    \mathfrak n_h : = \nu_{h^C} \quad ``= e^{h^C}|dz| = e^{\frac{2}{Q} h}|dz|"
\end{equation}
where $\nu_{h^C}$ is the critical LQG boundary length measure associated with the field $h^C$ given in Definition~\ref{def:Q-disk}. The second half of \eqref{eq:supercritical-length} is an informal expression based on the following heuristic reasoning: since $h^D$ is a zero-boundary GFF, we have $h^D\lvert_{\partial U}=0$ and hence $h^C\lvert_{\partial U}=\frac{2}{Q} h\lvert_{\partial U}$. Thus, if $(U,h)$ is an embedding of a $Q$-LQG disk, $f:\widetilde U \to U$ is a conformal map, and $\tilde h = h \circ f + Q\log|f'|$, then we have $f_* \mathfrak n_{\tilde h} = \mathfrak n_h$ almost surely \cite[Lemma~2.8]{AG23}.

\subsection{Coupling with \texorpdfstring{CLE$_4$}{CLE(4)}}

For $\kappa \in (8/3,4]$, the \textbf{conformal loop ensemble} (CLE$_\kappa$) on a simple connected domain $U \subsetneq \CC$ is a random countable collection of disjoint non-nested Jordan curves in $U$ which look locally like SLE$_\kappa$. It can be defined via branching SLE curves \cite{She09} or as the outer boundaries of the outermost clusters of a Brownian loop soup on $U$ \cite{SW12}. We are interested in the critical case $\kappa=4$, which corresponds to the Brownian loop soup of intensity 1/2.

The coupling of CLE$_4$ with a supercritical LQG disk in \cite{AG23} is given by combining two different couplings. First, we record the following integrable description of CLE$_4$ on an independent critical LQG disk proved in \cite{AG23-critical} (see \cite{BBCK18,CCM20,MSW22} for the subcritical analog).

\begin{proposition}\label{prop:critical-gasket-coupling}
    Let $(U,h^C)$ be an embedding of the critical LQG disk with boundary length $L>0$. Let $\Gamma$ be a (non-nested) CLE$_4$ in $U$ sampled independently from $h^C$. For each loop $\ell \in \Gamma$, let $U_\ell\subset U$ be the open region enclosed by $\ell$ and $Z(\ell):=\nu_{h^C}(\ell)$ denote the critical LQG length of the loop $\ell$.
    \begin{enumerate}[(a)]
        \item The 2-LQG surfaces $\{(U_\ell, h^C|_{U_\ell})/\!\sim_2\}_{\ell \in \Gamma}$ are conditionally independent critical LQG disks given their boundary lengths $\{Z(\ell)\}_{\ell\in\Gamma}$.

        \item Let $(\zeta(t))_{t\geq 0}$ be a 3/2-stable L\'evy process with no downward jumps\footnote{\label{footnote:scale-par}We note that this is defined up to the choice of the scale parameter $C>0$, where the L\'evy measure of the process is equal to $Cx^{-5/2}\,\ind_{\{x>0\}}\,dx$. The precise choice of $C$ is irrelevant for us, since, as discussed in the introduction of \cite{CCM20}, any choice of $C$ yields the same distribution $\rho^{(L)}$.} and $\tau_L:= \inf\{t\geq 0: \zeta(t)=-L\}$. Denote by $(\Delta\zeta)_L^\downarrow$ the sizes of upward jumps of $\zeta$ in $[0,\tau_L]$ enumerated in a decreasing order. Define the probability distribution $\rho^{(L)}$ on $(\RR_+)^{\NN}$ so that for each measurable function $F:(\RR_+)^{\NN}\to \RR$, we have
  \begin{equation}
    \label{eq:1}
    \int F\,d\rho^{(L)}=\frac{\EE[(\tau_L)^{-1}F((\Delta \zeta)_L ^\downarrow)]}{\EE [(\tau_L)^{-1}]}.
  \end{equation}
        The sequence of boundary lengths $\{Z(\ell)\}_{\ell\in\Gamma}$ arranged in a decreasing order has the law $\rho^{(L)}$. Note that by the scaling relation for stable processes, 
\begin{equation}
    (\Delta\zeta)_L^\downarrow \stackrel{d}{=} L \cdot (\Delta_\zeta)_1^\downarrow.
\end{equation}
    \end{enumerate}  

\end{proposition}

The second coupling that we recall is the level-line coupling of CLE$_4$ with zero-boundary GFF due to Miller and Sheffield (see \cite{ASW19} for a proof).

\begin{proposition}\label{prop:level-line-coupling}
    Given a domain $U \subsetneq \CC$, there is a coupling of the zero-boundary GFF $h^D$ and the (non-nested) CLE$_4$ in $U$ such that $\Gamma$ is almost surely determined by $h^D$ and the following holds:
   There exists a sequence of random variables $\{Y_\ell\}_{\ell\in \Gamma}$ such that, conditioned on $\Gamma$, $\{Y_\ell\}_{\ell\in \Gamma}$ are i.i.d.\ Rademacher random variables and $\{h^D|_{U_\ell} - \pi Y_\ell\}_{\ell \in \Gamma}$ are independent zero-boundary GFFs on the respective domains $\{U_\ell\}_{\ell \in \Gamma}$.
\end{proposition}

The coupling of a supercritical LQG disk with CLE$_4$ in \cite{AG23} is obtained by sampling independent fields $h^C$ and $h^D$ on a domain $U\subsetneq \CC$ such that $(U,h^C)/\!\sim_Q$ is a critical LQG disk with boundary length $L>0$ and $h^D$ is a zero-boundary GFF, and then defining the field $h$ as in \eqref{eq:supercritical-field} and the CLE$_4$ $\Gamma$ as a function of $h^D$ as in Proposition~\ref{prop:level-line-coupling}. 

It is straightforward to adapt Proposition~\ref{prop:critical-gasket-coupling} to this supercritical coupling except that, since $h^D$ has a ``height gap'' of size $\pi Y_\ell$ across the loop $\ell \in \Gamma$ as in Proposition~\ref{prop:level-line-coupling}, the supercritical field $h$ has a corresponding gap of size $\beta_Q Y_\ell$ across the loop $\ell$ where $\beta_Q := \pi \sqrt{4-Q^2}/Q$ as defined in \eqref{eq:beta-Q}. This means that the length of loop $\ell$ measured using the $Q$-LQG length measure \eqref{eq:supercritical-length} is different based on whether we use the field to the inside or outside of $\ell$. More precisely, define the \textbf{outer boundary length} of $\ell$ as 
\begin{equation}
    \mathfrak{n}_h^\out(\ell) := \nu_{h^C}(\ell)
\end{equation}
and the \textbf{inner boundary length} of $\ell$ as 
\begin{equation}\label{eq:inner-length}
    \mathfrak n_h^\inn(\ell) = \mathfrak n_{h|_{U_\ell}}(\ell) := \nu_{h^C + \beta_Q Y_\ell}(\ell).
\end{equation}
In particular, the two lengths are related by ratio 
\begin{equation} \label{eq:length-ratio} \frac{\mathfrak n_h^\inn(\ell)}{\mathfrak n_h^\out(\ell)} = e^{\beta_Q Y_\ell}. \end{equation}
The notation $\mathfrak n_{h|_{U_\ell}}(\ell)$ makes sense since if we let $f_\ell:\mathbb D \to U_\ell$ be a conformal map, then combining the coordinate change rule \eqref{eq:170} of a $Q$-LQG surface and the definition \eqref{eq:supercritical-length} of the $Q$-LQG boundary length of a domain, we have $\mathfrak n_h^\inn(\ell) = \mathfrak n_{(h|_{U_\ell}) \circ f_\ell + Q\log |f_\ell'|}(\partial \mathbb D)$.

\begin{proposition}[{\cite[Theorem~2.9]{AG23}}]\label{prop:gasket-coupling}
    Let $Q\in (0,2]$ and $(U,h)$ be an embedding of a $Q$-LQG disk with boundary length $L>0$. There exists a coupling $(h,\Gamma)$ of $h$ with a (non-nested) CLE$_4$ $\Gamma$ on $U$ such that the following is true. For each loop $\ell \in \Gamma$, let $U_\ell \subset U$ be the open region enclosed by $\ell$ and $Z_Q(\ell):= \mathfrak n_h^\inn(\ell)$ be the supercritical LQG length of the loop $\ell$ measured using the field $h|_{U_\ell}$ within the loop.
    \begin{enumerate}[(a)]
        \item The $Q$-LQG surfaces $\{(U_\ell, h|_{U_\ell})/\!\sim_Q\}_{\ell \in \Gamma}$ are conditionally independent $Q$-LQG disks given their boundary lengths $\{Z_Q(\ell)\}_{\ell\in\Gamma}$.

        \item Let $\{Z(\ell)\}_{\ell\in\Gamma}$ be the sequence with the law $\rho^{(L)}$ given in \eqref{eq:1}, and let $\{Y_\ell\}_{\ell \in \Gamma}$ be independent Rademacher random variables which are also independent from $\{Z(\ell)\}_{\ell\in\Gamma}$. Recall the constant $\beta_Q:= \pi \sqrt{4-Q^2}/Q$.  Then, $\{Z_Q(\ell)\}_{\ell \in \Gamma}$ has the same law as $\{Z(\ell) \exp(\beta_Q Y_\ell)\}_{\ell\in\Gamma}$ if both are arranged in decreasing orders. Let us denote this law as $\rho_Q^{(L)}$.
    \end{enumerate}
    In this coupling, $\Gamma$ is neither independent from nor a.s.\ determined by $h$.
\end{proposition}

\subsection{Coupling with nested \texorpdfstring{CLE$_4$}{CLE(4)} and the perimeter cascade}

For a simply connected domain $U \subsetneq \CC$, the \textbf{nested CLE$_4$} $\overline \Gamma$ on $U$ is obtained from the non-nested CLE$_4$ by the following inductive procedure. Let $\Gamma^0:=\{\partial U\}$. Given collections of loops $\{\Gamma^i\}_{i\in [\![0,n]\!]}$, sample a conditionally independent CLE$_4$ $\Gamma_\ell$ for each open region $U_\ell$ enclosed in a loop $\ell \in \Gamma^n$ and let $\Gamma^{n+1} = \bigcup_{\ell\in \Gamma^n} \Gamma_\ell$. The nested CLE$_4$ is defined as $\overline \Gamma:= \bigcup_{n\in\NN^\#} \Gamma^n$, where we recall that $\NN^{\#}=\NN\cup\{0\}$. We call the elements of $\Gamma^n$ as \textbf{$\boldsymbol n$th generation loops}.

Given a zero-boundary GFF $h^D$ in a simply connected domain $U\subsetneq \CC$, we can produce the CLE$_4$ $\Gamma_\ell$ in this iterative construction as a function of the zero-boundary GFF $h^D|_{U_\ell} - \pi Y_\ell$ as in Proposition~\ref{prop:level-line-coupling}. Sampling an embedding $(U,h^C)$ of a critical LQG disk independently from $h^D$ and again defining $h$ as a linear combination of $h^C$ and $h^D$ as in \eqref{eq:supercritical-field}, we have the following coupling of a supercritical LQG disk and a nested CLE$_4$. 

\begin{proposition}[{\cite[Theorem~2.13]{AG23}}]\label{prop:nested-coupling}
    Let $Q\in (0,2]$ and let $(U,h)$ be an embedding of a $Q$-LQG disk with boundary length $L>0$. There exists a coupling of $(h,\overline\Gamma)$ of $h$ with a nested CLE$_4$ $\overline \Gamma$ on $U$ such that the following is true. For each loop $\ell \in \overline \Gamma$, define $U_\ell \subset U$ to be the subdomain enclosed by $\ell$ and $Z_Q(\ell) = \nu_h^\inn(\ell)$ to be the inner boundary length of the loop $\ell$ measured using $h|_{U_\ell}$ as defined in \eqref{eq:inner-length}. Let $\overline \Gamma(\ell):= \{\tilde \ell \in \overline \Gamma: \tilde \ell \subset U_\ell\}$ be the subcollection of CLE$_4$ loops $\overline \Gamma$ which are inside the loop $\ell$.
    \begin{enumerate}[(a)]
        \item For each $n\in \NN$, if we condition on the inner boundary lengths $\{Z_Q(\ell)\}_{\ell \in \Gamma^n}$ of the $n$th level loops, then the $Q$-LQG surfaces $\{(U_\ell, h|_{U_\ell},\overline \Gamma(\ell))/\!\sim_Q\}_{\ell \in \Gamma^n}$ are conditionally independent nested-CLE$_4$-decorated $Q$-LQG disks with given boundary lengths.

        \item In particular, conditioned on the inner boundary lengths $\{Z_Q(\ell)\}_{\ell \in \Gamma^n}$ of the $n$th level loops, those of $(n+1)$th level loops in each $\ell \in \Gamma^n$ --- i.e., $\{Z_Q(\tilde \ell)\}_{\tilde \ell \in \Gamma_\ell}$ ---  are conditionally independent with distribution $\rho_Q^{(Z_Q(\ell))}$ defined in Proposition~\ref{prop:gasket-coupling}. 
    \end{enumerate}
    In this coupling, $\overline \Gamma$ is neither independent from nor a.s.\ determined by $h$.
\end{proposition}

Consequently, the inner boundary lengths $\{Z_Q(\ell)\}_{\ell\in \overline \Gamma}$ form a multiplicative cascade. To specify its law, let us index the loops of nested CLE$_4$ by the Ulam tree 
\begin{equation}\mathcal U := \bigcup_{n\in \NN^\#} \NN^n \qquad \text{where} \quad \NN^0:= \{\root\}. \end{equation} 
That is, each element of $\mathcal U$ other than $\root$ is a word consisting of finitely many positive integers. For $u\in \mathcal U$, let $|u|$ denote the length of the word (``generation"): e.g., if $u = u_1u_2\dots u_n$, then $|u|=n$. Define $|\root|=0$. We assign a partial order $\prec$ on $\mathcal U$ where $u \prec v$ if $|u|<|v|$ and the first $|u|$ letters of $v$ is equal to $u$. If $u \prec v$, then we call $u$ an \textbf{ancestor} of $v$ and $v$ a \textbf{descendant} of $u$. If $u\prec v$ and $|u| + 1 = |v|$, then $u$ is the \textbf{parent} of $v$ and $v$ is a \textbf{child} of $u$. We define $\root$ to be the parent of every element of $\mathcal U$ with length 1. Now, given the coupling $(h,\overline \Gamma)$ as in Proposition~\ref{prop:nested-coupling}, let us inductively build a bijection between $\overline \Gamma$ and $\mathcal U$ so that $\Gamma^n$ is mapped to $\mathcal U^n:=\{u\in \mathcal U: |u|=n\}$. In particular, if $\ell \in \Gamma^n$ is indexed by $u = u_1u_2\dots u_n\in \mathcal U$, then the loops in $\Gamma_\ell$ (i.e., the non-nested CLE$_4$ in $U_\ell$ consisting of the outermost loops of $\overline\Gamma(\ell)$) are indexed by the children of $u$. Moreover, for $k\in \NN$, let $uk = u_1u_2\dots u_n k \in \mathcal U$ be the index of the loop $\ell_k \in \Gamma_\ell$ with the $k$th largest value of the \textit{critical} LQG length $\nu_{h^C}(\ell_k)$ among such loops. The following statement is a rephrasing of the law of $\{Z_Q(\ell)\}_{\ell \in \overline \Gamma}$ described in Proposition~\ref{prop:nested-coupling} given this setup.

\begin{corollary}\label{cor:supercritical-cascade}
    Fix $Q\in (0,2]$ and recall $\beta_Q:= \pi \sqrt{4-Q^2}/Q$.
    Let $\{X^u_k\}_{u\in \mathcal U,k\in\NN}$ be a collection of random variables such that for each fixed $u\in \mathcal U$, the sequence $\{X^u_k\}_{k\in \NN}$ has the law $\rho^{(1)}$ defined in Proposition~\ref{prop:critical-gasket-coupling}. Furthermore, as $u$ is varied in $\cU$, the sequences $\{X^u_k\}_{k\in \NN}$ are mutually independent. Let $\{Y_u\}_{u\in \mathcal U}$ be i.i.d.\ Rademacher random variables sampled independently of $\{X^u_k\}_{u\in \mathcal U,k\in \NN}$. Define $Z(\root)=Z_Q(\root)=1$ and, for $u=u_1u_2\dots u_n \in \mathcal U$,
    \begin{equation}\label{eq:172}
        Z(u) = \prod_{i=1}^{n} X^{u_1 \dots u_{i-1}}_{u_i} \quad \text{and} \quad
        Z_Q(u) = Z(u) \prod_{i=1}^n e^{\beta_Q Y_{u_1\dots u_i}} = \prod_{i=1}^n \big( X^{u_1 \dots u_{i-1}}_{u_i} e^{\beta_Q Y_{u_1\dots u_i}} \big).
    \end{equation}
    Then, under the coupling of a unit boundary length $Q$-LQG disk with nested CLE$_4$ in Proposition~\ref{prop:nested-coupling}, $\{Z_Q(\ell)\}_{\ell \in \overline \Gamma}$ has the same law as $\{Z_Q(u)\}_{u\in \mathcal U}$ with the above bijection between $\overline \Gamma$ and $\mathcal U$. 
\end{corollary}

\section{Non-existence of the supercritical volume measure}\label{sec:3}

In this section, we analyze the multiplicative cascade $Z_Q=(Z_Q(u))_{u\in \cU}$ of inner boundary lengths in the coupling of supercritical LQG with CLE$_4$ presented in Section~\ref{sec:2}. The key observation is Proposition~\ref{prop:many-macroscopic-loops}, which states that there almost surely exist infinitely many quantum-macroscopic CLE loops in this coupling. This shall eventually lead to a proof of Theorem \ref{thm:3}, and this is the main content of this section. 

Another consequence of Proposition~\ref{prop:many-macroscopic-loops} is that the probability $F(p) = \PP_\infty^{(p)}(M \textrm{ is finite})$ tends to 0 as the outermost boundary length $p$ increases  to $\infty$ (Lemma~\ref{lem:7}), which is proved in Section~\ref{sec:prob-dis}. This lemma is the first step in the proof of Proposition~\ref{thm:2}, where we strengthen this first estimate for $F(p)$ to an exponential decay in $p$.

\subsection{The perimeter process contains infinitely many macroscopic loops}\label{sec:macro-loops}

Recall the multiplicative cascade $Z_Q = (Z_Q(u))_{u\in \cU}$ given in \eqref{eq:172}. In Corollary~\ref{cor:supercritical-cascade}, we saw that its law agrees with that of the inner boundary lengths of the CLE$_4$ loops coupled with a unit boundary length $Q$-LQG disk. The main property of the multiplicative cascade $Z_Q$ used in this paper is the following.

\begin{proposition}\label{prop:many-macroscopic-loops}
    For a Borel subset $A\subset (0,\infty)$, denote by $N_n(A;Z_Q)$ the cardinality of the set $\{u\in \mathcal U: |u|=n, Z_Q(u) \in A\}$. If the Lebesgue measure of $A$ is positive, then $N_n(A;Z_Q) \to \infty$ almost surely as $n\to\infty$.
\end{proposition}

Our analysis of the multiplicative cascade $Z_Q$ relies on rephrasing it in terms of a branching random walk. Recall its law given in Corollary~\ref{cor:supercritical-cascade} and observe that the process $S_Q = (S_Q(u))_{u\in\mathcal U}$ given by $S_Q(u) = -\log Z_Q(u)$ is a branching random walk, which almost surely does not go extinct since each element in the multiplicative cascade $Z_Q$ has infinitely many descendants. We rephrase here the results from the branching random walk literature in terms of the multiplicative cascade $Z_Q$. 

The main property of the branching random walk $S_Q$ that we use is the location of its extreme positions. This is a well-studied topic with precise asymptotic estimates \cite{BR09, HS09, AS10, Aid13}. For our purposes, however, the following cruder result on the velocity of the extreme position suffices.

\begin{lemma}\label{lem:max-location}
    For $Q\in (0,2)$, 
    define the \textbf{Biggins transform} of the multiplicative cascade $Z_Q$ as the function
\begin{equation} \label{eq:biggins-transform}
    \phi_Q(\theta) := \EE \Bigg[\sum_{|u|=1} \big(Z_Q(u)\big)^\theta \Bigg].
\end{equation}
Then,
    \begin{equation}\label{eq:max-location-speed}
        \lim_{n\to\infty} \frac{1}{n} \log \bigg(\sup_{|u|=n}Z_Q(u)\bigg) = \mu_Q := \inf_{\theta>0} \frac{1}{\theta} \log \phi_Q(\theta)>0 \quad \text{almost surely}.
    \end{equation}
\end{lemma}

\begin{proof}
    This lemma is an immediate consequence of \cite{Big76} (see also \cite[Theorem~1.3]{Shi15}). All we need to show is $\mu_Q \in (0,\infty)$.

We can compute $\phi_Q(\theta)$ explicitly using the Biggins transform of the multiplicative cascade $Z$ from Corollary~\ref{cor:supercritical-cascade}, which was calculated\footnote{\label{footnote:ccm}The analysis in \cite{CCM20} for the multiplicative cascade generated by the jumps of a spectrally positive $\alpha$-stable process applies to any $\alpha = a-1/2 \in (1,2)$.} in equation (17) of \cite{CCM20} to be
\begin{equation} 
    \phi(\theta) := \EE \Bigg[\sum_{|u|=1} \big(Z(u)\big)^\theta \Bigg] =  \begin{cases} {\sec(\pi \theta)} & \text{if } \theta \in (\frac{3}{2},\frac{5}{2}), \\ +\infty & \text{otherwise}. \end{cases} 
\end{equation}
Recall from \eqref{eq:172} that if $\{Y_i\}_{i\in\NN}$ are i.i.d.\ Rademacher random variables which are independent of $Z$, then for each $i\in \NN$, we have $Z_Q(i) = Z(i) \exp(\beta_Q Y_i)$. Hence, 
\begin{equation} \label{eq:190}
    \phi_Q(\theta) =\EE \Bigg[\sum_{i\in \NN} \big(Z(i)\big)^\theta e^{\beta_Q Y_i \theta} \Bigg] = \phi(\theta) \cdot \frac{e^{\beta_Q\theta}+e^{-\beta_Q\theta}}{2} = \begin{cases} \cosh(\beta_Q \theta)/\cos(\pi \theta) & \text{if } \theta \in (\frac{3}{2},\frac{5}{2}), \\ +\infty & \text{otherwise}. \end{cases}  
\end{equation}
Hence, $\mu_Q<\infty$. 
Note that $\phi_Q(\theta) \to \infty$ as $\theta \downarrow 3/2$ and $\theta\uparrow 5/2$, so $\theta^{-1} \log \phi_Q(\theta)$ attains its minimum $\mu_Q$ in $(3/2,5/2)$. Moreover, $\phi_Q(\theta)>1$ for all $\theta\in(3/2,5/2)$, so $\mu_Q>0$.
\end{proof}

Lemma~\ref{lem:max-location} tells us that, almost surely, $\sup_{|u|=n} Z\inner(u)$ grows exponentially fast. We extrapolate Proposition~\ref{prop:many-macroscopic-loops} from this fact by showing that these large loops have many macroscopic children. This requires the following estimate on the offspring distribution of the cascade $Z_Q$.

\begin{lemma}\label{lem:next-generation-loops}
Let $A\subset (0,\infty)$ be a Borel subset with positive Lebesgue measure. Recall that we denote by  
\begin{equation}
\label{eq:N1}
    N_1(A;Z_Q) = \sum_{|u|=1} \ind\{Z_Q(u) \in A \}  
\end{equation}
the number of first-generation elements of $Z_Q$ whose sizes are in $A$. Then, for any fixed $k\in \NN$, there exists a constant $c>0$ such that 
\begin{equation}
    \PP\big( N_1(L^{-1} A;Z_Q) \leq k \big) = O\big(e^{-cL^{15/16}}\big)
\end{equation}
as $L\to\infty$.
\end{lemma}
\begin{proof}
    From the law of $Z_Q$ in \eqref{eq:172}, if $(Z(i))_{i\in \NN}$ are sampled from the law $\rho_Q^{(1)}$ given in Proposition~\ref{prop:gasket-coupling} and $\{Y_i\}_{i\in \NN}$ are i.i.d.\ Rademacher random variables independent from $\{Z(i)\}_{i\in \NN}$, then $(Z_Q(i))_{i\in \NN} \stackrel{d}{=} (Z(i)\exp(\beta_Q Y_i))_{i\in \NN}$. Hence, with  \begin{displaymath}
        \tilde c = \log (2/(1+e^{-1})),
    \end{displaymath} 
    we have
     \begin{equation}\label{eq:175}
     \begin{split}
        \EE\Big[e^{-N_1(L^{-1} A;Z_Q)}\Big] &= \EE\bigg[\prod_{i\in \NN} \EE\Big[e^{-\ind\{Z(i)e^{\beta_Q Y_i} \in L^{-1}A\}}\Big|Z\Big]\bigg] \\&= \EE\bigg[ \prod_{i\in \NN} \frac{e^{-\ind\{Z(i) \in L^{-1}e^{\beta_Q}A \}} + e^{-\ind\{Z(i) \in L^{-1}e^{-\beta_Q}A \}} }{2} \bigg] \\&\leq \EE \bigg[\prod_{i\in \NN} e^{-\tilde c\ind \{Z(i) \in L^{-1}(e^{\beta_Q} A \cup e^{-\beta_Q}A)\}}\bigg] = \EE\Big[e^{-\tilde c N_1(L^{-1}(e^{\beta_Q} A \cup e^{-\beta_Q} A);Z)}\Big],
    \end{split}\end{equation}
    where $N_1(\cdot;Z)$ used in the last term above is defined by simply replacing $Z_Q$ by $Z$ in \eqref{eq:N1}. For convenience, let us define
    \begin{displaymath}
     \widetilde A = e^{\beta_Q} A \cup e^{-\beta_Q}A
    \end{displaymath}
    for the remainder of this proof.

    Let us now recall the law of $(Z(i))_{i\in \NN}$ from Proposition~\ref{prop:critical-gasket-coupling}. 
    Let $\zeta$ be a $3/2$-stable L\'evy process with no negative jumps started at 0, and let $\tau$ denote the hitting time of $-1$ of this process. That is, for some scale parameter $C>0$, with $\Gamma(\cdot)$ denoting the Gamma function, the L\'evy measure of $\zeta$ is $(C/\Gamma(-3/2))\ind_{\{x>0\}}x^{-5/2}dx$ and $\EE[e^{-\lambda \tau}] = \exp(-(\lambda/C)^{2/3})$ for $\lambda\geq 0$. As mentioned in Footnote~\ref{footnote:scale-par}, the scale parameter $C$ does not change the law of the multiplicative cascade $Z$; hence, we choose $C=1$ for simplicity.  Let $({\Delta \zeta})^\downarrow$ denote the sequence consisting of the sizes of the jumps of $\zeta$ up to the stopping time $\tau$. Then,
     \begin{equation}
        \EE\Big[ e^{-\tilde cN_1(L^{-1}\widetilde A;Z)}\Big] = \frac{\EE\big[ \tau^{-1} \exp\big(-\tilde c\sum_{x \in (\Delta \zeta)^{\downarrow}} \ind \{x \in L^{-1}\widetilde A \} \big)\big]}{\EE\big[\tau^{-1}]}.
    \end{equation}
    For fixed $t_0>0$, the number of jumps of $\zeta$ in the time interval $[0,t_0]$ whose sizes are in $L^{-1} \widetilde A$ is a Poisson random variable with rate given by $t_0$ times the L\'evy measure
    \begin{equation}
         \Pi(L^{-1}\widetilde A) := \frac{1}{\Gamma(-3/2)}\int_{L^{-1}\widetilde A} s^{-5/2}\,ds =  \frac{L^{3/2}}{\Gamma(-3/2)}\int_{\widetilde A} s^{-5/2}\,ds = L^{3/2}\Pi(\widetilde A) .
    \end{equation}
    On the event $\{\tau > t_0\}$, this gives a lower bound for $\sum_{x \in (\Delta \zeta)^{\downarrow}} \ind \{x \in L^{-1}\widetilde A \}$.
    Hence, 
    \begin{equation}
        \EE\Big[\tau^{-1} e^{-\tilde c\sum_{x \in (\Delta \zeta)^{\downarrow}} \ind \{x \in L^{-1}\widetilde A \} } \cdot \ind\{ \tau  > t_0\} \Big] \leq  (t_0)^{-1} e^{-(1-e^{-\tilde c}) \Pi(\widetilde A) L^{3/2} t_0}.
    \end{equation}
    It is well known that $\tau$ is a positive $2/3$-stable random variable. We also have the tail estimate
    \begin{equation}
        \PP(\tau \in dt) = \exp(-(1+o(1))t^{-5/3})\,dt \quad \text{as} \quad t\downarrow 0
    \end{equation}
    from, e.g., \cite[Chapter~2.5]{Zol86}.
    Letting $c_A = (1-e^{-\tilde c})\Pi(\widetilde A) $, for each fixed $\alpha\in (0,3/2)$, we thus have
    \begin{equation}\begin{split}
            \EE \Big[ \tau^{-1} e^{-\tilde c\sum_{x \in (\Delta \zeta)^{\downarrow}} \ind \{x \in L^{-1}\widetilde A \}}  \Big] & \leq \EE\big[ \tau^{-1} \cdot \ind\{\tau \leq L^{-\alpha}\} \big] + \EE\Big[\tau^{-1} e^{-\tilde c\sum_{x \in (\Delta \zeta)^{\downarrow}} \ind \{x \in L^{-1}\widetilde A \} } \cdot \ind\{ \tau  > L^{-\alpha}\} \Big] \\
            &\leq  \int_0^{L^{-\alpha}} t^{-1}\exp(-(1+o(1))t^{-5/3})\,dt + L^\alpha \exp(-c_AL^{\frac{3}{2}-\alpha})\\
            &= \frac{5}{3}\int_{L^{5\alpha/3}}^\infty \frac{1}{s}e^{-(1+o(1))s}\,ds + L^\alpha \exp(-c_AL^{\frac{3}{2}-\alpha}) \\
            &= O\big(e^{-c(L^{5\alpha/3} \wedge L^{3/2 - \alpha})}\big)
            \end{split}
    \end{equation}
    as $L\to \infty$. Optimizing over $\alpha$, we have a constant $c>0$ depending on $A$ such that 
    \begin{equation}\label{eq:176}
        \EE\Big[ e^{-N_1(L^{-1}\widetilde A;Z)}\Big] = O(e^{-cL^{15/16}}) \quad \text{as} \quad L\to\infty.
    \end{equation}
    The lemma now follows by applying the Markov inequality to \eqref{eq:175} and \eqref{eq:176}. 
\end{proof}

\begin{proof}[Proof of Proposition~\ref{prop:many-macroscopic-loops}]
    Fix $\epsilon\in (0,\mu_Q)$, where $\mu_Q$ is the constant defined in \eqref{eq:max-location-speed}.    By Lemma~\ref{lem:next-generation-loops},
    \[ \PP( N_{n+1}(A;Z_Q)\leq k | N_n([e^{n(\mu_Q-\epsilon)},\infty);Z_Q)\geq 1 ) \leq O(e^{-ce^{(15/16)(\mu_Q-\epsilon)n}}). \]
    Moreover, we saw in Lemma~\ref{lem:max-location} that with probability one, $N_n([e^{n(\mu_Q-\epsilon)},\infty);Z_Q)\geq 1$ for all sufficiently large $n$. Since the above upper bound is summable, we conclude that for every positive integer $k$, we have $\PP(N_{n}(A;Z_Q)\leq k \text{ for infinitely many } n ) = 0$. The positive integer $k$ here can be chosen arbitrarily large, so we conclude $\lim_{n\to\infty} N_n(A;Z_Q) = \infty$ almost surely.
\end{proof}

\subsection{Proof of non-existence}\label{sec:proof-nonexistence}
The proof of Theorem~\ref{thm:3} is straightforward given our analysis of the multiplicative cascade $Z_Q$ in the previous section.

\begin{proof}[Proof of Theorem~\ref{thm:3}]
    Suppose $\mathfrak m$ is a functional satisfying the conditions given in the theorem.
    For each $L>0$, let $P^{(L)}$ denote the law of the total volume $\mathfrak{m}_h(\mathbb{D})$ where $(\mathbb D, h)$ is an embedding of the $Q$-LQG disk with boundary length $L$.

    Let $(\mathbb{D},h,\overline \Gamma)$ be an embedding of a unit boundary length $Q$-LQG disk coupled with nested CLE$_4$ as in Proposition~\ref{prop:nested-coupling}. Let $(\ell_u)_{u\in \cU}$ be the indexing of the nested CLE$_4$ $\overline \Gamma$ as described immediately above Corollary~\ref{cor:supercritical-cascade}.   For $u\in\mathcal{U}$, let $D_u$ denote the open set enclosed by $\ell_u$, so that $D_u$ contains $\ell_v$ for every $v\succ u$. Then,
    \begin{equation}
        \mathfrak{m}_h(\mathbb{D}) = \mathfrak{m}_h\bigg(\mathbb{D} \setminus \bigcup_{|u|=1} D_u\bigg) + \sum_{|u|=1} \mathfrak{m}_h(D_u)
    \end{equation}
    almost surely. 
    Recall that conditioned on the inner boundary length $\mathfrak n_h^\inn(\ell_u)$ as defined in \eqref{eq:inner-length}, the equivalence class $(D_u, h|_{D_u})/\!\sim_Q$ has the law of a $Q$-LQG disk with boundary length $\mathfrak n_h^\inn(\ell_u)$. 
    By the assumption on the coordinate change rule \eqref{eq:109}, conditioned on $\ell \in \overline \Gamma$ as well as a conformal map $f$ which takes the open disk $D_\ell$ enclosed by $\ell$ to the unit disk $\mathbb D$, we a.s.\ have $\mathfrak m_h(D_\ell) = \mathfrak m_{h \circ f + Q\log |f'|}(\mathbb D)$ since $h \circ f + Q\log |f'|$ is locally absolutely continuous to a GFF.
    Therefore, the conditional law of $\mathfrak{m}_h(D_u)$ given $\ell_u$ and $\mathfrak n_h^\inn(\ell_u)$ is given by $P^{(\mathfrak n_h^\inn(\ell_u))}$. Moreover, given the inner boundary lengths $\{\mathfrak n_h^\inn(\ell_u)\}_{|u|=1}$, the supercritical LQG surfaces $\{(D_u, h|_{D_u})/\sim_Q\}_{|u|=1}$ are conditionally independent and hence so are the volumes $\{\mathfrak{m}_h(D_u)\}_{|u|=1}$. This implies that the family of conditional laws $\{P^{(L)}\}_{L\in\RR_+}$ satisfies exactly one of the following scenarios.
    \begin{enumerate}[(1)]
        \item There exists a set $A\subset \RR_+$ of zero Lebesgue measure such that for every $L \in \RR_+ \setminus A$,  we have $P^{(L)}(\mathfrak{m}_h(\mathbb D) = 0) = 1$. Since the L\'evy measure of a spectrally positive stable process is absolutely continuous with respect to the Lebesgue measure on $\RR_+$, in this case, by Proposition~\ref{prop:gasket-coupling}, the inner boundary lengths $\{\mathfrak n_h^\inn(\ell_u)\}_{|u|=1}$ of the outermost CLE$_4$ loops are all contained in $A$ with probability one. Hence, this condition is equivalent to the case that, almost surely, the CLE$_4$ gasket $\mathbb{D}\setminus\bigcup_{|u|=1} D_u$ accounts for all of the $\mathfrak{m}_h$-volume of the supercritical LQG disk: i.e., $\mathfrak{m}_h(D_u)=0$ for every $|u|=1$.

        \item On the complement of the first case, there exists a set $A\subset \RR_+$ with a positive Lebesgue measure and a constant $\epsilon>0$ such that if $L\in A$, then $P^{(L)}(\mathfrak{m}_h(\mathbb D)>\epsilon)>\epsilon$.
    \end{enumerate}
    
    We claim that in the first case, $P^{(L)}(\mathfrak{m}_h(\mathbb{D})=0)=1$ for every $L\in(0,\infty)$ and not just on a subset of full Lebesgue measure. For this, let us choose a specific embedding $(\mathbb D, h^{(L)})$ of a $Q$-LQG disk with boundary length $L$. As in Definition~\ref{def:Q-disk}, we let 
    \begin{equation} h^{(L)} = \frac{Q}{2}(h^C + \log L) + \frac{\sqrt{4-Q^2}}{2}h^D =  \frac{Q}{2}h^C + \frac{\sqrt{4-Q^2}}{2}\left( h^D  + \frac{Q}{\sqrt{4-Q^2}} \log L\right) \end{equation}
    where $h^D$ is a zero boundary GFF on $\mathbb D$ independent from $h^C = \hat h \circ f + \log|f'|$ where $\hat h$ is the field on the strip $\mathcal S = \RR \times (0,2\pi)$ in Definition~\ref{def:2-disk} and $f$ is a conformal map from $\mathbb D$ to $\mathcal S$.
    Now choose any $\widetilde L \in (0,\infty)$ such that $P^{(\widetilde L)}(\mathfrak{m}_h(\mathbb{D})=0)=1$.
    Since $h^D  + (Q/\sqrt{4-Q^2}) \log (L/\widetilde L)$ is absolutely continuous with respect to $h^D$ away from the boundary $\partial \mathbb D$ \cite[Proposition~3.4]{ig1}, the two Gaussian fields are absolutely continuous when restricted to the open ball $B_r(0)$ for any $r\in (0,1)$.
    Then, since $\mathfrak{m}_{h^{(\widetilde L)}}(B_r(0)) = 0$ a.s., we must have $\mathfrak{m}_{h^{(L)}}(B_r(0))=0$ almost surely. Letting $r\uparrow 1$, we have $P^{(L)}(\mathfrak{m}_h(\mathbb{D})=0)=1$ for every $L \in (0,\infty)$ as claimed.
    
    In the second case, considering the disks $D_u$ cut out by the $n$th generation CLE$_4$ loops $\ell_u \in \Gamma^n$, we have
    \begin{equation}
        \mathfrak{m}_h(\mathbb{D}) \geq \sum_{|u|=n} \mathfrak{m}_h(D_u)
    \end{equation}
    for every positive integer $n$. Proposition~\ref{prop:many-macroscopic-loops} implies that almost surely, by choosing a sufficiently large $n$, we can find an arbitrarily large number of disks $D_u$ with $|u|=n$ such that $\mathfrak n_h^\inn(\ell_u)\in A$. By Proposition~\ref{prop:nested-coupling}, conditioned on the boundary lengths $\{\mathfrak n_h^\inn(\ell_u)\}_{|u|=n}$, each of these disks with $\mathfrak n_h^\inn(\ell_u)\in A$ has an independent chance of size $\epsilon>0$ such that $\mathfrak{m}_h(D_u)> \epsilon$. Thus, $\mathfrak{m}_h(\mathbb{D}) = \infty$ almost surely. Moreover, since $(D_u, h_{D_u})/\!\sim_Q$ is a $Q$-LQG disk conditioned on its boundary length, we almost surely have $\mathfrak m_h(D_u) = \infty$ for every $u \in \cU$. Now, this implies $\mathfrak m_h(U)=\infty$ a.s.\ for any open subset $U$ of $\mathbb D$, since there almost surely exists $u \in \cU$ such that $D_u \subset U$ (this follows from the local finiteness of nested CLE$_4$ \cite[Lemma~2.1]{APP23}). Now, note that the field $\hat h$ in Definition~\ref{def:2-disk}, when restricted to compact subsets of $\mathcal S= \RR\times (0,2\pi)$ away from $\{0\}\times (0,2\pi)$, is mutually absolutely continuous with respect to a free-boundary GFF. Recalling the conformal map $f:\mathbb D\to \mathcal S$, we conclude the proof by noting that for any domain $U\subset \mathbb D$ at a positive distance away from $\partial \mathbb D \cup f^{-1}(\{0\}\times (0,2\pi))$, the field $h|_U$ is mutually absolutely continuous with respect to a free boundary GFF on $U$ since both $h^C|_U$ and $h^D|_U$ are. For general domains $U \subset \CC$, we apply an affine transformation paired with the coordinate change rule \eqref{eq:109}.
\end{proof}

\subsection{Measure constructed from the multiplicative cascade} \label{sec:cascade-measure}
An important condition in Theorem~\ref{thm:3} was that the supercritical LQG measure $\mathfrak m_h$ is locally determined by the field $h$. 
We conclude this section by describing in Proposition~\ref{prop:cascade-measure} a natural family of random measures on the CLE$_4$-coupled supercritical LQG disk $(\mathbb D,h, \overline \Gamma)/\!\sim_Q$ arising from its multiplicative cascade structure. While they do not satisfy the locality condition as the CLE$_4$ in this coupling is not determined by the field $h$, they satisfy the $Q$-LQG coordinate change formula \eqref{eq:109} and almost surely assign finite and strictly positive values to each Euclidean open set.

Recall from \eqref{eq:190} the notation $\phi_Q(\theta) = \EE[\sum_{|u|=1} (Z_Q(u))^\theta]$, which, for $\theta\in(3/2,5/2)$, is finite and takes the value $\cosh(\beta_Q \theta)/\cos(\pi \theta)$. For $\theta$ in this range, observe that
\begin{equation}
    M_n^{(\theta)}:= \frac{1}{\big(\phi_Q(\theta)\big)^n} \sum_{|u|=n} \big(Z_Q(u)\big)^\theta
\end{equation}
is a martingale with respect to the natural filtration of the multiplicative cascade $Z_Q$. In the branching random walk literature, $M_n^{(\theta)}$ is called the \textbf{additive martingale} of the branching random walk $S_Q(u) = -\log Z_Q(u)$, which was introduced in \cite{Kin75,Big77}. Since $M_n^{(\theta)}$ is nonnegative for all $n$, it converges almost surely to a nonnegative random variable $M_\infty^{(\theta)}$. A fundamental result regarding the additive martingale is the Biggins martingale convergence theorem, which describes the conditions under which the limit $M_\infty^{(\theta)}$ almost surely does not vanish. We describe this condition for the multiplicative cascade $Z_Q$. Recall the notation
\[ \mu_Q:= \inf_{\theta>0}\frac{1}{\theta}\log\phi_Q(\theta)  \]
from Lemma~\ref{lem:max-location}.

\begin{lemma} \label{lem:biggins-martingale-convergence}
For each $Q\in (0,2)$, there exists a unique $\theta^*>0$ such that $(1/\theta^*)\log\phi_Q(\theta^*)=\mu_Q$. Moreover, the following are equivalent.
\begin{enumerate}[(i)]
    \item $\theta \in (\frac{3}{2},\theta^*)$
    \item $\{M_n^{(\theta)}\}_{n\in \NN^\#}$ is uniformly integrable.
    \item $M_\infty^{(\theta)}>0$ a.s.
\end{enumerate}
\end{lemma}
\begin{proof}
    We first show that $\mu_Q=(1/\theta^*)\log\phi_Q(\theta^*)$ is achieved at a unique value of  $\theta^*\in(\frac{3}{2},\frac{5}{2})$. At $\theta^*$, we have
    \begin{equation}\label{eq:177} \frac{d}{d\theta} \bigg(\frac{1}{\theta}\log\phi_Q(\theta)\bigg)\bigg|_{\theta=\theta^*}= \frac{1}{(\theta^*)^2}\bigg(\theta^*\frac{\phi_Q'(\theta^*)}{\phi_Q(\theta^*)}-\log\phi_Q(\theta^*)\bigg)=0. \end{equation}
    Using the explicit formula $\phi_Q(\theta) = \cosh(\beta_Q\theta)/\cos(\pi\theta)$ from \eqref{eq:190}, 
    we furthermore have
    \begin{equation} \frac{d}{d\theta} \bigg(\theta\frac{\phi_Q'(\theta)}{\phi_Q(\theta)}-\log\phi_Q(\theta)\bigg) = \theta\bigg(\frac{\pi^2}{\cos^2(\pi\theta)} + \frac{(\beta_Q)^2}{\cosh^2(\beta_Q\theta)}\bigg)>0 \end{equation}
    for all $\theta \in (\frac{3}{2},\frac{5}{2})$. Hence, there can be at most one solution to the equation \eqref{eq:177}. Moreover, we have $\theta\phi_Q'(\theta)/\phi_Q(\theta)<\log\phi_Q(\theta)$ if $\theta\in (\frac{3}{2},\theta^*)$.

    The Biggins martingale convergence theorem \cite[Theorem~1]{Big77} states that the following three conditions are equivalent:
    \begin{enumerate}[(i)]
    \item $M_\infty^{(\theta)}>0$ a.s. 
    \item $\EE M_\infty^{(\theta)}=1$
    \item $\theta  (\phi_Q)'(\theta)/\phi_Q(\theta) < \log \phi_Q(\theta)$ and $\EE\big[M_1^{(\theta)} \log_+ M_1^{(\theta)}\big]<\infty$.
    \end{enumerate}
    We verify the third condition for each $\theta\in(3/2,\theta^*)$ using the moment estimates for the critical multiplicative cascade $Z$ in \cite[Lemma~13]{CCM20} (see Footnote~\ref{footnote:ccm}). This result states that $\EE\big[(\sum_{|u|=1}Z(u)^\theta)^p\big]<\infty$ for every $p<\frac{5}{2\theta}$. Recall from \eqref{eq:172} that $Z_Q(i) = Z(i) \exp(\beta_Q Y_i)$ for $i\in \NN$ where $\{Y_u\}_{i \in \NN}$ are i.i.d.\ Rademacher random variables that are also independent from the cascade $Z$. Therefore,
    \begin{equation}
        \EE\big[M_1^{(\theta)} \log_+ M_1^{(\theta)}\big] \leq \EE\big[\big(M_1^{(\theta)}\big)^p\big] = \EE\bigg[\bigg(\sum_{|u|=1} (Z(u))^\theta e^{\theta \beta_Q Y_u}\bigg)^p\bigg] \leq e^{p\theta \beta_Q}\, \EE\bigg[\bigg(\sum_{|u|=1} Z(u)^\theta\bigg)^p\bigg] <\infty
    \end{equation}
    choosing any $p\in (1,\frac{5}{2\theta})$.
\end{proof}

Using the explicit formula \eqref{eq:190} for $\phi_Q$, the equation \eqref{eq:177} satisfied by $\theta^*$ is equivalent to
    \begin{equation}\label{eq:194}
        \pi\tan(\pi \theta^*) + \beta_Q\tanh(\beta_Q \theta^*) = \frac{1}{\theta^*} \log \frac{\cosh(\beta_Q \theta^*) }{ \cos(\pi \theta^*)}.
    \end{equation}
The numerical solution of this equation for the range $Q\in (0,2)$ is displayed in Figure~\ref{fig:theta-star}.
    \begin{figure}
        \centering
        \includegraphics[width=0.5\linewidth]{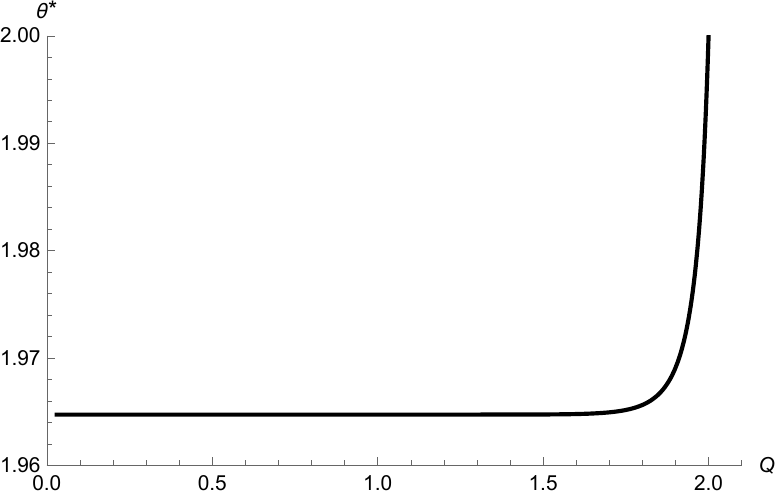}
        \caption{The unique value of $\theta^*$ satisfying $\mu_Q=(1/\theta^*)\log \phi_Q(\theta^*)$, obtained by solving \eqref{eq:194} numerically. We observe $\theta^*\approx 1.9647$ as $Q\to 0$. }
        \label{fig:theta-star}
    \end{figure}

For $\theta\in(\frac{3}{2},\theta^*)$ and $u\in \cU$, define
\begin{equation}
    M_n^{(\theta)}(u) := \frac{1}{\big(\phi_Q(\theta)\big)^n}\sum_{|v|=n, v\succ u} \big(Z_Q(v)\big)^\theta \quad \text{and} \quad W_u^{(\theta)} := \lim_{n\to\infty} M_n^{(\theta)}(u),
\end{equation}
where $W_u^{(\theta)}$ is the a.s.\ limit of the uniformly integrable martingale $\{M_n^{(\theta)}(u)\}_{n\geq |u|}$. Note that, for every $u\in \cU$ and $n\geq k> |u|$, we have
\begin{equation}
    M_n^{(\theta)}(u) = \frac{\sum_{|v'|=n, v'\succ u} \big(Z_Q(v')\big)^\theta}{\big(\phi_Q(\theta)\big)^n} =  \sum_{|v|=k, v\succ u} \frac{\sum_{|v'|=n, v'\succ v} \big(Z_Q(v')\big)^\theta }{\big(\phi_Q(\theta)\big)^n}   = \sum_{|v|=k, v\succ u} M_n^{(\theta)}(v).
\end{equation}
Thus, we have 
\begin{equation}W_u^{(\theta)} = \lim_{n\to\infty} \sum_{|v|=k, v\succ u}  M_k^{(\theta)}(v) \geq \sum_{|v|=k, v\succ u} \lim_{n\to\infty}  M_k^{(\theta)}(v) = \sum_{|v|=k, v\succ u} W_v^{(\theta)}\end{equation} 
almost surely by Fatou's lemma. On the other hand, since $\{M_n^{(\theta)}(v)\}_{n\geq |v|}$ is uniformly integrable for each $v\in \NN$, 
\begin{equation}
    \EE \big[W_u^{(\theta)}\big] = \EE \big[M_{k}^{(\theta)}(u)\big] = \sum_{|v|=k, v\succ u} \EE\big[ M_k^{(\theta)}(v)] =\sum_{|v|=k, v\succ u} \EE\big[W_v^{(\theta)}\big] = \EE \bigg[ \sum_{|v|=k, v\succ u} W_v^{(\theta)} \bigg].
\end{equation}
Hence,
\begin{equation}
    W_u^{(\theta)} = \sum_{|v|=k, v\succ u} W_v^{(\theta)} \quad \text{a.s.}
\end{equation}
and thus $W_u^{(\theta)}$ defines a random measure on $\cU$ with respect to the product $\sigma$-algebra.

The analog of Lemma~\ref{lem:biggins-martingale-convergence} holds for the \textbf{derivative martingale} 
\begin{equation}
     M_n^* := \sum_{|u|=n} \frac{(Z_Q(u))^{\theta^*}}{(\phi_Q(\theta^*))^n}\bigg( n \frac{\phi_Q'(\theta^*)}{\phi_Q(\theta^*)} - \log Z_Q(u)\bigg) = -\frac{dM_n^{(\theta)}}{d\theta}\bigg|_{\theta=\theta^*}
\end{equation}
defined in \cite{BK04}.
This fact can be verified by checking the conditions in \cite{AS14} again using the formulas in \cite{CCM20}. Hence, 
\begin{equation}
    W_u^* := \lim_{n\to\infty} \sum_{|v|=n, v\succ n} \frac{Z_Q(v)^{\theta^*}}{\phi_Q(\theta^*)^n} \bigg(n \frac{\phi_Q'(\theta^*)}{\phi_Q(\theta^*)} - \log Z_Q(v)\bigg)
\end{equation}
also defines a random measure for the product $\sigma$-algebra on $\cU$. 

The following proposition gives a construction for a random finite measure $\xi$ which is almost surely determined by the CLE$_4$-coupled supercritical LQG disk $(\mathbb D, h, \overline \Gamma)/\!\sim_Q$ through the multiplicative cascade $Z_Q$.

\begin{proposition} \label{prop:cascade-measure}
    Let $(\ell_u)_{u\in\cU}$ be the loops of a nested CLE$_4$ in $\mathbb{D}$ indexed so that $u \prec v$ if and only if $\ell_u$ encloses $\ell_v$. Denote the open set enclosed by $\ell_u$ as $D_u$. Suppose $\{W_u\}_{u\in \cU}$ is a collection of nonnegative random variables which satisfy $W_u = \sum_{k\in \NN} W_{uk}$ for every $u\in\cU$. Then, there exists an almost surely unique random Borel measure $\xi$ on $\overline{\mathbb{D}}$ such that $\xi(\overline{D_u}) = W_u$. Moreover, if $\xi^{(n)}$ is any sequence of random Borel measures on $\overline{\mathbb{D}}$ such that $\xi^{(n)}(\overline{D_u}) = W_u$ almost surely for every $|u|\leq n$, then, almost surely, $\xi^{(n)}$ converges weakly to $\xi$ as $n\to\infty$.
\end{proposition}

\begin{proof}
    Let $\mathcal C$ be a countable collection of nonnegative continuous functions on $\overline{\mathbb{D}}$ such that two (deterministic) Borel measures $\mu$ and $\tilde \mu$ on $\overline{\mathbb{D}}$ are the same if $\int_{\overline{\mathbb D}} f\,d\mu = \int_{\overline{\mathbb D}} f\,d\tilde \mu$ for every $f\in \mathcal C$. 

    \textit{Uniqueness.} Suppose $\xi$ is a random Borel measure on $\overline{\mathbb{D}}$ that almost surely satisfies $\xi(\overline{D_u}) = W_u$ for every $u\in\cU$. For each $f\in \mathcal C$ and $n\in \NN$, define 
    \begin{equation}
        W_{f-}^{(n)} := \sum_{|u|=n} W_u \cdot \inf_{z\in \overline{D_u}} f(z) \quad \text{and} \quad W_{f+}^{(n)}:= \sum_{|u|=n} W_u \cdot \sup_{z\in \overline{D_u}} f(z).
    \end{equation}
    Note that since $\xi(\overline{D_u}) = \sum_{k\in\NN}\xi(\overline{D_{uk}})$, we have $\xi(\overline{D_u} \setminus (\bigcup_{k\in \NN} \overline{D_{uk}})) = 0$ a.s for every $u\in \cU$. Hence,
    \begin{equation}\label{eq:biggins-measure-ineq}
      W_{f-}^{(n)} = \sum_{|u|=n} \xi(\overline{D_u}) \cdot \inf_{z\in \overline{D_u}} f(z) \leq \int_{\overline{\mathbb D}} f\,d\xi \leq \sum_{|u|=n} \xi(\overline{D_u}) \cdot \sup_{z\in \overline{D_u}} f(z) = W_{f+}^{(n)} \quad \text{a.s.}
    \end{equation}
    Also, since $\overline{D_{uk}} \subset \overline{D_u}$ for every $k\in\NN$,
    \begin{equation}
        W_{f-}^{(n)} = \sum_{|u|=n} \xi(\overline{D_u}) \cdot \inf_{z\in \overline{D_u}} f(z) \leq \sum_{|u|=n} \sum_{k\in\NN} \xi(\overline{D_{uk}}) \cdot \inf_{z\in \overline{D_{uk}}} f(z) = W_{f-}^{(n+1)} \quad \text{a.s.}
    \end{equation}
    and similarly $W_{f+}^{(n)} \geq W_{f+}^{(n+1)}$ a.s.\ for every $n$. 
    By \cite[Lemma~2.1]{APP23}, 
    \begin{equation}\lim_{n\to\infty}\sup_{|u|=n} \mathrm{diam}(\overline{D_u}) \to 0 \quad \text{a.s.}\end{equation}
    Hence, for each $\epsilon>0$, there exists a positive integer $N$ such that 
    \begin{equation}\PP\bigg(\sup_{|u|\geq N} \sup_{x,y\in\overline{D_u}} |f(x)- f(y)| \leq \epsilon\bigg) \geq 1-\epsilon .\end{equation} 
    On this event, $0 \leq W_{f+}^{(n)} - W_{f-}^{(n)} \leq \epsilon \sum_{|u|=n} W_u = \epsilon W_\root$ for every $n\geq N$. Taking $\epsilon \downarrow 0$, we have that $\lim_{n\to\infty} W_{f\pm}^{(n)} = \int_{\overline{\mathbb D}} f\,d\xi$ almost surely. Since $W_{f\pm}^{(n)}$ does not depend on the choice of $\xi$, we conclude that this random measure is almost surely unique.
    
    \textit{Existence.} Suppose $\xi^{(n)}$ is a sequence of random Borel measures on $\overline{\mathbb{D}}$ such that $\xi^{(n)}(\overline{D_u}) = W_u$ almost surely for every $|u|\leq n$. For instance, we may choose $\xi^{(n)}$ where, if $\mathrm{Leb}$ is the Lebesgue measure on $\overline{\mathbb{D}}$, then
    \begin{equation}
        \frac{d\xi^{(n)}}{d\mathrm{Leb}}(z) = \sum_{|u|=n} \frac{W_u}{\mathrm{Leb}(D_u)} \ind\{z\in \overline{D_u}\}.
    \end{equation}
    We have $W_{f-}^{(n)} \leq \int f\,d\xi^{(n)} \leq W_{f+}^{(n)}$ a.s.\ for continuous $f:\overline{\mathbb{D}}\to [0,\infty)$, in an analogous way to \eqref{eq:biggins-measure-ineq}. Hence, 
\begin{equation}
    \lim_{n\to\infty} \int_{\overline{\mathbb D}} f\,d\xi^{(n)} = \lim_{n\to\infty} W_{f\pm}^{(n)} =: W_f \quad \text{a.s.}
\end{equation}
The convergence $\int_{\overline{\mathbb D}}  f\,d\xi^{(n)} \xrightarrow{d} W_f$ for all nonnegative continuous $f$ on $\overline{\mathbb D}$ implies the existence of a random measure $\xi$ such that $(\int_{\overline{\mathbb D}}  f\,d\xi)_{f\in\mathcal C} \stackrel{d}{=} (W_f)_{f\in \mathcal C}$ jointly \cite[Lemma~5.1]{Kal86}. Recall that $W_f$ is a measurable function of $(W_u,\ell_u)_{u\in\cU}$. Since $(\int_{\overline{\mathbb D}}  f\,d\xi)_{f\in \mathcal C}$ determines $\xi$ uniquely at each point in the probability space, there exists a coupling of $\xi$ with $(W_u,\ell_u)_{u\in\cU}$ such that $\int_{\overline{\mathbb D}}  f\,d\xi = W_f$ a.s.\ for every $f\in \mathcal C$, and hence for every continuous $f$ on $\overline{\mathbb{D}}$. 

Let us now check that $\xi(\overline{D_u}) = W_u$ a.s.\ for every $u\in \cU$. Fixing $u\in \cU$ and choosing a sequence of random continuous functions $f_k$ approximating $\ind_{D_u}$ from below, we have 
\begin{equation}
    \xi(\overline{D_u}) \geq \sup_k W_{f_k} = \sup_k \sup_n W_{f_k -}^{(n)} \quad \text{a.s.}
\end{equation}
Since $\xi^{(n)}$ is supported on $\bigcup_{|v|=n} \overline{D_v}$, which almost surely does not intersect $\eta_u = \partial {D_u}$ when $n>|u|$, we have $\sup_k W_{f_k-}^{(n)} = \xi^{(n)}(\overline{D_u}) = W_u$ almost surely. Hence, $\xi(\overline{D_u}) \geq W_u$ a.s.\
 for every $u\in \cU$. Since the loops of nested CLE$_4$ are a.s.\ disjoint, we have 
 \begin{equation}
     \sum_{|u|=n} W_u \leq \sum_{|u|=n} \xi(\overline{D_u}) \leq \xi(\overline{\mathbb{D}}) = \sum_{|u|=n} W_u \quad \text{a.s.}
 \end{equation}
Therefore, $\xi(\overline{D_u}) = W_u$ almost surely for every $u\in \cU$.
\end{proof}
A natural choice for the process $W=(W_u)_{u\in \cU}$ is to set either $W = W^{(\theta)}$ for $\theta\in(\frac{3}{2},\theta^*)$ or $W = W^*$. For any of these choices, since the $Q$-LQG lengths $(\mathfrak n_h(\ell_u))_{u\in \cU}$ are invariant under different choices of embedding the nested-CLE$_4$-decorated supercritical LQG disk, the measure $\xi$ that we constructed also satisfies the $Q$-LQG coordinate change rule \eqref{eq:109}. We do not have a suggestion for a value of $\theta$ that has a special meaning within the context of Liouville theory with central charge $\mathbf{c}_{\mathrm L} \in (1,25)$.

\begin{remark}
The measure $\xi$ violates the locality condition of Theorem~\ref{thm:3}; that is, $\xi(U)$ is not determined by the domain $U\subset \mathbb D$ and the local field $h|_U$. Indeed, given the CLE$_4$ loop $\ell_u \in \overline \Gamma$ and the field $h|_{D_u}$ inside the loop $\ell_u$, the value of $\xi(D_u)$ still depends on $|u|$ --- i.e., the generation of the loop $\ell_u$ within the nested CLE$_4$ $\overline \Gamma$ --- since $\phi_Q(\theta)>1$ for any $Q\in (0,2)$ and $\theta \in (3/2,\theta^*)$.

Note that $\xi$ is almost surely supported on the set of points which are surrounded by infinitely many CLE$_4$ loops. Using the many-to-one formula for the branching random walk (see, e.g., \cite[Theorem~1.1]{Shi15}), one can check that if $W = W^{(\theta)}$ (resp.\ $W^*$), then the corresponding random measure $\xi$ is a.s.\ supported on the points $z\in \mathbb D$ such that if $\ell_{u_n}$ is the $n$th generation CLE$_4$ loop enclosing $z$, then $\frac{1}{n}\log Z_Q(u_n) \to \phi_Q'(\theta)/\phi_Q(\theta)$ (resp. $\phi_Q'(\theta^*)/\phi_Q(\theta^*)$) as $n\to\infty$. It would be interesting to investigate how these points are related with the thick points of the field.
\end{remark}

\section{Background on random planar maps}
\label{sec:boltzmann}

We now move on to the second part of this article, which studies the random planar map model of a supercritical LQG disk introduced in Section~\ref{sec:discrete-model}.
We begin this section by giving the precise definitions for the laws of planar maps involved in the construction of this discrete model. Its connection to the continuum model described in Section~\ref{sec:2} is given in Proposition~\ref{prop:26}, which states that the perimeters of marked faces in our discrete model converge in the scaling limit to the lengths of CLE$_4$ loops on the supercritical LQG disk.

We also introduce the random walk representation of the perimeters of faces in a Boltzmann map (Proposition~\ref{prop:10}), which is a key tool in the proof of Proposition~\ref{thm:2} in Section~\ref{sec:prob-dis}. We finally survey the convergence results for subcritical Boltzmann maps (Propositions~\ref{prop:4} and \ref{prop:6}) that form the key input for the proof of Theorem~\ref{thm:1} in Section~\ref{sec:conv}.

\subsection{Definitions of Boltzmann maps and ring distributions} \label{sec:maps-def}

There were two kinds of planar maps needed in the construction of Definition~\ref{def:model}: Boltzmann maps and rings.
We first give the definition of a Boltzmann map with a fixed perimeter for the root face. Recall that $\per(f)$ refers to the half-perimeter of the face $f$.
\begin{definition}\label{def:boltzmann-map}
For $p\in \NN$, let $\mathcal M^{(p)}$ denote the set of all finite bipartite rooted planar maps with the degree of the root face equal to $2p$.
  Given a \textbf{weight sequence} $\bfq=(q_i)_{i\in \NN}$ of nonnegative real numbers, define the \textbf{Boltzmann weight} of a map $M \in \bigcup_{p\in \NN} \mathcal M^{(p)}$ as
  \begin{equation}
    \label{eq:73}
    w_{\bfq}(M)=\prod_{f\in \fF(M)} q_{\per(f)},
  \end{equation}
  where $\fF(M)$ refers to the collection of inner (non-root) faces of $M$.
 We say that the weight sequence $\bfq$ is \textbf{admissible} if $W_\bfq^{(p)}:=\sum_{M\in \mathcal M^{(p)}} w_\bfq(M) < \infty $ for all $p \in \NN$. In this case, we define for each $p\in \NN$ the corresponding probability measure 
 \begin{equation}
     \PP_\bfq^{(p)}(M) = \frac{1}{W_\bfq^{(p)}} w_\bfq(M), \quad M \in \mathcal M^{(p)}
 \end{equation}
 of Boltzmann maps with boundary length $2p$.
\end{definition}

Given a weight sequence $\bfq$, define
\begin{equation}
    f_\bfq(x):= \sum_{k=1}^\infty \binom{2k-1}{k-1}q_kx^{k-1}, \quad x\geq 0.
\end{equation}
By \cite[Proposition 1]{MM07}, $\bfq$ is admissible if and only if the equation $f_\bfq(x)=1-1/x$ has a positive solution. Let $Z_\bfq$ denote the smallest such solution, so that the measure $\mu_\bfq$ in Definition~\ref{def:classification} below is a probability measure. It arises as the offspring distribution of a Galton--Watson tree associated to a Boltzmann map with weight sequence $\bfq$ via the Bouttier--Di Francesco--Guitter bijection \cite{BDG04} with the Janson--Stef\'ansson trick \cite{JS15}; see \cite[Proposition~4]{CCM20} for further details.

\begin{definition}\label{def:classification}
  If $\bfq$ is an admissible weight sequence, we associate to it a probability measure $\mu_{\bfq}$ on nonnegative integers given by
  \begin{equation}
    \label{eq:81}
     \quad \mu_{\bfq}(k)=(Z_{\bfq})^{k-1} \binom{2k-1}{k-1} q_k, \quad k\geq 1
  \end{equation}
  and $\mu_{\bfq}(0) = 1/Z_{\bfq}$, where $Z_\bfq$ is the smallest positive solution to the equation $f_\bfq(x) = 1-1/x$.
\end{definition}

The measure $\mu_\bfq$ allows us to classify the weight sequence $\bfq$ depending on the large-scale behavior of Boltzmann maps sampled from $\PP_\bfq^{(p)}$. The following classification is drawn from \cite[Definition~1]{CCM20} and \cite[Section~5.3]{Cur19}.

\begin{definition}
  \label{def:nongen}
    An admissible weight sequence $\bfq = (q_i)_{i\in \NN}$ is said to be \textbf{critical} if the associated probability measure $\mu_\bfq$ in \eqref{eq:81} has mean 1 and \textbf{subcritical} if its mean is strictly less than 1.
    
    A critical weight sequence $\bfq$ is said to be \textbf{generic critical} if $\mu_\bfq$ has a finite variance. A critical weight sequence $\bfq$ is \textbf{non-generic critical of type $a\in (3/2,5/2)$} if $k^{a+1/2}\mu_\bfq(k)$ converges to a positive constant as $k\to\infty$.
\end{definition}

In the critical generic case, there are no macroscopic faces and the scaling limit of the Boltzmann maps as the number of total vertices tends to infinity is the celebrated \textit{Brownian map} \cite{MM07,LeGal11,Mar18}. The analogous scaling limit in the critical non-generic case depends on the value of $a$ and is called a \textit{stable map} \cite{LM11}. The scaling limit in the subcritical case is the continuum random tree; see Section~\ref{sec:crt} for further details. 

For the remainder of this work, we fix $\bfq = (q_i)_{i\in \NN^\#}$ to be a \textbf{non-generic critical weight sequence of type $\boldsymbol{a=2}$}. An important example of such a weight sequence is one corresponding to the gasket of a rigid $O(2)$ loop model-decorated quadrangulation.\footnote{The works \cite{ADH24,Kam24} show that the gasket of a rigid $O(2)$ loop-decorated bipartite map is a non-generic critical Boltzmann map of type $a=2$ in a more general sense than Definition~\ref{def:nongen}, where a slowly varying correction is allowed for the tail of $\mu_\bfq$ as in \cite[Definition~2.2]{Ric18} and \cite[Definition~1]{CR20}. 
For the gaskets of rigid $O(2)$ loop-decorated quadrangulations, A\"id\'ekon, Da Silva, and Hu verify that $k^{5/2}\mu_\bfq(k)$ either asymptotic to a constant or grows logarithmically as $k\rightarrow \infty$ \cite[Equation~B.8]{ADH24}, and gives an explicit condition for the former case.

We believe that Theorem~\ref{thm:1} would continue to hold for random planar map models of supercritical LQG constructed using this more general class of non-generic critical Boltzmann maps of type $a=2$.
In this case, with slowly varying corrections to the estimates for $T_S^{(p)}$ and $L_S^{(p)}$ in Lemma~\ref{lem:215}, our results continue to hold prior to Lemma~\ref{lem:18}. However, we then no longer have the exponential decay in \eqref{eq:lowtail}, so our proof of Lemma~\ref{lem:18} is not valid with this generalization.}
More generally, a Boltzmann map $M$ sampled from the law $\PP_{\bfq}^{(p)}$ associated with such a weight sequence can be thought of as a discrete version of the $\mathrm{CLE}_4$ gasket, in the sense that the set of perimeters $\{\per(f_i)\}_{i\in \fF(M)}$ of the non-boundary faces converges in the scaling limit to the corresponding set of lengths of loops in a critical LQG disk decorated with independent $\CLE_4$ (see Lemma~\ref{lem:13}). 

We now define the distribution of rings which puts our discrete model in the universality class of supercritical LQG with given central charge. As promised earlier in the introduction, we allow the following more general distribution of rings in the rest of this article.
\begin{definition}
  \label{def:ring}
  Let $\big(\PP^{(p)}_{\ring}\big)_{p \in \NN}$ be a sequence of probability measures on planar maps $R$ with two distinguished faces --- an outer face $f_{\mathrm{out}}$ and an inner face $f_{\mathrm{in}}$ --- such that $\per(f_{\mathrm{out}}) = p$ and $\per(f_\inn) \in \NN^\#$. Let $\Rat(R)$ denote the ratio $\per(f_{\mathrm{in}})/\per(f_{\mathrm{out}})$ of the ring $R$. 
  We assume that this sequence of distributions satisfy the following conditions.
  
  \begin{enumerate}[(1)]
  \item \textit{Background charge:} Recall the background charge $Q\in (0,2)$ associated with supercritical LQG. The law of $\Rat(R)$ under $\PP_\ring^{(p)}$ converges in distribution to $\exp(\beta_Q Y)$ as $p\to\infty$, where $Y$ is a Rademacher random variable and
  \begin{equation}\label{eq:beta}
      \beta_Q :=\frac{\pi \sqrt{4-Q^2}}{Q}.
    \end{equation}

 \item \textit{Non-thickness:} Let $\Vol(R)$ denote the total number of vertices in $R$. There is a constant $C>0$ such that, for any $p\in \NN$, a ring $R$ sampled from $\PP_{\ring}^{(p)}$ satisfies $\Vol(R)\leq C(\per(f_\mathrm{out})+ \per(f_{\mathrm{in}}))$ almost surely. 

 \item \textit{Non-triviality:} $\PP_\ring^{(p)}(\Rat(R)>0) > 0$ for all $p \in \NN$.
  
  \item \textit{Lower tail: }For each $\lambda>0$, there exists a constant $c>0$ such that $\EE_\ring^{(p)} [ e^{-\lambda p\,\Rat (R)}] \leq e^{-cp}$ for all $p\in \NN$.
  
  \item \textit{Upper tail:} There exists a constant $\delta>0$ such that $\sup_{p\in\NN} \EE_\ring^{(p)}[(\Rat(R))^{(2+\delta)}]<\infty$. 

  \item \textit{Gluing of boundaries:} There exists a constant $K<\infty$ such that, for any $p\in \NN$, a ring $R$ sampled from $\PP_\ring^{(p)}$ satisfies the following property almost surely:
  each vertex on the outer boundary of $R$ is at most distance $K$ from its inner boundary, and vice versa.\footnote{This condition is not needed in our proofs. However, it is imposed so that planar maps sampled from the \emph{unconditioned} law $\PP_\infty^{(p)}$ can be reasonably expected to converge in the scaling limit to the supercritical LQG disk decorated by CLE$_4$. For instance, this condition excludes a ring where the inner and outer boundaries are connected through a single edge, since this would yield an outsize impact to the distances between the vertices to the inside and the outside of the ring.}
  
  \end{enumerate}
\end{definition}

The main condition that we require is (1). We emphasize that the dependence of our planar map model on the central charge $\cL$ is precisely through the constant $\beta_Q$ appearing in the law of $\Rat(R)$. In terms of the continuum picture in \cite{AG23}, the limiting law in the first condition corresponds to the ratio between the inner and outer perimeters of a CLE$_4$ loop in a CLE$_4$-decorated supercritical LQG disk (see \eqref{eq:length-ratio}). 
The remaining conditions are technical hypotheses needed in our proofs which we do not claim to be optimal. For instance, with a more restrictive law for the rings (e.g., $\per(f_\inn)\stackrel{d}{=} \lfloor p\exp(\beta_Q Y)\rfloor$ for all $p\in \NN$ as in \cite{AG23}), many of the estimates in our work can be improved. We refer the reader to Remark~\ref{rem:gen} for a further discussion of this point. 

\begin{remark}
  \label{rem:model}
  Our construction of the planar map in Definition~\ref{def:model} has two main differences from that in \cite{AG23}. First, in place of gaskets of the fully-packed $O(2)$ loop-decorated triangulations, we use critical non-generic Boltzmann maps of type $a=2$. This change allows us to extend and apply the results of \cite{CCM20}, which provided an exact solvability of the perimeter process of the loop $O(n)$ model for $n\in (0,2)$ via an analysis of the gaskets in terms of the corresponding Boltzmann maps. We note that, for $n\in (0,2)$, the distribution of the gasket of a loop $O(n)$ model-decorated planar map has been identified in \cite{BBG11} with a critical non-generic Boltzmann map with $a=2+(1/\pi)\arccos(n/2)$. Furthermore, the above was extended to the case $n=2$ independently in the works \cite{ADH24} and  \cite{Kam24}.

  The other difference in our model from that of \cite{AG23} is the general class of rings that we allow in Definition~\ref{def:ring}. Since Definition~\ref{def:model} outputs an infinite map with positive probability (see Lemma~\ref{lem:1}), it is a priori conceivable that minor changes in the sequence $\big(\PP_\ring^{(p)}\big)_{p\in\NN}$ may drastically affect the law $\PP_\ttF^{(p)}$. Our main result Theorem~\ref{thm:1} implies that for any reasonable choice of ring distribution, its effect on the global geometry of the planar map, when conditioned to be finite, becomes negligible as $p\to\infty$.  
\end{remark}

\subsection{The connection between Boltzmann maps and random walks}
\label{sec:perproc}
An important ingredient in this work is the understanding of the distribution of the perimeters of the faces of a Boltzmann map sampled from $\PP_{\overline \bfq}^{(p)}$ --- i.e., the probability measure on rooted bipartite maps with boundary perimeter $2p$ which is proportional to the Boltzmann weight \eqref{eq:73}. Given a Boltzmann map $M$ sampled from $\PP_{\overline \bfq}^{(p)}$, let $\bchi_M=(\chi_M(i))_{i\in \NN}$ encode the half-perimeters $\per(f)$ for $f\in \fF(M)$ listed in a non-increasing order with duplicity, padded by $\chi_M(i)=0$ for $i>|\fF(M)|$. A key tool in our proof is the encoding of the law of $\bchi_M$ is the following random walk.

\begin{definition}
  \label{def:4}
  Let $\overline \bfq$ be an admissible weight sequence. Consider the random walk $S_n=\sum_{i=1}^n X_i$ starting at $0$ and with steps $X_i$ sampled independently from the distribution 
  \begin{equation}\label{eq:step-dist}
    \PP\{X_i=k\} = \mu_{\overline \bfq}(k+1),\quad k\in \{-1,0,1,\dots\},
\end{equation}
where $\mu_{\overline \bfq}$ is the probability measure associated with the weight sequence $\bfq$ defined in \eqref{eq:81}.
  Let $T^{(p)}_S=\inf\{n\geq 0 : S_n=-p\}$ be the first time that this walk hits $-p$. Let $L^{(p)}_S=\sum_{i=1}^{T^{(p)}_S}\ind_{\{X_i=-1\}}$ be the total number of negative steps (always $-1$) taken by the walk up to time $T^{(p)}_S$. Let $\bchi_S^{(p)}=(\chi_S^{(p)}(i))_{i\in \NN}$ be the non-increasing sequence of integers satisfying 
  \begin{equation}
    \label{eq:120}
  \{\chi_S^{(p)}(i)\}_{i\in [\![1,T_S^{(p)}]\!]}=\{X_n+1\}_{n\in[\![1,T_S^{(p)}]\!]}
  \end{equation}
  as multi-sets and $\chi_S^{(p)}(i)=0$ for $i>T^{(p)}_S$.
\end{definition}

 The following identity is given in \cite[Section 2.3.1]{CCM20} based on the bijections of \cite{BDG04,JS15}.
\begin{proposition}
  \label{prop:10}
   Let $\bfq$ be an admissible weight sequence and let $M$ be a Boltzmann map sampled from $\PP_{\bfq}^{(p)}$ for any $p\in \NN$. Then, the law of $\bchi_M$ is equal to that of $\bchi_S^{(p)}$ weighted by $(L_S^{(p)}+1)^{-1}$. That is, for any positive measurable function $F\colon \RR^\NN\rightarrow \RR$, we have
  \begin{equation}
    \label{eq:85}
    \EE_{\overline \bfq}^{(p)}[F(\bchi_M)]=\frac{\EE[(L_S^{(p)}+1)^{-1}F(\bchi_S^{(p)})]}{\EE[(L_S^{(p)}+1)^{-1}]}.
  \end{equation}
\end{proposition}

Our proofs of Theorems \ref{thm:2} and \ref{thm:1} in Sections~\ref{sec:prob-dis} and \ref{sec:estimates} make heavy uses of this random walk encoding of Boltzmann maps and its basic properties that we now collect. We first consider properties of the random walk associated with our critical non-generic weight sequence $\bfq$ of type $a=2$. These results can be stated with slight modification for other values of $a\in(3/2,5/2)$ as well, but we do not stray away from the $a=2$ case that we consider exclusively in our work.

\begin{lemma}
  \label{lem:215}
  Let $\bfq$ be a critical non-generic weight sequence of type $a=2$. Let $S_n$ be the associated random walk and $T_S^{(p)}$ and $L_S^{(p)}$ be the random variables defined as in Definition~\ref{def:4}. 
 Then, $p^{-3/2} T_S^{(p)}$ converges in distribution to a positive $2/3$-stable random variable. Moreover, there is a constant $c>0$ depending only on $\bfq$ such that for all $\epsilon>0$ and $p\in \NN$, we have
  \begin{equation}\label{eq:171}
\PP(T_S^{(p)}\leq \epsilon p^{3/2})\leq \PP(L_S^{(p)}\leq \epsilon p^{3/2})\leq e^{-c\epsilon^{-2}}.
\end{equation}
Consequently, there exists a constant $C>0$ such that $\EE[ (L_S^{(p)}+1)^{-1}]\geq C p^{-3/2}$ for all $p\in \NN$.
\end{lemma}
\begin{proof}
The convergence of $p^{-3/2}T_S^{(p)}$ is established within the proof of \cite[Proposition~3]{CCM20} and \eqref{eq:171} is proved in Lemma~5 of the same article.
    To prove the last estimate, it suffices to show $\EE[(T_S^{(p)})^{-1}]\geq Cp^{-3/2}$ since $L_S^{(p)}\leq T_S^{(p)}$. We know from \eqref{eq:171} that $\{p^{3/2}(T_S^{(p)})^{-1}\}_{p\in \NN}$ is uniformly integrable. Denoting by $\tau$ the positive 2/3-stable random variable that $p^{-3/2}T_S^{(p)}$ converges in law to, we conclude $p^{3/2}\EE[1/T_S^{(p)}]\rightarrow \EE[1/\tau]$. This limit is finite, as can be checked from well-known estimates on the density of $\tau$ (see, e.g., \cite[Chapter~2.5]{Zol86}).
\end{proof}

We now move on to estimates for $T_S^{(p)}$ and $L_S^{(p)}$ corresponding to a subcritical weight sequence. Let 
\begin{equation}
    m_{\overline\bfq} := \sum_{k=-1}^\infty k\mu_{\overline\bfq}(k+1) \in (-1, 0)
\end{equation}
be the mean of the step distribution \eqref{eq:step-dist} of the associated random walk $S$.
\begin{lemma}
  \label{lem:33}
  Let $\overline \bfq$ be a subcritical admissible weight sequence. Then,
  \begin{equation}
    \label{eq:128}
    \lim_{p\to\infty} \PP\big(T_S^{(p)} > -2p/m_{\overline \bfq}\big)= 0.
  \end{equation}
  Furthermore, for all $p\in \NN$, we have $p^{-1}\EE T_S^{(p)}=\EE T_S^{(1)}<\infty$ and
  \begin{equation}
    \label{eq:154}
    \EE \big[(L_S^{(p)}+1)^{-1}\big]\geq \EE \big[(T_S^{(p)}+1)^{-1}\big] \geq \big(p \,\EE [T_S^{(1)}] +1\big)^{-1}.
  \end{equation}
\end{lemma}
\begin{proof}
  Note that $\{T_S^{(p)} > -2p/m_{\overline \bfq}\}\subseteq \{S_{\lfloor-2p/m_{\overline \bfq}\rfloor} >-p\}$. Since $\EE [S_{\lfloor-2p/m_{\overline \bfq}\rfloor}] \sim -2p$, by the law of large numbers, the probability of the latter event decreases to $0$ as $p\rightarrow \infty$.

  To obtain \eqref{eq:154}, note that $\EE T_S^{(p)}=p\cdot \EE T_S^{(1)}$ since the only possible negative jumps of $S$ are $-1$. We have $\EE T_S^{(1)}<\infty$ from, e.g., \cite[Theorem 2.1]{Gut74}. The claim then follows by Jensen's inequality since the function $x\mapsto (x+1)^{-1}$ is convex.
\end{proof}

As a simple instance of extracting information about Boltzmann maps from the corresponding random walk, we show that \eqref{eq:128} gives an upper bound on the number of faces in a subcritical Boltzmann map. Recall that for a rooted planar map $M$, we use $\fF(M)$ to denote the set of interior (or non-root) faces of $M$.

\begin{lemma}
  \label{lem:38}
  If $\overline \bfq$ is a subcritical weight sequence, then 
  $\PP_{\overline \bfq}^{(p)}(|\fF(M)|\leq -2p/m_{\overline \bfq})\to 1$ as $p\to\infty$.
\end{lemma}
\begin{proof}
By Proposition~\ref{prop:10}, it suffices to show that
  \begin{equation}
    \label{eq:147}
    \lim_{p\to\infty} \frac{\EE[ (L_S^{(p)}+1)^{-1} \ind\{T_S^{(p)} > -2p/m_{\overline \bfq}\}]}{\EE[(L_S^{(p)}+1)^{-1}]}= 0.
  \end{equation}
  In Lemma~\ref{lem:fkg} below, we show that $(L^{(p)}_S+1)^{-1}$ is a bounded and decreasing function of the sequence $(X_1, X_2,\dots)$, whence $\ind\{T_S^{(p)}\geq -2p/m_{\overline \bfq}\}$ is an increasing function of the same sequence. Thus, the FKG inequality implies that the ratio in \eqref{eq:147} is bounded above by $\PP(T_S^{(p)} > -2p/m_{\overline \bfq})$, which tends to $0$ as $p\rightarrow \infty$ by Lemma \ref{lem:33}.
\end{proof}

\begin{lemma}\label{lem:fkg}
    Let $(X_1,X_2,\dots)$ be a sequence of integers taking values in $\{-1\} \cup \NN^\#$ and consider the walk $S_n = X_1+\cdots + X_n$ for $n\in \NN$. Let $T_S^{(p)}$ be the first time that the walk hits $-p$ and     $L_S^{(p)}$ be the total number of negative steps until this hitting time. Then, $T_S^{(p)}$ and $L_S^{(p)}$ are increasing functions of the sequence $(X_1,X_2,\dots)$ for every $p\in \NN$.
\end{lemma}
\begin{proof}
    Fix $p\in \NN$. Suppose $(\widetilde X_1,\widetilde X_2,\dots)$ is another sequence of integers such that $X_i \leq \widetilde X_i$ for every $i\in \NN$. Let $\widetilde S_n = \widetilde X_1 + \cdots + \widetilde X_n$ and define $T_{\widetilde S}^{(p)}$ and $L_{\widetilde S}^{(p)}$ analogously. Since $S_i \leq \widetilde S_i$ for all $i$, we have $T_S^{(p)} \leq  T_{\widetilde S}^{(p)}$ almost surely. 

    Let us show $L_S^{(p)} \leq L_{\widetilde S}^{(p)}$. It suffices to show this for the case that $X_j < \widetilde X_j$ at a unique index $j \in [\![1, T_{\widetilde S}^{(p)}]\!]$ and $X_i = \widetilde X_i$ for all other $i\leq T_{\widetilde S}^{(p)}$, since we can make comparisons while changing $(X_i)_{i\leq T_{\widetilde S}^{(p)}}$ to $(\widetilde X_i)_{i\leq T_{\widetilde S}^{(p)}}$ index-by-index. In fact, since $T_S^{(p)}=T_{\widetilde S}^{(p)}$ if $X_i = \widetilde X_i$ for every $i \in [\![1, T_S^{(p)}]\!]$, so    
    we only need to consider the case that $X_j < \widetilde X_j$ for some $j \in [\![1, T_S^{(p)}]\!]$ and $X_i = \widetilde X_i$ for all $i\neq j$ in this interval. If $X_j \neq -1$, then $X_j$ is not counted in $L_S^{(p)}$, so $T_S^{(p)} \leq T_{\widetilde S}^{(p)}$ implies $L_S^{(p)} \leq L_{\widetilde S}^{(p)}$. If $X_j = -1$ and $\widetilde X_j \geq 0$, then $\widetilde S_{T_S^{(p)}} > S_{T_S^{(p)}}=-p$, so there exists at least one negative $X_i$ among $T_S^{(p)} < i \leq T_{\widetilde S}^{(p)}$. Hence,
    \begin{equation}
        L_{\widetilde S}^{(p)} = \sum_{i=1}^{T_S^{(p)}} (\ind\{ X_i=-1\} - \ind_{i=j}) + \sum_{i=T_{S}^{(p)}+1}^{T_{\widetilde S}^{(p)}} \ind \{\widetilde X_i=-1\} \geq (L_S^{(p)}-1) + 1 = L_S^{(p)}.
    \end{equation}   
     Hence, $L_S^{(p)} \leq L_{\widetilde S}^{(p)}$ whenever $X_i\leq \widetilde X_i$ for all $i$; this completes the proof.
\end{proof}

\subsection{Perimeter cascade of supercritical planar maps}
\label{sec:cascade}
In this subsection, we investigate the law of the perimeters of faces that appear in the iterative construction of supercritical maps in Definition~\ref{def:model}. Let us first define a process indexed by the Ulam tree $\mathcal U = \bigcup_{n\in {\NN^\#}}\NN^n$ which describes these perimeters along with their genealogy. For $u\in \NN^n\subset \mathcal U$, denote $|u| = n$.

\begin{definition}
  \label{def:ulam}
  Given $p \in \NN$, let $\{(M_i,\fF_i)\}_{i\in \NN^\#}$ be the Markov chain in Definition~\ref{def:model} with the law $\PP_\infty^{(p)}$. Let $M$ be the random planar map resulting from this chain. We define the processes $\bchi_M^\inn = (\chi_M^\inn(u))_{u\in \mathcal U}$ and the indexing $f:\mathcal U \to \{\partial M,\varnothing\} \cup ( \bigcup_{i\in \NN^\#} \fF_i)$ through the following inductive procedure.
  \begin{itemize}
      \item Let $f_\root = \partial M$ and $\chi_M^\inn(\root) = p$.
      \item Suppose we have $\{(f_u,\chi_M^\inn(u))\}_{|u|= i}$. Here is how we define $\{(f_u,\chi_M^\inn(u))\}_{|u|=i+1}$.
      \begin{itemize}
      \item For each $u\in \mathcal U$ with $|u|=i$ and $f_u \neq \varnothing$, recall that $\fF(M(f_u))$ are the faces of the Boltzmann map glued into the face $f_u \in \fF_i$. With $m = |\fF(M(f_u))|$, let $f_{u1},f_{u2},\dots,f_{um}$ be an enumeration of the faces $\fF(M(f_u))$ such that $\per(f_{u1}),\dots,\per(f_{um})$ is in a non-increasing order. Let $\chi_M^\inn(uk)$ be the inner half-perimeter $\per_\inn(f_{uk})$ of the ring attached to the face $f_{uk}$.

        \item If $f_u =\varnothing$ or $k>|\fF(M(f_u))|$, let $f_{uk} = \varnothing$ and $\chi_M^\inn(uk) = 0$.
        \end{itemize}
  \end{itemize}  
\end{definition}

The main result of this subsection is that the above perimeter process $\bchi_M^\inn$ converges in distribution as $p\to\infty$ to the multiplicative cascade $Z_Q$ of the inner boundary lengths of CLE$_4$ loops in a unit boundary length supercritical LQG disk.

\begin{proposition}
  \label{prop:26}
  For each $p\in \NN$, let $\bchi_{M^{(p)}}^\inn$ denote the perimeter process of Definition~\ref{def:ulam} sampled from $\PP_\infty^{(p)}$ corresponding to the background charge $Q\in(0,2)$. Also, let $(Z_Q(u))_{u\in \mathcal U}$ be the multiplicative cascade given in Corollary~\ref{cor:supercritical-cascade}. Then, as $p\to\infty$,
    \begin{equation}
    \label{eq:86}
    \big( p^{-1}\chi_{M^{(p)}}^\inn(u) \big)_{u\in \cU} \stackrel{d}{\rightarrow} \big(Z_Q(u)\big)_{u\in \cU}
  \end{equation}
  with respect to the product topology on $\RR^\cU$.
\end{proposition}

This result is the supercritical analog of \cite[Theorem~1]{CCM20}, which gave the convergence of the perimeter process in the loop $O(n)$ model for $n\in (0,2)$ to a multiplicative cascade. Its proof is based on the following description of the law of perimeters of faces in the Boltzmann maps that constitute our planar map. Note the similarity between its law \eqref{eq:85} and the law \eqref{eq:1} of the lengths of (non-nested) CLE$_4$ loops on an independent unit boundary length critical LQG disk, which we denoted $\rho^{(1)}$. It was proved in \cite{CCM20} that, in fact, the former converges to the latter in the scaling limit.\footnote{Though \cite{CCM20} assumes $a\in (\frac{3}{2},\frac{5}{2})\setminus\{2\}$ throughout, their proof of this proposition only requires that the step distribution of the random walk is centered and is in the domain of attraction of the totally asymmetric stable law of parameter $\alpha = a-1/2 \in (1,2)$. The proof for this general case can by found in \cite[Proposition~C.2]{ADH24}.}
\begin{lemma}[{\cite[Proposition~3]{CCM20}, \cite[Proposition~C.2]{ADH24}}]
  \label{lem:13}
  Let $\bfq$ be a critical non-generic weight sequence of type $a=2$. For each $p\in \NN$, let $M_0^{(p)}$ be a Boltzmann map sampled from the law $\PP_\bfq^{(p)}$ given in Definition~\ref{def:boltzmann-map} and recall that $\bchi_{M_0^{(p)}}$ refers to the decreasing sequence of half-perimeters of internal faces of $M_0^{(p)}$. 
  Then, as $p\to\infty$, the law of the sequence $(p^{-1}\chi_{M_0^{(p)}}(i))_{i\in \NN}$ converges to $\rho^{(1)}$ with with respect to the product topology on $\RR^\NN$.
\end{lemma}

\begin{proof}[Proof of Proposition~\ref{prop:26}]

Comparing Definition~\ref{def:ulam} of the perimeter process $\bchi_{M^{(p)}}^\inn$ with the iterative construction of the map $M^{(p)}$ in Definition~\ref{def:model}, it suffices to show the convergence of the first generation: i.e.,
\begin{equation}
    \big(p^{-1} \chi_{M^{(p)}}^\inn(i)\big)_{i\in \NN} \stackrel{d}{\to} \big( Z_Q(i) \big)_{i\in \NN}
\end{equation}
as $p\to \infty$. 

Let $M_0^{(p)}$ be a Boltzmann map sampled from the law $\PP_\bfq^{(p)}$ and let $f_1,f_2,\dots,f_{|\fF(M_0^{(p)})|}$ be an enumeration of its interior faces so that $\chi_{M_0^{(p)}}(i) = \per(f_i)$ for each $i$. That is, $\per(f_1),\per(f_2),\dots$ is in a non-increasing order. For each face $f_i \in  \fF(M_0^{(p)})$, sample a ring $R_i$ conditionally independently from the law $\PP_\ring^{(\per(f_i))}$. By our choice of the ring distribution (Definition~\ref{def:ring}), the ratio $\Rat(R_i)$ of the inner and outer parameters of the ring $R_i$ converge jointly in distribution to $\exp(\beta_Q Y_i)$ where $Y_1,Y_2,\dots$ are i.i.d.\ Rademacher random variables. Hence, if we put $\per(f_i) \Rat(R_i):=0$ for $i>|\fF_{M_0^{(p)}}|$ and let $(Z(i))_{i\in \NN}$ be a sequence with the law $\rho^{(1)}$ sampled independently from $\{Y_i\}_{i\in \NN}$, then by Lemma~\ref{lem:13} we have
\begin{equation}
    \big(p^{-1} \chi_{M^{(p)}}^\inn(i)\big)_{i\in \NN} = \big( p^{-1} \per(f_i) \Rat(R_i) \big)_{i\in \NN} \stackrel{d}{\to} \big(Z(i)e^{\beta_Q Y_i}\big)_{i\in \NN} = \big(Z_Q(i)\big)_{n\in \NN}
\end{equation}
as $p\to\infty$ with respect to the product topology on $\RR^{\NN}$.
\end{proof}

\subsection{Convergence of subcritical Boltzmann maps to the continuum random tree}
\label{sec:crt}

We now recall the definitions of the continuum random tree (CRT) and the Gromov--Hausdorff distance that appear in our main result. Then, we survey the literature on convergence of subcritical Boltzmann maps to the CRT, which shall provide key inputs to our proof of Theorem~\ref{thm:1}.

\subsubsection{Continuum random tree}

The Brownian CRT is a random real tree defined from a Brownian excursion that arises as the scaling limit of a large class, as defined and investigated by Aldous in the pioneering works \cite{Ald91, Ald91+, Ald93}. More recent results on the convergence of random discrete structures to the CRT are surveyed in \cite{Stu20}.
\begin{definition}
  Let $\bbe\colon [0,1]\rightarrow \RR_+$ be the normalized Brownian excursion. Consider the pseudo-distance defined on the interval $[0,1]$ by
  \begin{equation}
    \label{eq:97}
    d_{\bbe}(s,t)=\bbe_s+\bbe_t-2\min_{s\wedge t\leq u \leq s\vee t}\bbe_u
  \end{equation}
  for $s,t \in [0,1]$. Let $\sim_{\bbe}$ be an equivalence relation on $[0,1]$ given by $s\sim_{\bbe} t$ if and only if $d_{\bbe}(s,t) = 0$.
  The \textbf{continuum random tree} is the random metric space $(\cT_{\bbe},d_{\cT_{\bbe}}):= ([0,1],d_{\bbe})/\!\sim_{\bbe}$. The equivalence class containing 0 and 1 is the \textbf{root} $\root$ of the CRT $(\cT_{\mathbb e},d_{\mathbb e})$. With an abuse of notation, we also use the acronym CRT to denote the law of the above object.
\end{definition}

\subsubsection{Gromov--Hausdorff convergence} The Gromov--Hausdorff distance measures how close two metric spaces are being isometric to each other and has been used widely to describe convergence of random discrete structures to continuum metric spaces.

\begin{definition}
Let $\mathcal X = (X,d_X)$ and $\mathcal Y=(Y,d_Y)$ be compact metric spaces. The following are equivalent definitions for the \textbf{Gromov--Hausdorff distance} $d_{\mathrm{GH}}(\mathcal X, \mathcal Y)$ between $\mathcal{X}$ and $\mathcal{Y}$. 
\begin{enumerate}[(i)]
\item Given two compact sets $K_1,K_2$ in a metric space $\mathcal Z = (Z,d)$, recall that their \textbf{Hausdorff distance} is given by 
\begin{equation}
    d_{\mathrm H}(K_1,K_2) = \bigg(\sup_{z_1\in K_1} d(z_1, K_2) \bigg) \vee \bigg( \sup_{z_2\in K_2} d(z_2,K_1) \bigg).
\end{equation}
Then, 
\begin{equation}
    d_{\mathrm{GH}}(\mathcal X,\mathcal Y) = \inf_{\mathcal Z, \varphi_{ X},\varphi_{ Y}} d_{\mathrm H}( \varphi_X(X),\varphi_Y(Y))
\end{equation}
where the infimum is over all metric spaces $\mathcal Z= (Z,d)$ and isometric embeddings $\varphi_X:\mathcal X\to \mathcal Z$, $\varphi_Y:\mathcal Y \to \mathcal Z$.
\item A \textbf{correspondence} between $\mathcal X$ and $\mathcal Y$ is a subset $\mathcal R\subset X\times Y$ such that $\mathcal R \cap (\{x\}\times Y) \neq \varnothing$ for every $x\in X$ and $\mathcal R \cap (X \times \{y\}) \neq \varnothing$ for every $y\in Y$. The \textbf{distortion} of the correspondence $R$ is defined as
\begin{equation}
    \dis(\mathcal R) = \sup\{ |d_X(x_1,x_2) - d_Y(y_1,y_2)| : (x_1,y_1),(x_2,y_2)\in \mathcal R\}. 
\end{equation}
Then,
\begin{equation}
    d_{\mathrm{GH}}(\mathcal X,\mathcal Y) = \frac{1}{2}\inf_{\mathcal R} \dis(\mathcal R)
\end{equation}
where the infimum is taken over all correspondences $\mathcal R$ of $\mathcal X$ and $\mathcal Y$.
\end{enumerate}
\end{definition}

See, e.g., \cite[Section~7.3]{BBI01} for the equivalence of the two definitions. We shall often switch between the above two equivalent formulations depending on which one is more convenient for the particular application at hand. 
\subsubsection{Previous convergence results for subcritical maps}
As mentioned earlier, the behavior of Boltzmann maps sampled from $\PP_{\overline \bfq}^{(p)}$ depends strongly on the choice of the weight sequence $\overline \bfq$. If the weight sequence is critical non-generic, the corresponding Boltzmann maps have macroscopic faces, and upon renormalizing distances appropriately, these Boltzmann maps are expected to converge to a ``stable map with a boundary" as in \cite{LM11}. On the other hand, in the critical generic case, there are no macroscopic faces and the scaling limit is known to be the Brownian map \cite{MM07,LeGal11,Mar18}.

If the weight sequence is subcritical, there are again no macroscopic faces in the Boltzmann map, but in this case, the scaling limit turns out to be a $\CRT$. This was proved in \cite{JS15,Mar22} for the case of Boltzmann maps without boundary. We will need the analogous convergence for subcritical Boltzmann maps with boundary (i.e., sampled from $\PP_{\overline \bfq}^{(p)})$ in two different flavors.
First, we state a result from \cite[Corollary 5]{KR20} about the convergence of the outer boundary of the Boltzmann map to the $\CRT$. Recall that for a planar map $M$ with vertex set $V_M$ and graph distance $d_M$, we let $rM$ refer to the metric space $(V_M, rd_M)$ for $r\in \RR_+$.
\begin{proposition}
   \label{prop:4}
   Let $\overline \bfq$ be a subcritical weight sequence. For each $p\in \NN$, let $M_0^{(p)}$ be a Boltzmann map sampled from $\PP^{(p)}_{\overline \bfq}$. Then, there exists a constant $K_{\overline \bfq}$ such that 
   \begin{equation}
     \frac{K_{\overline \bfq}}{\sqrt{p}}\partial M_0^{(p)} \stackrel{d}{\rightarrow} \CRT
   \end{equation}
   as $p\to\infty$ with respect to the Gromov--Hausdorff distance.
 \end{proposition}

    The proof of this result in \cite{KR20} is based on a bijection between the outer boundary $\partial M_0^{(p)}$ to a looptree associated with a critical Bienaym\'e--Galton--Watson process, which we follow closely for our proof of Theorem~\ref{thm:1}. See Section~\ref{sec:looptree} for further details.

 With additional assumptions on the decay rate of $\mu_{\overline \bfq}$, the entire map $M^{(p)}$ converges to a $\CRT$.
 \begin{proposition}
  \label{prop:6}
   Let $\overline \bfq$ be a subcritical  weight sequence such that $\mu_{\overline \bfq}([k,\infty))=o(k^{-1})$ as $k\to\infty$. If $M^{(p)}$ has the law $\PP^{(p)}_{\overline \bfq}$ for each $p\in \NN$, then
  \begin{equation}\label{eq:180}
    \frac{1}{\sqrt{2p}} M^{(p)} \stackrel{d}{\rightarrow} \CRT
  \end{equation}
  as $p\to\infty$ with respect to the Gromov--Hausdorff distance. 
  Moreover, 
  \begin{equation}\label{eq:181}
      \frac{1}{\sqrt p}\max_{v \in V_{M^{(p)}}} d_{M^{(p)}}(v, \partial M^{(p)}) \stackrel{d}{\to} 0.
  \end{equation}
\end{proposition}
\begin{proof}
For the proof of \eqref{eq:180}, in view of \cite[Theorem~1.3]{Mar22}, it suffices to show that 
  \begin{equation}
    \label{eq:131}
    \frac{1}{p^2}\sum_{f\in \fF(M^{(p)})}[\per(f)]^2 \to 0
  \end{equation}
  in distribution as $p\to\infty$.
  We use the connection between random walks and Boltzmann maps as explained in Section \ref{sec:perproc}. Let $S_n=X_1+\dots +X_n$ be a random walk with step distribution \eqref{eq:step-dist}.   By Proposition \ref{prop:10}, the convergence \eqref{eq:131} is equivalent to
  \begin{equation}
    \label{eq:129}
    \lim_{p\to\infty} \frac{\EE\Big[\frac{1}{L^{(p)}_S+1}\ind\big\{\sum_{i=1}^{T_S^{(p)}}(X_i+1)^2>\epsilon p^{2}\big\}\Big]}{\EE\Big[\frac{1}{L^{(p)}_S+1}\Big]}= 0 \qquad \text{for all } \epsilon>0.
  \end{equation}
  By Lemma~\ref{lem:fkg}, the indicator $\ind\{\sum_{i=1}^{T_S^{(p)}}(X_i+1)^2>\epsilon p^2\}$ is an increasing function of the sequence $(X_1,X_2,\dots)$ while $(L^{(p)}_S+1)^{-1}$ is a decreasing function of the same sequence. Using the FKG inequality, the problem at hand reduces to checking
  \begin{equation}
    \label{eq:125}
    \PP\bigg(\sum_{i=1}^{T_S^{(p)}}(X_i+1)^2>\epsilon p^2\bigg)\rightarrow 0.
  \end{equation}
   Since $\overline \bfq$ is subcritical, in view of Lemma \ref{lem:33}, it suffices to show that
  \begin{equation}
    \label{eq:130}
     \PP\bigg(\sum_{i=1}^{\lfloor-2p/m_{\overline\bfq}\rfloor}(X_i+1)^2>\epsilon p^2\bigg)\rightarrow 0.
  \end{equation}
  This follows by the law of large numbers for i.i.d.\ heavy tailed random variables (e.g., \cite[Theorem~6.17]{Kal21} with $p=1/2$ therein).

  In fact, the convergence \eqref{eq:180} holds because the boundary of $(2p)^{-1/2}M^{(p)}$ converges to the CRT, whereas its interior faces disappear under the scaling limit as expressed in \eqref{eq:181}. This reasoning, first given heuristically in \cite{Bet15}, was stated within the proof of \cite{Mar22} in terms of the ``key bijection" between (a uniformly chosen negative pointed map version of) $M^{(p)}$ and a labelled forest combining those of \cite{BDG04} and \cite{JS15}. In this bijection, the leaves of the forest are in bijection with the non-distinguished vertices of $M^{(p)}$. Moreover, we can check directly from its construction (see \cite[Section~2.3]{Mar18-alea}) that each tree in the forest contains a leaf that is mapped to a vertex on $\partial M^{(p)}$. Since the labels on the leaves correspond to the graph distance to the distinguished vertex in $M^{(p)}$, the maximum distance from any vertex in $M^{(p)}$ to $\partial M^{(p)}$ is bounded above by the maximum difference of the labels in each tree. The latter quantity is equal to $\max \widetilde L^{(p)} - \min \widetilde L^{(p)}$ in the notation of \cite{Mar22}. By Theorem~2.6(1) in that paper, it follows from \eqref{eq:131} that $p^{-1/2}(\max \widetilde L^{(p)} - \min \widetilde L^{(p)}) \to 0$ in distribution as $p\to\infty$. This proves \eqref{eq:181}.
\end{proof}

\section{The probability that a supercritical map is finite}
\label{sec:prob-dis}
The goal of this section is to prove Proposition \ref{thm:2}, which will be done via a subadditivity argument. A priori, it is not even clear if $F(p) = \PP_\infty^{(p)}(M \text{ is finite})<1$ for any value of $p$, so we turn to this first. We saw in Proposition~\ref{prop:many-macroscopic-loops} that in the continuum limit, macroscopic loops are abundant at all stages in the supercritical multiplicative cascade $Z_Q$.
We now employ a comparison to a supercritical Bienaym\'e--Galton--Watson (BGW) tree to deduce the non-triviality of $F$.
\begin{lemma}
  \label{lem:7}
  We have the convergence $F(p)\rightarrow 0$ as $p\rightarrow \infty$.
\end{lemma}
\begin{proof}
For each positive integer $p$, let $n_q = n_q(p)$ be a constant to be determined later.
Let $M$ be a map sampled from $\PP_\ttF^{(p)}$ and recall the associated perimeter cascade $\bchi_M^\inn = (\chi_M^\inn(u))_{u\in \cU}$ from Definition~\ref{def:ulam}. Consider the random variable
  \begin{equation}
    \label{eq:10}
    N_{m}= \#\{u\in \cU: |u|=mn_q, \chi_M^\inn(u)>p\}.
  \end{equation}
Then,
\begin{equation}
    F(p)\leq \PP_\infty^{(p)}(N_m=0 \textrm{ for some } m\in \NN).
\end{equation}
The goal now is to show that the  right-hand side of this inequality tends to 0 as $p\to \infty$. We do so by constructing a supercritical $\BGW$ process whose size at each generation $m$ is stochastically dominated by $N_m$ and has a sufficiently small extinction probability.

  As a consequence of Proposition~\ref{prop:many-macroscopic-loops}, for each $q\in (0,1)$, we can find $n_q\in \NN$ such that
  \begin{equation}
    \label{eq:8}
    \PP( \#\{u\in \cU: |u|=n_q, Z_Q(u)>2\}\geq 100) \geq q.
  \end{equation}
 Combining this with the distributional convergence of Proposition~\ref{prop:26}, we see that for each sufficiently large $p\in\NN$, we can choose $q = q(p)\in (0,1)$ with $q$ increasing to 1 as $p\to \infty$ such that for every integer $\tilde p \geq p$, we have
  \begin{equation}
    \label{eq:9}
    \PP_\infty^{(\tilde p)}( \#\{u\in \cU: |u|=n_q, \chi^\inn_M(u)>p\}\geq 100) \geq q.
  \end{equation}
  Now, consider a $\BGW$ process with offspring distribution $\mu$ such that $\mu(100)=q$ and $\mu(0)=1-q$. Let $Z_m$ denote the number of offspring of this $\BGW$ process in generation $m$ and define
  \begin{equation}\begin{split}
  \widetilde N_m = \#\{u\in\cU: |u|= mn_q, \chi^\inn_M(u_k)> p\; & \text{for $1\leq k\leq m$}\\& \text{where $u_k$ is the ancestor of $u$ with $|u_k|=kn_q$}\}.\end{split}
  \end{equation}
  Then, $(\widetilde N_1, \widetilde N_2,\dots)$ sampled from $\PP_\infty^{(p)}$ is a BGW process whose offspring distribution $\tilde \mu$ satisfies $\tilde \mu(k) \geq \mu(k)$ for all $k>0$, whence $(Z_1,Z_2,\dots)$ is stochastically dominated by $(\widetilde N_1, \widetilde N_2,\dots)$. Since $\widetilde N_m \leq N_m$ for every $m$, 
   \begin{equation}
     \label{eq:11}
     (N_1,N_2,\dots)\geq_{\mathrm{SD}} (Z_1,Z_2,\dots).
   \end{equation}
  On the other hand, for the corresponding supercritical $\BGW$ process, we know that the extinction probability $\PP(Z_m = 0 \text{ for some }m\in\NN)$ tends to 0 as $q\to 1$. 
  By stochastic domination, we conclude $\PP_\infty^{(p)}(N_m = 0 \text{ for some $m$})\to 0$ as $p\to \infty$ as claimed.
\end{proof}
\begin{lemma}
  \label{lem:1}
  For every $p\in \NN$, we have $0<F(p)<1$.
\end{lemma}
\begin{proof}
    Fix $p\in \NN$ and let $M$ be a map sampled from $\PP_\infty^{(p)}$. Let us first show $F(p)>0$. The construction of the map $M$ in Definition~\ref{def:model} terminates right away if the gasket $M_0$ has no interior faces. Since $M_0$ has the marginal law of a Boltzmann map sampled from $\PP_\bfq^{(p)}$, in terms of the random walk representation of the perimeter process $\chi_{M_0}$ (Proposition~\ref{prop:10}), this corresponds to the event that $X_i=-1$ for all $i\in [\![1,p]\!]$.
Using the trivial bound $L_S^{(p)}\geq p$, we thus have
\begin{equation}
  \label{eq:133}
  F(p)\geq \frac{\EE\big[(L_S^{(p)}+1)^{-1}\ind\{X_i=-1\;\; \forall i\in [\![1,p]\!]\}\big]}{\EE\big[(L_S^{(p)}+1)^{-1}\big]}=\frac{(p+1)^{-1}(\mu_{\bfq}(-1))^p}{\EE\big[(L_S^{(p)}+1)^{-1}\big]}\geq (\mu_{\bfq}(-1))^p = (Z_\bfq)^{-p}>0.
\end{equation}
    
    We now prove $F(p)<1$. Choose a face $f_1 \in \fF_0 = \fF(M_0)$ with the longest perimeter: that is, $\per(f_1)= \max_{f\in \fF_0}\per(f)=\chi_{M_0}(1)$. Note that for each $k\in \NN$, we have $\PP_\bfq^{(p)}(\per(f_1)>k) = \mu_{\bfq}([k,\infty))>0$ by the definition of the non-generic critical weight sequence $\bfq$. 
  We now fix $k$ to be large enough that $\sup_{m> k}F(m)<1/2$.
  Then,
  \begin{equation}
    \label{eq:12}\begin{split}
    F(p)\leq \EE_\infty^{(p)}\big[F(\per_\inn(f_1))\big]
    &\leq  \PP_\infty^{(p)}(\per_{\inn}(f_1)\leq k)+\PP_\infty^{(p)}(\per_{\inn}(f_1)>k)\sup_{m>k}F(m)\\&\leq 1-\frac{1}{2}\PP_\infty^{(p)}(\per_{\inn}(f_1)>k).
  \end{split}\end{equation}
  Recall from Definition \ref{def:ring} that $\Rat(R)$ for a ring $R$ sampled from $\PP_\ring^{(m)}$ converges in distribution to $\exp(\beta_Q Y)$ where $Y$ is a Rademacher random variable. Hence, we may assume (by choosing a larger $k$ if necessary) that $\inf_{m> k}\PP_\ring^{(m)}(\Rat(R)\geq 1) >0$.
  Since $\per_\inn(f_1)$ is the inner half-perimeter of the ring sampled from $\PP_\ring^{(\per(f_1))}$, we have
  \begin{equation}\label{eq:135}
    \PP_\infty^{(p)}(\per_\inn(f_1)>k) \geq \PP_{\bfq}^{(p)}(\chi_{M_0}(1)>k) \cdot \inf_{m> k}\PP_\ring^{(m)}(\Rat(R)\geq 1)>0.
  \end{equation}
  This combined with \eqref{eq:12} completes the proof.
\end{proof}

\begin{proof}[Proof of Proposition \ref{thm:2}]
  We first show the existence of $\alpha\in [0,\infty)$ such that $-p^{-1}\log F(p) \to \alpha$ as $p\rightarrow \infty$.
By a standard subadditivity argument (see \cite[Theorem~23]{DE52}), it suffices to find a rational function $g(p,q)$ and $p_0>0$ such that if $p,q\geq p_0$, then
  \begin{equation}
    \label{eq:2}
F(p+q)\geq g(p,q)F(p)F(q).
\end{equation}
 Note that $\alpha = \inf_p (-p^{-1} \log F(p))<\infty$ from Lemma~\ref{lem:1}.

By the construction of supercritical planar maps from Definition \ref{def:model}, we have the recursive relation
\begin{equation}
  \label{eq:119}
  F(p+q) = \EE_\infty^{(p+q)}\bigg[\prod_{f\in \fF_0 }F(\per_\inn(f))\bigg] = \EE_\bfq^{(p+q)}\bigg[\prod_{i\in \NN }\EE_\ring^{(\chi_{M_0}(i))} \big[F(\chi_{M_0}(i)\Rat(R))\big]\bigg].
\end{equation}
Now, we transfer the above expression to one involving the corresponding random walk in Definition~\ref{def:4}. Let $\{X_i+1\}_{i\in\NN}$ be i.i.d.\ random variables with the law $\mu_\bfq$ and let $S_n = X_1 + \cdots + X_n$ be the corresponding random walk. Given the walk $S$, sample $R_i$ conditionally independently for each $i\in \NN$ from the marginal law of $\Rat(R)$ under $\PP_\ring^{(X_i+1)}$.
By an application of Proposition \ref{prop:10} along with \eqref{eq:120}, we can rewrite \eqref{eq:119} as
\begin{equation}
  \label{eq:4}
  F(p+q)=\frac{\EE\big[(L_S^{(p+q)}+1)^{-1}\prod_{i=1}^{T_S^{(p+q)}}F((X_i+1)R_i)\big]}{\EE\big[(L_S^{(p+q)}+1)^{-1}\big]}.
\end{equation}
We now consider the modified walk $\hat S_j:=S_{j+T_{S}^{(p)}}+p$, which is the part of the walk $S$ from the first hitting time of $-p$ up to the hitting time of $-(p+q)$, translated upwards by $p$. Also denote $\hat X_i=X_{i+T_S^{(p)}}$ and $\hat R_i=R_{i+T_S^{(p)}}$ for simplicity. Then, since the only possible negative steps of $S$ are of unit size, $T^{(p+q)}_S=T^{(p)}_S+T^{(q)}_{\hat S}$ and $L^{(p+q)}_S=L^{(p)}_S+L^{(q)}_{\hat S}$.
Moreover, rearranging $(X_{T_S^{(p)}+1}+1,\dots ,X_{T_S^{(p+q)}}+1)$ in a decreasing order and then padding zeroes to its end  yields $\bchi^{(q)}_{\hat S}$, which is independent from $\bchi_S^{(p)}$ by the strong Markov property of the walk $S$.
  As a consequence, from \eqref{eq:4}, we obtain that for some constant $C_1$, 
\begin{equation}
  \label{eq:5}\begin{split}
  F(p+q)
        &\geq \frac{\EE\big[(L^{(p)}_S+1)^{-1}\prod_{i=1}^{T_S^{(p)}}F((X_i+1)R_i) \cdot (L^{(q)}_{\hat S}+1)^{-1}\prod_{j=1}^{T_{\hat S}^{(q)}}F((\hat X_j+1)\hat R_j)\big]}{\EE\big[(L_S^{(p+q)}+1)^{-1}\big]}\\
        &= \frac{\EE\big[(L^{(p)}_S+1)^{-1}\prod_{i=1}^{T_S^{(p)}}F((X_i+1)R_i) \big]\cdot \EE\big[(L^{(q)}_{\hat S}+1)^{-1}\prod_{j=1}^{T_{\hat S}^{(q)}}F((\hat X_j+1)\hat R_j)\big]}{\EE\big[(L_S^{(p+q)}+1)^{-1}\big]}\\
  &=\frac{\EE\big[(L_S^{(p)}+1)^{-1}\big]\EE\big[(L_{\hat S}^{(q)}+1)^{-1}\big]}{\EE\big[(L_S^{(p+q)}+1)^{-1}\big]}F(p)F(q)
  \geq \frac{C_1(p+q+1)}{p^{3/2}q^{3/2}}F(p)F(q)
\end{split}\end{equation}
where, in the last line, we used $L_S^{(p+q)}\geq p+q$ along with Lemma~\ref{lem:215}. We have thus shown \eqref{eq:2} with $g(p,q)=(p+q+1)/(pq)^2$.

We now show that, in fact, $\alpha>0$. Starting again from from \eqref{eq:4}, since $T_S^{(p)} \geq L_S^{(p)}\geq p$, we have 
  \begin{equation}
    \label{eq:7}\begin{split}
    F(p)&= \frac{\EE\big[(L_S^{(p)}+1)^{-1}\prod_{i=1}^{T_S^{(p)}}F((X_i+1)R_i))\big]}{\EE\big[(L_S^{(p)}+1)^{-1}\big]} \leq \frac{\EE\big[\prod_{i=1}^p F((X_i+1)R_i))\big]}{\EE\big[(L_S^{(p)}+1)^{-1}\big]}
    \\&\leq Cp^{3/2}\big(\EE[F((X_1+1)R_1)]\big)^p
  \end{split}\end{equation}
  for some constant $C$; for the second inequality, we used Lemma~\ref{lem:215} and that 
   $\{(X_i+1)R_i\}_{i\in [\![1,p]\!]}$ are i.i.d.
  Since $\mu_\bfq([k,\infty))>0$ for all $k$ and $\PP_\ring^{(m)}(\Rat(R)\geq 1) \to 1/2$ as $m\to\infty$, we have $\PP((X_1+1)R_1>0)>0$. Thus,
  $\EE[F((X_1+1)R_1)]<1$ by Lemma~\ref{lem:1}. This completes the proof.
\end{proof}

\section{Estimates on the size of a supercritical map conditioned to be finite}
\label{sec:estimates}
The broad intuition of this section is that supercritical LQG conditioned to be finite should in some sense look ``subcritical," similar to what happens for $\BGW$ trees. 

Recall from Definition~\ref{def:model} the Markov chain $(M_i, \fF_i)$ with the law $\PP_\infty^{(p)}$, where $M_i$ are the partial maps which increase to the whole map $M$ and $\fF_i$ are the set of ``active" faces inside with the next layer of rings and critical Boltzmann maps are inserted.
The law $\PP_\ttF^{(p)}$ on finite planar maps appearing in Theorem \ref{thm:1} can then be naturally considered as the marginal law on $M$ when we condition the chain $(M_i,\fF_i)_{i\in\NN^\#}$ to terminate at a finite time. That is, there exists $n\in\NN$ such that $M_i = M$ and $\fF_i = \varnothing$ for all $i\geq n$. We thus use $\PP_\ttF^{(p)}$ to also denote this conditional law of the chain terminated at a finite time.

The plan is to first prove that the marginal law of the outermost gasket $M_0$ for a map $M$ sampled from $\PP_\ttF^{(p)}$ is that of a subcritical Boltzmann map (Proposition~\ref{prop:2}). From here, we extrapolate that the expected total sum of the inner perimeter lengths decreases exponentially (Corollary~\ref{lem:21}). Using this estimate, we conclude that the map $M$ sampled from $\PP_\ttF^{(p)}$ looks essentially like its outermost gasket $M_0$ in the sense that the volumes of submaps $M|_f$ within the faces $f\in \fF_0 = \fF(M_0)$ grow slower than $p^{1/2}$ (Proposition~\ref{lem:23}).

\subsection{Law of the gasket sampled from \texorpdfstring{$\PP_\ttF^{(p)}$}{the measure on planar maps conditioned to be finite}}

We first establish that the outermost gasket $M_0$ for a map $M$ sampled from $\PP_\ttF^{(p)}$ is a subcritical Boltzmann map. Recall the notation $\per_{\inn}(f)$ for the inner half-perimeter of the ring attached to each face $f\in \fF_0 = \fF(M_0)$. Here is an explicit description of the corresponding weight sequence, which is subcritical.

\begin{lemma}
  \label{lem:9}
  Let $\bfq$ be the critical non-generic weight sequence with $a=2$ as in Definition~\ref{def:nongen}. Consider the weight sequence $\bfq'$ given by
  \begin{equation}
    \label{eq:87}
    q'_k= \EE_\ring^{(k)}[ F(\per_\inn(f)) ]\cdot q_k
  \end{equation}
where $\per_\inn(f)$ is the half-perimeter of the inner face of the ring sampled from the distribution $\PP_\ring^{(k)}$ in Definition \ref{def:ring}. (In particular, the outer half-perimeter of the ring is $k$.) Then, the weight sequence $\bfq'$ is admissible and subcritical.
\end{lemma}
\begin{proof}
Let 
\begin{equation}
    g_\bfq(x) := 1 - x + xf_\bfq(x) = 1 - x + \sum_{k=1}^\infty \binom{2k-1}{k-1} q_k x^{k}
\end{equation}
and define $g_{\bfq'}$ analogously. Recall from Definition~\ref{def:classification} that $\bfq$ is admissible if and only if $f_\bfq(x) = 1-1/x$ has a positive solution and the smallest such solution is denoted $Z_\bfq$. Then, $g_\bfq(Z_\bfq)=0$.
For every $k\in \NN$, since $\PP_\ring^{(k)}(\per_\inn(f) > 0)>0$ by the non-triviality condition in Definition~\ref{def:ring}, we have $q_k' < q_k$ by Lemma~\ref{lem:1} and thus $g_{\bfq'}(Z_\bfq) < g_\bfq(Z_\bfq)=0$. Since $g_{\bfq'}(0)=1$, we conclude that $\bfq'$ is admissible and $Z_{\bfq'}$, the smallest positive solution of $f_{\bfq'}(x) = 1-1/x$, is less than $Z_\bfq$.

Note from \eqref{eq:81} that the mean of the measure $\mu_{\bfq'}$ is given by $1+(g_{\bfq'})'(Z_{\bfq'})$. Since $0\leq q_k'<q_k$ for all $k$, we have $(g_{\bfq'})'(Z_{\bfq'}) \leq (g_{\bfq'})'(Z_\bfq) < (g_\bfq)'(Z_\bfq)$. Since $\bfq$ is critical, $1+(g_{\bfq'})'(Z_{\bfq'}) < 1+(g_{\bfq})'(Z_{\bfq}) = 1$ and thus $\bfq'$ is subcritical.
\end{proof}

\begin{proposition}
  \label{prop:2}
  For every $p\in \NN$, the marginal law of the outermost gasket $M_0$ under $\PP_\ttF^{(p)}$ is equal to $\PP_{\bfq'}^{(p)}$.
\end{proposition}
\begin{proof}
Consider the chain $(M_i,\fF_i)$ sampled from $\PP_\infty^{(p)}$. By construction, the random variables $\{\per_\inn(f)\}_{f\in \fF_{0}}$ are conditionally independent given the outermost gasket $M_0$. Now, we know that $\PP_\infty^{(p)}(M \textrm{ is finite} \,\lvert\, M_0, \{\per_\inn(f)\}_{f\in \fF_0} ) = \prod_{f\in \fF_0}F(\per_\inn(f))$. Hence, for any measurable set $A$ on the space of finite planar maps with outer boundary length $p$,
  \begin{equation}
    \label{eq:13}\begin{split}
    \PP_\ttF^{(p)}(M_0\in A) &= \frac{\PP_\infty^{(p)}(M_0\in A,\, M \textrm{ is finite} )}{\PP_\infty^{(p)}(M \textrm{ is finite} )}\\    &=\frac{\EE_\infty^{(p)}\big[\ind\{M_0\in A\}\cdot \prod_{f\in \fF_0}F(\per_\inn(f))\big]}{\EE_\infty^{(p)}\big[\prod_{f\in \fF_{0}}F(\per_\inn(f))\big]}\\
    &=\frac{\EE_{\bfq}^{(p)}[\ind\{M_0 \in A\} \cdot \prod_{f\in \fF_0}\EE_\ring^{(\per(f))}[F(\per_\inn(f))]]}{\EE_{\bfq}^{(p)}[\prod_{f\in \fF}\EE[F(\per_\inn(f))]]}\\
    &= \PP_{\bfq'}^{(p)}(M_0 \in A)
        \end{split}
  \end{equation}
  as claimed.
\end{proof}

We now show that conditioning $M$ to be finite leads to the removal of all macroscopic loops.

\begin{lemma} 
  \label{lem:11}
    There exists a constant $\Delta>2$ such that for any fixed $\epsilon>0$, we have 
    \begin{equation}
    \label{eq:136}
    \lim_{p\to\infty} \PP_{\ttF}^{(p)}\bigg( \sup_{f\in \fF_{0}}\per_\inn(f)  \leq \epsilon p^{1/\Delta}  \bigg) = 1.
  \end{equation}
\end{lemma}

The following lemma is needed for the proof of Lemma~\ref{lem:11}.
\begin{lemma}
  \label{lem:15}
  Let $\bfq'$ be the weight sequence defined by \eqref{eq:87}. Then, there exists constants $c>0$ such that $\mu_{\bfq'}([k,\infty))= O(e^{-ck})$ as $k\to\infty$.
\end{lemma}
\begin{proof}
  As we saw in the proof of Lemma~\ref{lem:9}, 
  $Z_{\bfq'}< Z_{\bfq}$. Hence,
  \begin{equation}\begin{split}
      \mu_{\bfq'}(k) = (Z_{\bfq'})^{k-1} \binom{2k-1}{k-1} q_{k}'  &\leq (Z_{\bfq})^{k-1} \binom{2k-1}{k-1} \,\EE_\ring^{(k)}[F(\per_\inn(f))] \,q_{k} \\&=  \EE_\ring^{(k)}[F(k\, \Rat(R))] \,\mu_{\bfq}(k) \end{split}
  \end{equation}
  where $\per_\inn(f)$ is the inner half-perimeter of the ring $R$ sampled from $\PP_\ring^{(k)}$.
  
  The lower tail condition in Definition~\ref{def:ring} implies that there exist constants $\epsilon,c >0$ satisfying
  \begin{equation}
  \label{eq:rat}
  \PP_\ring^{(k)}(\Rat(R)< \epsilon)= O(e^{-ck})
  \end{equation}
  as $k\to\infty$.
  Hence, by the exponential decay of $F(p)$ proved in Proposition~\ref{thm:2}, choosing a smaller constant $c$ if necessary, we have that for all $k$ large enough,
    \begin{equation}\label{eq:F(kratk)}
        \EE_\ring^{(k)}[F(k\,\Rat(R))] \leq \sup_{p\geq k \epsilon} F(p) +\PP_\ring^{(k)}(\Rat(R)< \epsilon ) = O(e^{-ck}).
    \end{equation}
    Here, we used the trivial bound $F\leq 1$.
  We conclude that as $k\to \infty$,
  \begin{equation}
    \label{eq:16}
    \mu_{\bfq'}([k,\infty)]) \leq \mu_\bfq([k,\infty)) \cdot \sup_{p\geq k}\EE_\ring^{(p)}[F(p\, \Rat(R))]
    = O(e^{-ck})
\end{equation}
    since $\mu_\bfq(k)$ decays polynomially in $k$ by Definition~\ref{def:nongen} of a non-generic critical weight sequence.
\end{proof}

We now use above results to bound $\sup_{f\in \fF_0}\per_\inn(f)$ as claimed.

\begin{proof}[Proof of Lemma~\ref{lem:11}]
  Let $\{X_i+1\}_{i\in\NN}$ be i.i.d.\ random variables with the law $\mu_{\bfq'}$ and, given these, sample $R_i$ conditionally independently for each $i\in \NN$ from the marginal law of $\Rat(R)$ under $\PP_\ring^{(X_i+1)}$. 
  By Propositions~\ref{prop:10} and \ref{prop:2}, if $L_S^{(p)}$ is the number of negative steps in the random walk $S_n = X_1 + \cdots + X_n$ up to the first time $T_S^{(p)}$ that this walk visits $-p$, then
  \begin{equation}
    \label{eq:18}
    \PP_{\ttF}^{(p)}\bigg(\sup_{f\in \fF_0} \per_\inn(f)> \epsilon p^{1/\Delta} \bigg)=\frac{\EE\big[(L_S^{(p)}+1)^{-1}\ind\big\{\sup_{i\in [\![1,T^{(p)}_S]\!]}(X_i+1)R_i>\epsilon p^{1/\Delta}\big\}\big]}{\EE\big[(L_S^{(p)}+1)^{-1}\big]}.
  \end{equation}
    Using the trivial bound $(L_S^{(p)}+1)^{-1}\leq (p+1)^{-1}$ and the estimate $1/\EE[(L_S^{(p)}+1)^{-1}] \leq p\,\EE[T_S^{(1)}]+1$ from Lemma \ref{lem:33}, we can find a constant $C>0$ such that for all $p$,
  \begin{equation}
      \frac{\EE\big[(L_S^{(p)}+1)^{-1}\ind\big\{\sup_{i\in [\![1,T^{(p)}_S]\!]}(X_i+1)R_i>\epsilon p^{1/\Delta}\big\}\big]}{\EE[(L_S^{(p)}+1)^{-1}]} \leq C \,\PP\bigg(\sup_{i\in [\![1,T^{(p)}_S]\!]}(X_i+1)R_i>\epsilon p^{1/\Delta}\bigg) .
  \end{equation}
  By Lemma \ref{lem:38}, it suffices find $\Delta>2$ such that
  \begin{equation}
    \label{eq:149}
    \lim_{p\to\infty} \PP\bigg(\sup_{i\in [\![1,-2p/m_{\bfq'}]\!]}(X_i+1)R_i>\epsilon p^{1/\Delta}\bigg) = 0
  \end{equation}
  to complete the proof. We can pick $\Delta = 2 + \delta/3$ where $\delta>0$ is the constant appearing in the upper bound condition of Definition~\ref{def:ring}, since then, for sufficiently large $p$,
  \begin{equation}
    \label{eq:150}\begin{split}
    \PP((X_1+1)R_1> \epsilon p^{1/\Delta}) &\leq  \PP(X_1+1> {\epsilon} (\log p)^2)+\PP(R_1>  p^{1/(2+\delta/2)}) \\&\leq \mu_{\bfq'}(({\epsilon}(\log p)^2-1,\infty)) + \sup_{k\in \NN} \PP_\ring^{(k)}(\Rat(R)> p^{1/(2+\delta/2)}) = o(p^{-1}) \end{split}
  \end{equation}
  by applying Lemma~\ref{lem:15} to the first term and the upper tail condition of Definition~\ref{def:ring} to the second term. Since $(X_i+1)R_i$ are identically distributed, taking the union bound over $i\in [\![1,-2p/m_{\bfq'}]\!]$ gives \eqref{eq:149}.
\end{proof}

\begin{remark}
  \label{rem:gen}
Lemma \ref{lem:11} can be improved if we impose further restrictions on the distribution of rings than in Definition~\ref{def:ring}. For instance, the random planar model for supercritical LQG disk suggested in \cite[Section~3.2]{AG23} has the property that the inner half-parameter of a ring sampled from $\PP_\ring^{(k)}$ is exactly equal to $\lfloor k\exp(\pm \beta_Q)\rfloor$ with probability 1/2. 
In this case, we can replace $\epsilon p^{1/\Delta}$ with $\epsilon(\log p)^{2}$ in \eqref{eq:136} since $\sup_k \PP_\ring^{(k)}(\Rat(R) > c) = 0$ for sufficiently large $c$.
\end{remark}

\subsection{Maps sampled from \texorpdfstring{$\PP_\ttF^{(p)}$}{the measure on planar maps conditioned to be finite} are subcritical}
\label{sec:subcrit}

Now that we have identified that the marginal law of the outermost gasket $M_0$ sampled from $\PP_\ttF^{(p)}$ is that of a subcritical Boltzmann map, we extrapolate from here an exponential decay in $i$ for the total perimeter of the maps added at step $i$ of the iterative construction.

We find it useful to consider
\begin{equation}
  \label{eq:17}
  h(p):=-\log F(p)
\end{equation}
with $h(0) = 0$ as a proxy for $p$, since $h(p) \asymp p$ by Proposition \ref{thm:2}.
The goal of this subsection is to obtain the following estimate.
The fact that $s_0<1$ in the following proposition is the sense in which the iterative construction of a map sampled from $\PP_\ttF^{(p)}$ behaves as a subcritical branching process.
\begin{proposition}
  \label{prop:14}
  There exists a constant $s_0\in (0,1)$ such that for all $p\in \NN$, we have
  \begin{equation}
    \EE^{(p)}_{\ttF}\bigg[\sum_{f\in \fF_{0}}h(\per_{\inn}(f))\bigg]\leq s_0 h(p).
  \end{equation}
\end{proposition}

 By a simple iteration argument, we obtain from Proposition~\ref{prop:14} the following exponential decay for the expected total perimeter of each generation for a supercritical map conditioned to be finite.
  \begin{corollary}
    \label{lem:21}
    There exist constants $c>0$ and $s_0\in (0,1)$ such that for all $p\in \NN$ and all $i\in \NN$, we have
    \begin{equation}
      \label{eq:39}
       \EE_\ttF^{(p)}\bigg[\sum_{f\in \fF_{i}}h(\per_{\inn}(f))\bigg]\leq c(s_0)^i p.
    \end{equation}
  \end{corollary}
  \begin{proof}
      By the iterative construction of $\PP_\infty^{(p)}$ in Definition~\ref{def:model}, conditioning the whole planar map $M$ to be finite given $\{\per_\inn(f)\}_{f\in \fF_i}$ is equivalent to conditioning each submap $M(f)$ glued into the inner face of the ring $R(f)$ attached in $f \in \fF_i$ to be finite. In particular, this implies 
\begin{equation}\begin{split}
    \EE_\ttF^{(p)} \bigg[ \sum_{\tilde f \in \fF_{i+1}} h(\per_\inn(\tilde f)) \bigg| \{\per_\inn(f)\}_{f\in \fF_i} \bigg] &= \sum_{f\in \fF_{i} }\EE_\ttF^{(\per_\inn(f))} \bigg[ \sum_{\tilde f \in \fF(M(f))} h(\per_\inn(\tilde f)) \bigg]\\ &\leq s_0 \sum_{f\in \fF_{i}} h(\per_\inn(f)).
\end{split}\end{equation}
The claim follows immediately from Proposition~\ref{prop:14} by an induction on $i$.
  \end{proof}

The rest of this subsection deals with the proof of Proposition~\ref{prop:14}. The argument is somewhat convoluted due to the lack of a precise estimate for $F(p)$. The following simple lemma illustrates that the main challenge in the proof of Proposition \ref{prop:14} is the requirement that $s_0<1$.
\begin{lemma}
  \label{lem:16}
  There exist constants $C,c>0$ such that for all $p$ and all $s>1$, we have
  \begin{equation}
    \label{eq:19}
    \PP^{(p)}_{\ttF}\bigg(\sum_{f\in \fF_{0}} h(\per_{\inn}(f))>s h(p) \bigg)\leq Ce^{-c(s-1)p}.
  \end{equation}
\end{lemma}
\begin{proof}
Let $E_s$ denote the event $\{\sum_{f\in \fF_{0}}h(\per_{\inn}(f))>s h(p)\}$.  By Bayes' rule, we have 
  \begin{equation}
    \label{eq:20}
    \PP^{(p)}_{\ttF}(E_s )=F(p)^{-1}\PP^{(p)}_{\infty}(E_s )\PP^{(p)}_{\infty}(M \textrm{ is finite}\lvert E_s ).
  \end{equation}
  By Definition~\ref{def:model} of $\PP_\infty^{(p)}$ along with the relation $F(p)=e^{-h(p)}$, we have
  \begin{equation}
    \label{eq:21}
    \PP^{(p)}_{\infty}(M \textrm{ is finite}\lvert E_s )\leq e^{-s h(p)}.
  \end{equation}
  Using the trivial bound $\PP_{\infty}^{(p)}(E_s)\leq 1$, we obtain $\PP_{\ttF}^{(p)}(E_s)\leq e^{-(s-1) h(p)}$ from \eqref{eq:20}. Note that by Proposition \ref{thm:2} and Lemma~\ref{lem:1}, there is a constant $c>0$ such that $h(p)\geq cp$ for all $p$.
\end{proof}
Note that \eqref{eq:21} holds for any $s>0$. For the proof of Proposition \ref{prop:14}, we improve upon the trivial bound $\PP_{\infty}^{(p)}(E_s)\leq 1$ so that we can take $s \in (0,1)$. 
More precisely, we show the following estimate, which allows us to later replace $E_s$ with the event $\{\sum_{f\in \fF_0}h(\per_{\inn}(f)) \in [sh(p),rh(p)]\}$ for an appropriate choice of parameters $0<s<1<r$.

\begin{lemma}
  \label{lem:18}
 For each $r>0$, there is a constant $c_r>0$ depending on $r$ such that for all $p$, we have
  \begin{equation}
    \label{eq:24}
    \PP_{\infty}^{(p)}\bigg(\sum_{f\in \fF_0} h(\per_{\inn}(f)) <r h(p) \bigg) \leq e^{-c_rp}.
  \end{equation}
\end{lemma}
\begin{proof}
  As in the proof of Lemma~\ref{lem:11}, we use the random walk encoding (Proposition~\ref{prop:10}) of $M_0$ sampled from $\PP_\infty^{(p)}$, which by construction is a critical non-generic Boltzmann map sampled from $\PP_\bfq^{(p)}$. Let $S_n = X_1 + \cdots + X_n$ be a random walk with i.i.d.\ step distribution $\PP(X_i=k-1)=\mu_{\bfq}(k)$ and, given $S$, let $R_i$ be sampled conditionally independently  for each $i\in \NN$ from the marginal law of $\Rat(R)$ under $\PP_\ring^{(X_i+1)}$. Since $h(p)=-\log F(p)\asymp p$ (Proposition \ref{thm:2}), there exists a constant $a>0$ such that 
  \begin{equation}
    \label{eq:25}
    \PP_{\infty}^{(p)}\bigg(\sum_{ f\in \fF_0} h(\per_{\inn}(f)) <r h(p) \bigg) \leq \PP_{\infty}^{(p)}\bigg(\sum_{f\in \fF_0} \per_{\inn}(f) < arp \bigg).
    \end{equation}
    By Proposition~\ref{prop:10} as well as the bound $1/\EE [(L^{(p)}_S+1)^{-1}]=O(p^{3/2})$ coming from Lemma~\ref{lem:215}, we have
    \begin{equation}\begin{split}
 \PP_{\infty}^{(p)}\bigg(\sum_{f\in \fF_0} \per_{\inn}(f) < arp \bigg)&=\frac{\EE\big[(L^{(p)}_S+1)^{-1}\ind\big\{\sum_{i=1}^{T_S^{(p)}}(X_i+1)R_i < arp\big\}\big]}{\EE [(L^{(p)}_S+1)^{-1}]}\\
 &\leq C_0 p^{3/2} \, \PP\bigg(\sum_{i=1}^{T_S^{(p)}}(X_i+1)R_i < arp \bigg) 
  \end{split}\end{equation}
  for some constant $C_0>0$.
  
To complete the proof, it suffices to show that for some constant $c_r$ depending on $r$, we have
\begin{equation}
  \label{eq:122}
  \PP\bigg(\sum_{i=1}^{T_S^{(p)}}(X_i+1)R_i < arp\bigg) \leq e^{-c_r p}.
\end{equation}
  The key ingredient is Lemma~\ref{lem:215}, which states that that $T_S^{(p)}$ is very unlikely to be much smaller than $p^{3/2}$. Consequently, for any constant $c>0$, there exist positive constants $C_2$ and $C_3$ such that 
  \begin{equation}
  \label{eq:lowtail}
   \PP(T_S^{(p)}\leq cp)\leq C_2e^{-C_3p}   
  \end{equation}
   for all $p$. 
Each of the i.i.d.\ nonnegative random variables $\{(X_i+1)R_i\}_{i\in \NN}$ have a positive probability of being nonzero. By the standard Chernoff bound, we know that there exist positive constants $C_4$ and $C_5$ such that for all $\ell\in \NN$, we have
\begin{equation}
  \label{eq:123}
  \PP\bigg(\sum_{i=1}^n (X_i+1)R_i \leq C_4 n \bigg)\leq e^{-C_5 n}.
\end{equation}
Thus,
\begin{equation}
  \label{eq:124}
  \PP\bigg(\sum_{i=1}^{T_S^{(p)}} (X_i+1)R_i \leq arp\bigg)\leq \PP\big(T_S^{(p)}< arp/C_4\big)+ \sum_{n = \lceil arp/C_4\rceil }^\infty \PP\bigg(\sum_{i=1}^n (X_i+1)R_i \leq C_4 n \bigg) .
\end{equation}
We complete the proof by substituting \eqref{eq:lowtail} and \eqref{eq:123} to the right-hand side.
\end{proof}

We now improve upon the result of Lemma~\ref{lem:16}.
\begin{lemma}
  \label{lem:19}
  For each $r>1$, there is a constant $c_r>0$ such that for any $s\in (0,1)$ and for all $p$ large enough, we have
  \begin{equation}
    \PP_\ttF^{(p)}\bigg(\sum_{f\in \fF_{0}} h(\per_{\inn}(f)) \in [sh(p),rh(p)] \bigg)\leq e^{(1-s)h(p)-c_rp}. 
  \end{equation}
\end{lemma}
\begin{proof}
  Let $E$ denote the event $\{\sum_{f\in \fF_0} h(\per_{\inn}(f)) \in [sh(p),rh(p)]\}$. By using Bayes's rule as in \eqref{eq:20}, we have
  \begin{equation}
    \label{eq:30}
     \PP^{(p)}_{\ttF}(E)=F(p)^{-1}\PP^{(p)}_{\infty}(E)\PP^{(p)}_{\infty}(M \textrm{ is finite}\lvert E ).
   \end{equation}
   From the iterative construction of $M$ under $\PP_\infty^{(p)}$, we have the inequality $\PP_\infty^{(p)} ( M \textnormal{ is finite} |E) \leq e^{-s h(p)}$ just as in \eqref{eq:21}.
   Recalling $F(p)=e^{-h(p)}$ and applying Lemma \ref{lem:18}, we complete the proof by choosing a slighter smaller $c_r$ if necessary.
 \end{proof}

 We are now ready to complete the proof of Proposition \ref{prop:14} for large $p$.

 \begin{lemma}
  \label{lem:50}
  There exists constants $s_0\in (0,1)$ and $p_0\in \NN$ such that for all $p\geq p_0$, we have
  \begin{equation}\label{eq:193}
    \EE^{(p)}_{\ttF}\bigg[\sum_{f\in \fF_{0}}h(\per_{\inn}(f))\bigg]\leq s_0 h(p).
  \end{equation}
\end{lemma}
 \begin{proof}
  Fix $\epsilon>0$. By combining Proposition~\ref{thm:2} with Lemma~\ref{lem:19} for $s\in (0,1)$ sufficiently close to $1$, we find that there exist constants $p_0\in \NN$, $c,C>0$, and $s_0\in (0,1)$ such that for all $p\geq p_0$,
   \begin{equation}
     \label{eq:31}
     \PP_\ttF^{(p)} \bigg(\sum_{f\in \fF_0} h(\per_{\inn}(f)) \in [s_0 h(p),(1+\epsilon)h(p)] \bigg)\leq  Ce^{-cp}.
   \end{equation}
   By Lemma \ref{lem:16}, the constants $p_0$, $C$, and $c$ can be chosen such that for all $p\geq p_0$ and $\alpha \geq  (1+\epsilon)$, 
   \begin{equation}
     \label{eq:32}
     \PP_\ttF^{(p)} \bigg(\sum_{f\in \fF_0} h(\per_{\inn}(f)) > \alpha h(p) \bigg)\leq  Ce^{-c(\alpha-1)p}.
   \end{equation}
    Combining the two estimates, we that for all $p\geq p_0$,
   \begin{equation}
     \label{eq:33}\begin{split}
      &\EE_\ttF^{(p)}\bigg[\sum_{f\in \fF_0} h(\per_{\inn}(f)) - s_0h(p)\bigg]_+ \\
      & \leq (1+\epsilon-s_0)h(p)\cdot \PP_\ttF^{(p)} \bigg(\sum_{f\in \fF_0} h(\per_{\inn}(f)) \in [s_0 h(p), (1+\epsilon) h(p)] \bigg) \\ & \hspace{1.2in}+ \EE_\ttF^{(p)}\bigg[\bigg(\sum_{f\in \fF_0} h(\per_{\inn}(f)) \bigg) \cdot \ind \bigg\{\sum_{f\in \fF_0} h(\per_{\inn}(f)) > (1+\epsilon) h(p)\bigg\} \bigg]  \\
      &\leq (1+\epsilon-s_0)h(p) \cdot Ce^{-cp} + \int_{1+\epsilon}^\infty \PP_\ttF^{(p)} \bigg(\sum_{f\in \fF_0} h(\per_{\inn}(f)) > \alpha h(p)\bigg) h(p)\,d\alpha \\
      &\leq \widetilde C p(e^{-cp}+e^{-c\epsilon p})
    \end{split}\end{equation}
    for some constant $\widetilde C>0$. Since this bound tends to 0 as $p\to\infty$, we obtain \eqref{eq:193} by slightly increasing $p_0$ and $s_0$ if necessary.
  \end{proof}
   The proof of Proposition~\ref{prop:14} for smaller values of $p$ is achieved via the following lemma.
  \begin{lemma}
    \label{prop:5}
    For each $p_0\in \NN$, there is a constant $s_0\in (0,1)$ such that for all $p\in [\![1,p_0]\!]$, we have
    \begin{equation}
      \label{eq:35}
      \EE^{(p)}_{\ttF} \bigg[\sum_{f\in \fF_0}h(\per_{\inn}(f))\bigg]\leq s_0 h(p).
    \end{equation}
  \end{lemma}
  \begin{proof}
    Let us denote
    \begin{equation}
      \label{eq:36}
      Z :=\PP_\infty^{(p)}\big(M \textnormal{ is finite} \big\lvert \{\per_{\inn}(f)\}_{f\in \fF_0} \big)=\prod_{f\in \fF_0}F(\per_{\inn}(f)).
    \end{equation}
    Note that $\EE_\infty^{(p)}[Z] = F(p)$. Moreover,
    \begin{equation}\begin{split}
      \label{eq:37}
      \EE_\ttF^{(p)}\bigg[\sum_{f\in \fF_0}h(\per_{\inn}(f))\bigg]
      &=\frac{\EE^{(p)}_\infty\big[\big(\sum_{f\in \fF_0}h(\per_{\inn}(f))\big) \,\PP^{(p)}_\infty\big(M \textnormal{ is finite}\big\lvert \sum_{f\in \fF_0}h(\per_{\inn}(f))\big) \big]}{F(p)}\\
      &=F(p)^{-1}\EE_\infty^{(p)}[Z \log (1/Z)].
    \end{split}\end{equation}
    Since the function $x\mapsto x\log(1/x)$ is strictly concave for $x>0$, by Jensen's inequality, 
    \begin{equation}
      \label{eq:38}
      \EE_\ttF^{(p)}\bigg[\sum_{f\in \fF_0}h(\per_{\inn}(f))\bigg]< F(p)^{-1}(\EE_\infty^{(p)}[Z]) \log(1/(\EE_\infty^{(p)}[Z])) =\log F(p)^{-1} = h(p).
    \end{equation}
    In particular, the inequality is strict, so we can find $s_0 \in (0,1)$ as in \eqref{eq:35} for any finite set of $p$.
  \end{proof}
  \begin{proof}[Proof of Proposition \ref{prop:14}]
      Combine Lemmas~\ref{lem:50} and \ref{prop:5}.
  \end{proof}

  \subsection{Estimates on the largest volume of submaps of finite supercritical maps} \label{sec:tail-bounds}
  The goal of this subsection is to obtain the following estimate on the maximum volume of submaps inserted into each face $f\in \fF_0$ of the gasket $M_0$, where the volume of a map refers to the total number of its vertices. 
  This result will be used in Lemma~\ref{prop:8} to show that the submaps $\{M|_f\}_{f\in \fF_0}$, which are parts of $M$ inside $f$, become negligible as $p\to\infty$ when we rescale  so that the gasket $M_0$ converges to the continuum random tree.
  \begin{proposition}
  \label{lem:23}
  For $M$ sampled from $\PP_\ttF^{(p)}$,
  \begin{equation}
    \label{eq:137}
    \frac{1}{\sqrt p} \sup_{f\in \fF_0} \Vol(M\lvert_f)\stackrel{d}{\rightarrow} 0
  \end{equation}
  as $p\rightarrow \infty$.
\end{proposition}

The main idea for the proof of Proposition~\ref{lem:23} is that instead of the volume of the whole map, it suffices to keep track of the sum of the perimeters of the rings. 
More precisely, let us denote the total perimeter of a map $M$ sampled from $\PP_\ttF^{(p)}$ over all generations as
  \begin{equation}
      \label{eq:46}
     \TPerm(M) := p + \sum_{f\in \bigcup_{i\in \NN^\#}\fF_i}(\per_{\out}(f) + \per_{\inn}(f))
   \end{equation}
    where $\per_\out(f)$ and $\per_\inn(f)$ are, respectively, the outer and inner half-perimeters of the ring $R(f)$ inserted into the face $f \in \bigcup_{i\in \NN^\#}\fF_i$. 

\begin{lemma}  
    \label{lem:4}
    There is a constant $C>1$ such that for any $p\in \NN$, if $M$ is a map sampled from $\PP_\ttF^{(p)}$, then $\Vol(M)\leq C \cdot \TPerm(M)$ almost surely. 
  \end{lemma}
  \begin{proof}
Recalling the construction of supercritical maps from Definition \ref{def:model}, we have
    \begin{equation}
      \label{eq:91}
      \Vol(M)\leq \Vol(M_0) +\sum_{f\in \bigcup_{i\in \NN^\#}\fF_{i}} (\Vol(R(f)) + \Vol(M(f)))
    \end{equation}
    where $R(f)$ is the ring inserted into the face $f$ and $M(f)$ is the Boltzmann map inserted into the inner face of the ring $R(f)$. By the non-thickness condition in Definition \ref{def:ring}, there is a constant $C$ such that for any ring $R(f)$, we have $\Vol(R(f))\leq C(\per_\out(f) + \per_{\inn}(f))$. Observe also that the outermost gasket $M_0$ satisfies $\Vol(M_0) \leq 2(p + \sum_{f\in \fF_0}\per_\out(f))$, and similarly 
    \begin{equation}\Vol(M(f)) \leq 2\bigg(\per_\inn(f) + \sum_{\tilde f \in \fF(M(f))} \per_\out(\tilde f)\bigg)\end{equation}
    for every Boltzmann map $M(f)$ added during the construction. Hence, we obtain
    \begin{equation}
        \Vol(M) \leq 2p + (C+2) \sum_{f\in \bigcup_{i\in \NN^\#}\fF_{i}} (\per_\out(f) + \per_\inn(f)) \leq (C+2) \TPerm(M)
    \end{equation}
    by combining all of the above bounds.
  \end{proof}

  The first step in the proof of Proposition~\ref{lem:23} is to show that the gasket decomposition of a map sampled from $\PP_{\ttF}^{(p)}$ does not go on for many steps.
  \begin{lemma}
    \label{lem:35}
    Let $M$ be a map sampled from $\PP_\ttF^{(p)}$. Let $
      T_{\ext}(M):=\min\{i: \fF_i=\varnothing\}$
    be the total number of iterations in the construction of $M$.
    Then, there exist constants $c,C>0$ such that for every $p$ and $n$, we have 
    $\PP_\ttF^{(p)}(T_\ext(M)>n)\leq Cpe^{-cn}$.
  \end{lemma}  
  \begin{proof}
    If $T_\ext(M)>n$, then $\sum_{f\in \fF_n}h(\per_{\inn}(f))>0$. Since $\inf_{p' \in \NN}h(p')>0$ by Proposition~\ref{thm:2} and Lemma~\ref{lem:1}, Markov's inequality implies
    \begin{equation}
    \label{eq:markov}\begin{split}
        \PP_\ttF^{(p)}\bigg(\sum_{f\in \fF_n}h(\per_{\inn}(f))>0\bigg)&= \PP_\ttF^{(p)}\bigg(\sum_{f\in \fF_{n}}h(\per_{\inn}(f))\geq \inf_{p' \in \NN}h(p')\bigg)\leq \frac{\EE_\ttF^{(p)}\big[\sum_{f\in \fF_{n}}h(\per_{\inn}(f))\big]}{\inf_{p' \in \NN}h(p')}.
    \end{split}\end{equation}
    The claim now follows from Corollary~\ref{lem:21}.
  \end{proof}
  
  Now, we show that with high probability, the total perimeter at is kept small at every level.
 \begin{lemma}
    \label{lem:36}
    Consider the events
    \begin{equation}
      \label{eq:142}
      G_t:=\bigg\{\sum_{f\in \fF_i} \per_\inn(f) \geq t \quad \textnormal{for some }i  \bigg\} \quad \text{and}  \quad \widetilde G_t:=\bigg\{\sum_{f\in \fF_i} \per_\out(f) \geq t \quad \textnormal{for some }i \bigg\} .
    \end{equation}
    Then, there is a constant $c>0$ such that $\PP_\ttF^{(p)}(G_t \cup \widetilde G_t) \leq e^{-ct}/F(p)$ for all $t$.
  \end{lemma}
  \begin{proof}
    By Proposition~\ref{thm:2} and Lemma~\ref{lem:1}, there exists $c>0$ such that $h(p) \geq cp$ for all $p \in \NN$.
    The proof again uses the Bayes's rule argument from \eqref{eq:20}. Note that we have
    \begin{equation}
      \label{eq:143}
      \PP_\ttF^{(p)}(G_t)\leq F(p)^{-1}\PP_{\infty}^{(p)}(M \textrm{ is finite}\lvert G_t).
    \end{equation}
    On the event $G_t$, define $I_t := \inf\{i \in \NN^\#: \sum_{f\in \fF_{i}} \per_\inn(f) \geq t\}$, so that $\sum_{f\in \fF_{I_t}} h(\per_\inn(f))\geq ct$. Further conditioning on $I_t$, $\fF_{I_t}$, and $\{\per_\inn(f)\}_{f\in \fF_{I_t}}$, we have
    \begin{equation}
      \label{eq:144}\begin{split}
      \PP_{\infty}^{(p)}(M \textrm{ is finite}\lvert G_t)&=\EE_{\infty}^{(p)}\big[\PP_{\infty}^{(p)}(M \textrm{ is finite}\lvert I_t,\fF_{I_t},\{\per_\inn(f)\}_{f\in \fF_{I_t}})\big| G_t \big]\\
      &= \EE_{\infty}^{(p)}\bigg[ \prod_{f\in \fF_{I_t}}F(\per_\inn(f)) \bigg|G_t \bigg] =\EE_{\infty}^{(p)}\bigg[e^{-\sum_{f\in \fF_{I_t}}h(\per_{\inn}(f))} \bigg| G_t \bigg]\leq e^{-ct}.
    \end{split}\end{equation}
    Substituting the above into \eqref{eq:143} gives $\PP_\ttF^{(p)}(G_t) \leq e^{-ct}/F(p)$. 
    
    Define $\widetilde I_t$ analogously on the event $\widetilde G_t$. Then, similarly to \eqref{eq:144}, there is a constant $\tilde c>0$ such that 
    \begin{equation}
        \begin{split}
      \PP_{\infty}^{(p)}(M \textrm{ is finite}\lvert \widetilde G_t)&=\EE_{\infty}^{(p)}\big[\PP_{\infty}^{(p)}(M \textrm{ is finite}\lvert I_t,\fF_{I_t},\{\per_\out(f)\}_{f\in \fF_{I_t}})\big| \widetilde G_t \big]\\
      &= \EE_{\infty}^{(p)}\bigg[ \prod_{f\in \fF_{I_t}} \EE_{\ring}^{\per_\out(f)} \big[F(\per_\inn(f))\big] \bigg| \widetilde G_t \bigg]\\& \leq \EE_{\infty}^{(p)} \bigg[\prod_{f\in \fF_{I_t}} \EE_{\ring}^{\per_\out(f)} \big[e^{-c \cdot \per_{\inn}(f)} \big]\bigg| \widetilde G_t \bigg] \\& \leq \EE_{\infty}^{(p)}\bigg[e^{-\tilde c\sum_{f\in \fF_{I_t}}\per_{\out}(f)} \bigg| \widetilde G_t \bigg]\leq e^{-\tilde ct}
    \end{split}
    \end{equation}
    where we used the lower tail condition in Definition~\ref{def:ring} for the penultimate inequality. Using Bayes's rule as in \eqref{eq:143}, we obtain $\PP_\ttF^{(p)}(\widetilde G_t) \leq e^{-\tilde c t}/F(p)$.
  \end{proof}

Combining the above two lemmas, we obtain a bound on $\TPerm(M)$ for a finite supercritical map $M$.
\begin{lemma}
    \label{lem:37}
    Given $\delta>0$, there exist positive constants $C$ and $c$ such that for all $t\geq k^\delta$ and all $k\in \NN$, we have
\begin{equation}\PP_\ttF^{(k)}(\TPerm(M)\geq tk)\leq Ce^{-c\sqrt {t}}.\end{equation} 
 
   \end{lemma}
   \begin{proof} Using Lemma \ref{lem:4} along with the notation of Lemmas \ref{lem:35} and \ref{lem:36}, we have
     \begin{equation}
       \label{eq:146}
       \PP_\ttF^{(k)}\big(\TPerm(M)\geq t k\big)\leq \PP_\ttF^{(k)}\big(T_\ext(M)\geq \sqrt{t}\big) +  \PP_\ttF^{(k)}\big(G_{k\sqrt t /2} \cup \widetilde G_{k\sqrt t/2}\big) .
     \end{equation}
    By Lemma~\ref{lem:35}, the first term on the right-hand side is bounded above by $C_1ke^{-c_1\sqrt t}\leq C_1'e^{-c_1'\sqrt{t}}$ for some constants $C_1,C_1',c_1,c_1'>0$, where we have used that $t\geq k^\delta$. By Lemma~\ref{lem:36}, the second term in \eqref{eq:146} is bounded above by $C_2e^{c_2k(1-\sqrt t)}\leq C_2'e^{-c_2'\sqrt{t}}$ for some positive constants $C_2,C_2',c_2,c_2'$. This completes the proof.
   \end{proof}

  We are now ready to complete the proof of Proposition \ref{lem:23}.
\begin{proof}[Proof of Proposition \ref{lem:23}]
  By Lemma \ref{lem:4}, it suffices to show that
  \begin{equation}
    \label{eq:145}
     \frac{1}{\sqrt p}\sup_{f\in \fF_0} \TPerm(M\lvert_f)\stackrel{d}{\rightarrow} 0.
   \end{equation}
   By Lemmas~\ref{lem:38} and \ref{lem:11}, we know that for some constant $\Delta>2$, we have
   \begin{equation}
     \label{eq:151}
    \lim_{p\to\infty}\PP_\ttF^{(p)}\bigg(\sup_{f\in \fF_0}\per_\inn(f)\leq p^{1/\Delta} \;\; \text{and}\;\; |\fF_0|\leq -\frac{2p}{m_{\bfq'}}\bigg) = 1.
   \end{equation}
    Denote the above event as $A_p$. Now, for any fixed $\epsilon>0$, 
   \begin{equation}
     \label{eq:152}\begin{split}
     \PP_{\ttF}^{(p)}\bigg(\sup_{f\in \fF_0} \TPerm(M\lvert_f)\geq \epsilon \sqrt{p}\bigg)&\leq \PP_\ttF^{(p)}\big(A_p^c\big)+\bigg(-\frac{2p}{m_{\bfq'}}\bigg) \sup_{k\leq p^{1/\Delta}}\PP_{\ttF}^{(k)}\big( \TPerm(M) \geq \epsilon \sqrt{p}\big)\\
     &=\PP_\ttF^{(p)}\big(A_p^c\big)+\bigg(-\frac{2p}{m_{\bfq'}}\bigg) \sup_{k\leq p^{1/\Delta}}\PP_{\ttF}^{(k)}\big( \TPerm(M) \geq k(\epsilon\sqrt{p}/k)\big).
   \end{split}\end{equation}
    The second term tends to 0 as $p\to\infty$ by Lemma \ref{lem:37} and the fact that $1/\Delta<1/2$, thus proving \eqref{eq:145}.
    \end{proof}

\section{Convergence to the CRT}
\label{sec:conv}

We finally prove Theorem~\ref{thm:1} in this section, showing that the map $M^{(p)}$ sampled from $\PP_\ttF^{(p)}$ converges to the continuum random tree after appropriate scaling. 
While we have so far used $M$ to denote a sample from $\PP_\ttF^{(p)}$, we will always use $M^{(p)}$ to denote such a sample from now so as to avoid confusion in the many convergence statements in this section.
We denote its outermost gasket ($M_0$ in the iterative construction of Definition~\ref{def:model}) as $M_0^{(p)}$. We recall the notation that if $f$ is an internal face of $M_0^{(p)}$ (or $\partial M^{(p)}$), then $M^{(p)}|_f$ denotes the edges and vertices that bound this face in $M_0^{(p)}$ as well as those of $M^{(p)}$ that are glued into this face.

In Proposition~\ref{prop:2}, we saw that the marginal law on the outermost gasket $M_0^{(p)}$ sampled from $\PP_\ttF^{(p)}$ is that of a subcritical Boltzmann map $\PP_{\bfq'}^{(p)}$. Then, by Proposition~\ref{prop:6}, there is a a constant $\theta_{\bfq'}>0$ such that $(\theta_{\bfq'}/\sqrt p)M_0^{(p)}$ converges to the $\CRT$. Our goal in this section is to upgrade this result to one for the entire map $M^{(p)}$ by showing that the submaps added into the faces of $M_0^{(p)}$ change distances by a constant factor.

To do so, instead of the outermost gasket $M_0^{(p)}$, we will work with its outer boundary $\partial M^{(p)}$, which consists of the vertices and edges that border the root face of $M^{(p)}$. The reason is that the scaling limit of $\partial M^{(p)}$ is also the $\CRT$ by Proposition \ref{prop:4}, and it has additional integrability through its connection with $\BGW$ trees studied in the works \cite{CK15,Ric18}.
For the first step, we reduce Theorem~\ref{thm:1} to a convergence statement for a ``projection" of the whole map $M^{(p)}$ to its boundary $\partial M^{(p)}$.

For the rest of the section, given a graph $M$, we denote its set of vertices and the graph distance on it by $V_M$ and $d_M$, respectively. We remind the reader that $rM$ denotes the metric space $(V_M, rd_M)$ for $r\in \RR_+$.

\begin{proposition}
   \label{prop:8}
   Let $M^{(p)}$ be sampled from $\PP_\ttF^{(p)}$
   and denote $d_{M^{(p)}}^\partial:= d_{M^{(p)}}|_{V_{\partial M^{(p)}}\times V_{\partial M^{(p)}}}$.
   Then, there exists a constant $\theta_{\ttF}$ such that 
   \begin{equation}
     \label{eq:61}
     \mathcal X^{(p)} := \bigg(V_{\partial M^{(p)}}, \frac{\theta_{\ttF}}{\sqrt{p}} d_{M^{(p)}}^\partial \bigg) \stackrel{d}{\rightarrow}\CRT
   \end{equation}
   as $p\to\infty$ in the Gromov--Hausdorff distance.
 \end{proposition}

 \begin{proof}[Proof of Theorem \ref{thm:1} assuming Proposition \ref{prop:8}]  
   Let 
   \begin{equation}
     \label{eq:62}
      L^{(p)}:= \max_{v\in V_{M_0^{(p)}}} \min_{w\in V_{\partial M^{(p)}}}d_{M^{(p)}}(v,w)
   \end{equation}
   be the maximum distance from a vertex in $M_0^{(p)}$ to the boundary $\partial M^{(p)}$.
   Recalling that $M_0^{(p)}$ sampled from $\PP_\ttF^{(p)}$ is a subcritical Boltzmann map, we see from Proposition~\ref{prop:6} that 
   $L^{(p)}/\sqrt p \to 0$ in distribution as $p\to\infty$.
   Combining this with Proposition~\ref{lem:23}, 
   \begin{equation}\label{eq:156}
       \frac{1}{\sqrt p}\max_{f\in \fF(\partial M^{(p)})} \mathrm{diam}(M^{(p)}|_f) \leq \frac{1}{\sqrt p}\bigg( L^{(p)} + \max_{\tilde f\in \fF(M_{0}^{(p)})} \mathrm{diam}(M^{(p)}|_{\tilde f})\bigg)
       \end{equation}
       converges to 0 in distribution as $p\to\infty$.   
    
    Now consider ${\mathcal X}^{(p)}$ as naturally embedded within $(\theta_\ttF/\sqrt p) M^{(p)}$. Under this isometric embedding, the Hausdorff distance between the $V_{M^{(p)}}$ and $V_{\partial M^{(p)}}$ is bounded above by the left-hand side of \eqref{eq:156}. Hence, $d_{\mathrm{GH}}(\mathcal X^{(p)}, (\theta_\ttF/\sqrt p) M^{(p)})\to 0$ in distribution as $p\to\infty$.
 \end{proof}

 The rest of this section is thus dedicated to the proof of Proposition~\ref{prop:8}. The main idea is to study the looptree structure of $\partial M^{(p)}$ in association with Bienaym\'e--Galton--Watson trees. The argument proceeds along the same lines as the proof of Proposition~\ref{prop:4} in \cite{KR20}, which uses the spinal decomposition to prove the convergence of critical looptrees to the $\CRT$. The additional work we do is to show that the addition of vertices and edges going from $\partial M^{(p)}$ to $M^{(p)}$ affect the scaling limit merely by multiplying distances by a deterministic constant.

\subsection{Boundaries of Boltzmann maps and looptrees}
\label{sec:looptree}

We begin by introducing the connection between looptrees and boundaries of Boltzmann maps.  First, a few notations: $\cT$ refers to a rooted plane tree, by which we mean a tree with a distinguished vertex (called the root and denoted by $\root$) and an ordering specified among the children of any vertex. Thus, intuitively, we can think of the root of a rooted plane tree being at the bottom and the children of any vertex being located above the corresponding vertex. 

Here are the definitions of the looptree and the contracted looptree associated with a rooted planar tree as given in \cite{CK15}.\begin{definition}
 Let $\cT$ be a rooted plane tree. We define $\Loop(\cT)$ to be a planar map which has the same set of vertices as $\cT$, with an edge between $v_1,v_2\in V_{\Loop(\cT)}$ if and only if either of the following are true.
 \begin{enumerate}[(1)]
     \item $v_1$ and $v_2$ are the consecutive children of a common parent in $\cT$.
     \item $v_1$ is the leftmost/rightmost child of $v_2$ in $\cT$, or vice versa.
 \end{enumerate}
 The \textbf{contracted looptree} $\overline \Loop(\cT)$ is obtained from $\Loop(\cT)$ by contracting every edge $(u,v)$ for which $v$ is the rightmost child of $u$ in $\cT$.  
\end{definition}
\begin{figure}
    \centering
    \includegraphics[width=0.8\linewidth]{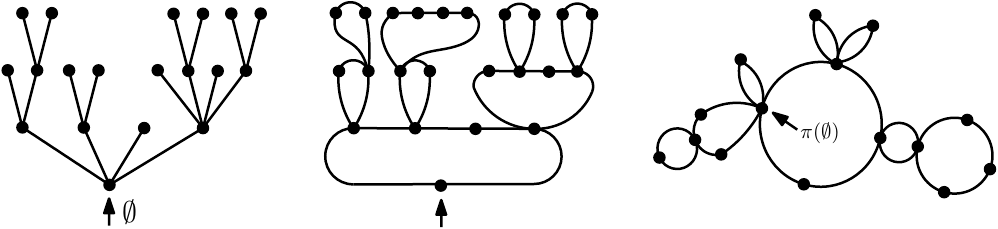}
    \caption{From left to right: a tree $\mathcal T$, the corresponding looptree $\Loop(\cT)$, and the contracted looptree $\overline{\Loop}(\cT)$. The root vertex $\root \in V_{\cT}$ and the corresponding vertices on $\Loop(\cT)$ and $\pi(\root)\in \overline{\Loop}(\cT)$ are marked with arrows.}
    \label{fig:looptree}
\end{figure}

See Figure~\ref{fig:looptree} for an illustration of this definition.
Let $\pi$ denote the natural projection from $V_{\cT}$ to $V_{\overline{\Loop}(\cT)}$. We note that $\pi$ is one-to-one on the leaves of $\cT$ (i.e., vertices with no children). 

For $v\in V_\cT$, let $f_{v} \in \fF(\overline{\Loop}(\cT))$ be the face enclosed by the vertices $\pi($children of $v$ in $\cT)\subset V_{\overline{\Loop}(\cT)}$. This forms a one-to-one correspondence between the faces of $\overline{\Loop(\cT)}$ and the non-leaf vertices of $\cT$. Let $f_v = \varnothing$ if $v$ is a leaf in $\cT$.

As observed in \cite[Lemma~4.3]{CK15} (see also \cite[Lemma~4.1]{Ric18} for the statement for non-triangulations), if $M$ is a rooted planar map (i.e., with a fixed outermost face on the right of the root edge, and the root vertex $\root$ given by the incident vertex of the root edge), then $\partial M$ is almost a looptree: we just need to duplicate each single edge connecting loops into a double edge. This procedure was described in detail in \cite[Section 2.3]{CK15}. 
\begin{definition}
  \label{def:scoop}
  Given a rooted planar map $M$, define the \textbf{scooped-out map} $\scoop(M)$ as the outer boundary $\partial M$ modified so that every edge $e$ in $\partial M$ which is adjacent to the outermost face on both sides is duplicated. 
\end{definition}
\begin{figure}
    \centering
    \includegraphics[width=0.6\linewidth]{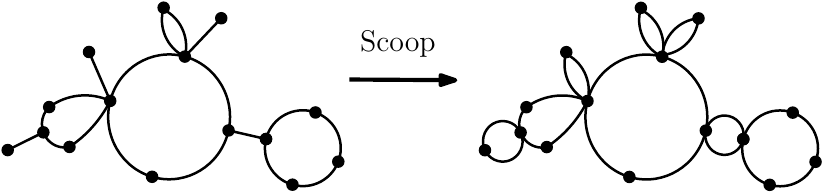}
    \caption{A planar map (left) and the corresponding scooped-out map (right).}
    \label{fig:scoop}
\end{figure}

Given a rooted planar map $M$, it is straightforward to see that there is a unique rooted plane tree $\cT_M$ such that $\scoop(M) = \overline{\Loop}(\cT_M)$ and $\pi($root vertex of $\cT_M)=$ root vertex of $\scoop(M)$. We use $\root$ to refer to both root vertices.
Note that $\scoop(M)$ and $\partial M$ can be different planar maps but define the same metric space. Hence, \eqref{eq:61} can be stated equivalently with $\scoop(M^{(p)}) = \overline{\Loop}(\cT_{M^{(p)}})$ in place of $\partial M^{(p)}$. 

For $M^{(p)}$ sampled from $\PP_\ttF^{(p)}$, note that $\cT_{M^{(p)}} = \cT_{M_0^{(p)}}$. The key fact is that since $M_0^{(p)}$ is a subcritical Boltzmann map, $\cT_{M_0^{(p)}}$ has the distribution of a BGW tree. The following proposition  is an immediate application of the relationship between the boundary of a subcritical Boltzmann map and a BGW tree described in \cite[Lemma 4.3]{CK15} and \cite[Proposition 3.6, Lemma~4.1]{Ric18} given Proposition~\ref{prop:2}.
\begin{proposition}\label{prop:41}
    There is a finite variance measure $\nu_{\bfq'}$ on the nonnegative even integers such that for every $p\in \NN$, the rooted tree $\cT^{(p)}:= \cT_{M^{(p)}}$ is a critical BGW tree with offspring distribution $\nu_{\bfq'}$ conditioned to have $2p+1$ total vertices.
\end{proposition}

We now describe how to sample $M^{(p)}$ given $\cT^{(p)}$. First, let $\widehat \PP^{( p)}_{\bfq'}$ denote the law of a Boltzmann map with weight sequence $\bfq'$ conditioned to have a simple boundary of perimeter $2 p$. For $ p=1$, this includes the map consisting of one edge and no interior faces. As usual, let $\fF(\partial M)$ be the collection of inner faces of $\partial M$ (i.e., excluding the root face of $M$).

\begin{lemma}[{\cite[Corollary 3.7]{Ric18}}]
  \label{lem:214}
      Suppose $M$ is a Boltzmann map with distribution $\PP_{\bfq'}^{(p)}$. Given $\partial M$, the submaps $\{M\lvert_f\}_{f\in \fF(\partial M)}$ are conditionally independent, with the conditional law of $M|_f$ given by $\widehat \PP_{\bfq'}^{(\per(f))}$ for each $f\in \fF(\partial M)$.
\end{lemma}

Let $\widehat \PP_{\ttF}^{( p)}$ be the probability measure on the chain $(M_i^{(p)},\fF_i^{(p)})_{i\in \NN^\#}$ obtained from $\PP_\ttF^{( p)}$ by conditioning on $M_0^{(p)}$ (and thus $M^{(p)}$) having a simple boundary.
Since the marginal law of $\PP_\ttF^{(p)}$ on the outermost gasket $M_0^{(p)}$ is that of a Boltzmann map (Proposition~\ref{prop:2}), we obtain the following variation from Lemma~\ref{lem:214}.
\begin{lemma}
  \label{lem:39}
  Let $M^{(p)}$ be sampled from $\PP_{\ttF}^{(p)}$. The maps $\{M^{(p)}\lvert_f\}_{f\in \fF(\partial M^{(p)})}$ are conditionally independent given $\partial M^{(p)}$, and the conditional distributions are equal to $\widehat \PP_{\ttF}^{(\per(f))}$.
\end{lemma}
We thus have a complete description of the conditional law of $M^{(p)}$ given $\cT^{(p)}$.

\subsection{The spine decomposition of the boundary map}
\label{sec:spine}

In this subsection, we introduce the spine decomposition of $\cT^{(p)}$ that we use to prove Proposition~\ref{prop:8}.
\begin{definition}
  Given a rooted plane tree $\cT$ with root $\root$ and vertex $u\in V_\cT$, we define the trunk $\Trunk(\cT,u)$ as the induced subgraph obtained by considering the unique path $\root, v_1, v_2, \dots, v_{k-1}, u$ from $\root$ to $u$ in the tree $\cT$ along with all the vertices that are the immediate children of $\root, v_1,\dots,v_{k-1}$.
\end{definition}

The goal is to construct $\Trunk(\cT, V_\unif)$ when $V_\unif$ is a uniformly chosen random vertex of a BGW tree $\cT$. This is done in two steps: first sample the random distance between $\root$ and $V_\unif$, and then construct the entire trunk given this distance. The latter law is given in terms of $\nu^*_{\bfq'}$, the sized biased version of $\nu_{\bfq'}$, defined as
  \begin{equation}
    \label{eq:47}
    \nu^*_{\bfq'}(i)= \frac{i\cdot \nu_{\bfq'}(i)}{\sum_{k\in\NN^\#} k\cdot \nu_{\bfq'}(k)}
  \end{equation}
  for $i\in \NN^\#$ as first given in \cite{Kes86}. Since $\nu_{\bfq'}$ has a finite variance, $\nu^*_{\bfq'}$ has a finite mean.
\begin{definition}
  \label{def:bias}
    Given $h\in \NN$, let $\Trunk^*_h$ be a random plane tree of height $h$ distributed as follows. Denoting its ``spine" as $v_0,\dots,v_{h-1}$ where $v_0$ is the root, each vertex $v_i$ has an independent number of children with distribution $\nu^*_{\bfq'}$. 
    Given the topology of the tree (that is, the tree without the orderings between its vertices), the location of $v_{i+1}$ is chosen uniformly among the children of $v_{i}$, with this choice being conditionally independent for all the values of $i$. 
\end{definition}

Again, we sample $M^{(p)}$ from $\PP_\ttF^{(p)}$ and define its root vertex to be the head of the root edge of the Boltzmann map $M_0^{(p)}$ (cf.\ Section~\ref{sec:boltzmann}). We denote by $\cT^{(p)}$ the unique plane tree such that $\scoop(M^{(p)}) = \overline{\Loop}(\cT^{(p)})$ where its root vertex inherited from $M^{(p)}$ is equal to $\pi(\root)$. We now record the spine decomposition of the $\BGW$ tree $\cT^{(p)}$ as described in Proposition~\ref{prop:41}.
\begin{lemma}[{\cite[Theorem 3]{KR20}}]
  \label{lem:219}
  Let $V_\unif^{(p)}$ be a uniform vertex chosen from $\cT^{(p)}$. There exists a constant $c>0$ such that if $X$ is a positive random variable with density $cxe^{-x^2}dx$, then 
  \begin{equation}
    d_{\mathrm{TV}}\big(\Trunk(\cT^{(p)}, V_\unif^{(p)}), \Trunk^*_{\lfloor X\sqrt{p}\rfloor}\big)\rightarrow 0
  \end{equation}
  as $p\to\infty$. Here, $d_{\mathrm{TV}}$ denotes the total variation distance. 
  \end{lemma}

The following key technical lemma links the metric structure of $M^{(p)}$ to that of $\cT^{(p)}$, allowing us to rephrase the main task in the proof of Proposition~\ref{prop:8} in terms of the spine decomposition of the BGW tree $\cT^{(p)}$. It is similar in its statement and proof to \cite[Lemma~15(ii)]{KR20}, with compares distances in a subcritical Boltzmann map $M$ to that in the associated tree $\cT_M$. Recall that $\pi$ denotes the natural projection from $V_\cT^{(p)}$ to $V_{\overline{\Loop}(\cT^{(p)})}=V_{\scoop(M^{(p)})} = V_{\partial M^{(p)}} \subset V_{M^{(p)}}$. 

\begin{proposition}
  \label{prop:129}
  Let $V_\unif^{(p)}$ be a uniformly chosen vertex on $\cT^{(p)}$. Then, there exists a constant $\hat \theta_{\ttF}$ such that 
\begin{equation}
    \label{eq:69}
    \frac{1}{\sqrt{p}}\big(d_{M^{(p)}}(\pi(\root),\pi(V_\unif^{(p)}))- \hat \theta_\ttF \, d_{\cT^{(p)}}(\root,V_\unif^{(p)})\big)\stackrel{d}{\rightarrow}0
\end{equation}
  as $p\rightarrow \infty$.
\end{proposition}
\begin{proof}
Observe that given the unique path $(v_0,v_1,v_2,\dots ,v_{h})$ from $v_0 = \root$ to $v_h = V_\unif^{(p)}$ in $\cT^{(p)}$ (hence $h = d_{\cT^{(p)}}(\root, V_\unif^{(p)})$), if we define
 \begin{equation}
    Y_i:=d_{{M}^{(p)}}(\pi(v_i),\pi(v_{i+1}))
\end{equation}    
for each $i=0,1,\dots,h-1$,
then we have 
\begin{equation}
    d_{M^{(p)}}(\pi(\root),\pi(V^{(p)}))= \sum_{i=0}^{h-1} Y_i.
\end{equation}
This is because any path from $\pi(\varnothing)$ to $\pi(v_h)$ in $M^{(p)}$ must pass through the faces of $\partial M^{(p)}$ which are conjoined by $\pi(v_1), \pi(v_2),\dots,\pi(v_{h-1})$ in this exact order.

Let $p_i$ denote half the number of children of $k_i$. 
The following description is an immediate consequence of Proposition~\ref{prop:41} and Lemma~\ref{lem:39}.
\begin{quote}
    Given $(p_0,p_1,\dots,p_{h-1})$, the $Y_i$ are conditionally independent. The conditional distribution of $Y_i$ agrees with that of $d_{M_i}(v,w)$ where $M_i$ is sampled from $\widehat{\PP}_\ttF^{(p_i)}$ and $v,w$ are vertices sampled conditionally independently given $M_i$ from the uniform distribution on $\partial M_i$.
\end{quote}
Let $\hat p$ be sampled from $\nu_{\bfq'}^*$ and $\widehat{M}$ from $\widehat \PP_\ttF^{(\hat p)}$. Let $\widehat Y = d_{\widehat M}(v,w)$, where $v$ and $w$ are vertices sampled conditionally independently given $\widehat M$ from the uniform distribution on $\partial \widehat M$. Then, since $\widehat Y \leq \hat p$ and $\nu^*_{\bfq'}$ has a finite mean, $\EE[\widehat Y]<\infty$. Let $\widehat Y_0, \widehat Y_1,\dots$ be i.i.d.\ random variables with the law of $\widehat Y$.

We claim that \eqref{eq:69} holds with $\hat \theta_\ttF := \EE[\widehat Y]$. 
By Lemma~\ref{lem:219}, there exists a random variable $X$ independent from $\{\widehat Y_i\}_{i\in \NN^\#}$ with density $cxe^{-x^2}dx$ such that the total variation distance between the two random-length sequences $(Y_0, Y_1,\dots,Y_{\ell-1})$ and $(\widehat Y_0, \widehat Y_1,\dots,\widehat Y_{\lfloor X\sqrt p\rfloor})$ tends to 0 as $p\to\infty$. Hence, it suffices to check that 
\begin{equation}
    \frac{1}{\sqrt p} \bigg( \sum_{i=1}^{\lfloor X\sqrt p\rfloor} \widehat Y_i - \lfloor X\sqrt p\rfloor \EE\big[\widehat Y\big] \bigg) \stackrel{d}{\to} 0
\end{equation}
as $p\to\infty$. This is an immediate consequence of the law of large numbers: note that for any $0<x_1<x_2 < \infty$ and $\epsilon>0$, we have 
\begin{equation}\sup_{x\in [x_1,x_2]} \PP\bigg(\frac{1}{ x\sqrt p} \sum_{i=1}^{\lfloor x\sqrt p\rfloor} (\widehat Y_i -  \EE[\widehat Y]) \geq \frac{\epsilon}{x}\bigg) \to 0\end{equation}
as $p\to\infty$.
\end{proof}

\subsection{Proof of Proposition~\ref{prop:8}}
\label{sec:implies}
 We now move on to defining the various height functions on $\cT^{(p)}$. Let $\root = v_0,v_1,\dots,v_{2p}$ be the enumeration of vertices of $\cT^{(p)}$ in lexicographic order (i.e., depth-first search). Given a metric $d^{(p)}$ on $\cT^{(p)}$, we define the corresponding \textbf{height function} $\mathrm{H}:[0,1]\to [0,\infty)$ as 
 \begin{equation}
     \mathrm{H}_{\cT^{(p)};d^{(p)}}\bigg(\frac{i}{2p}\bigg) = d^{(p)}(\root, v_i) \quad \text{for } i=0,1,\dots,2p
 \end{equation}
 and linearly interpolated in between. Recall that $\pi$ denotes the natural projection from $V_{\cT}$ to $V_{\overline{\Loop}(\cT)}$. There are three height functions that we will be particularly interested in.
 \begin{itemize}
     \item $\mathrm{H}_{\cT}^{(p)}$ defined from the graph metric $d_{\cT^{(p)}}$ on $\cT^{(p)}$.
     \item $\mathrm{H}_{\Loop}^{(p)}$ defined from the pull-back of the graph metric on ${\overline{\Loop}}(\cT^{(p)})$ by $\pi$ (that is, $d^{(p)}(v,w) = d_{{\overline{\Loop}}(\cT^{(p)})}(\pi(v),\pi(w))$).
     \item $\mathrm{H}^{(p)}_M$ defined from the pull-back of the graph metric on $M^{(p)}$ by $\pi$ (that is, $d^{(p)}(v,w) = d_{M^{(p)}}(\pi(v),\pi(w))$.
 \end{itemize}

We also need the \textbf{contour function} of $\cT^{(p)}$. This time, let $\root = w_0,w_1,\dots,w_{4p}=\root$ be the enumeration of the vertices with $\cT^{(p)}$ (with duplicity) as we follow along the contour of $\cT^{(p)}$. Define $\mathrm{C}_\cT^{(p)}:[0,1]\to[0,\infty)$ as 
\begin{equation}
    \mathrm{C}_\cT^{(p)}\bigg(\frac{i}{4p}\bigg) = d_{\cT^{(p)}}(\root, w_i) \quad \text{for } i = 0,1,\dots,4p
\end{equation}
and linearly interpolated in between.

\begin{lemma}
  \label{lem:31}
  Let $\sigma_{\bfq'}^2$ be the variance of the offspring distribution $\nu_{\bfq'}$ for the BGW tree $\cT^{(p)}$ found in Proposition~\ref{prop:41} and let $\hat \theta_\ttF$ refer to the constant found in Proposition~\ref{prop:129}. Denote $\theta_\ttF := \hat \theta_\ttF \sigma_{\bfq'} / (2\sqrt 2)$. Let $(\bbe_t:0\leq t\leq 1)$ be the normalized Brownian excursion. Then, as $p\to \infty$, the joint convergence
  \begin{equation}\label{eq:205}
      \bigg(\frac{\sigma_{\bfq'}}{2\sqrt{2p}}\mathrm{C}_\cT^{(p)}(t) , \frac{\sigma_{\bfq'}}{2\sqrt{2p}}\mathrm{H}_\cT^{(p)}(t) , \frac{\theta_\ttF }{\sqrt{p}}\mathrm{H}_M^{(p)}(t) \bigg)_{0 \leq t \leq 1}  \xrightarrow{d} \big(\bbe_t,\bbe_t, \bbe_t\big)_{0\leq t\leq 1} 
  \end{equation}
  holds in the space $C([0,1])^3$ endowed with the supremum norm in each coordinate.
\end{lemma}
\begin{proof}
    The proof is nearly identical to that of \cite[Proposition~13]{KR20}. Comparing Lemma~15(ii) in \cite{KR20} with Proposition~\ref{prop:129}, all that needs to be shown is the tightness of the processes $((1/\sqrt p)\mathrm{H}_M^{(p)})_{p\in \NN}$. This follows from \cite[Lemma~15(i)]{KR20}, which states that $((1/\sqrt p) \mathrm{H}_{\Loop}^{(p)})_{p\in \NN}$ is tight,\footnote{To be precise, the statement in \cite[Lemma~15]{KR20} is for the height function associated with the non-contracted looptree. The proof for the contracted looptree also holds as well, as discussed in Section~6.1 of the same paper.} since we have $d_{M^{(p)}} (\pi(v),\pi(w))\leq d_{\Loop(\cT^{(p)})}(\pi(v),\pi(w))$ for every $v,w\in V_{\cT^{(p)}}$. 
\end{proof}

We are finally ready to complete the proof of Proposition \ref{prop:8} and thereby that of Theorem \ref{thm:1}. The proof proceeds in a similar manner as \cite[Section 5.4]{KR20}.
\begin{proof}[Proof of Proposition \ref{prop:8}]
  We begin by choosing a Skorokhod embedding of $(\cT^{(p)})_{p\in \NN}$ and $\bbe$ in the same probability space such that the convergence in \eqref{eq:205} holds almost surely. Again let $\root= v_0^{(p)},v_1^{(p)},\dots,v_{2p}^{(p)}$ be the enumeration of vertices of $\cT^{(p)}$ in lexicographic order and let $\pi$ be the projection from $\cT^{(p)}$ onto $\overline{\Loop}(\cT^{(p)})$. 
  Let $\cT_{\bbe}$ be the CRT defined from the Brownian excursion $\bbe$ and let $\mathbf{p}_{\bbe}$ denote the canonical projection map from $[0,1]$ to $\cT_{\bbe}$. We now define a correspondence $\cR^{(p)}$ between ${\mathcal X}^{(p)}$ and $\cT_{\bbe}$ by
  \begin{equation}
    \label{eq:99}
    \cR^{(p)}=
    \left\{
      \big(\pi(v_i^{(p)}),\mathbf{p}_{\bbe}(s)\big)\in \mathcal X^{(p)}\times \cT_{\bbe} : i=\lfloor 2ps\rfloor,  s\in [0,1],i\in [\![0,2p]\!].
    \right\},
  \end{equation}
  Recall that the distortion of this correspondence is defined as
  \begin{equation}
    \label{eq:148}
    \dis(\cR^{(p)})=\sup\bigg\{\left|\frac{\theta_\ttF}{\sqrt{p}}d_{M^{(p)}}(v,w)-d_{\bbe}(s,t)\right|: (v,\mathbf{p}_{\bbe}(s)),(w,\mathbf{p}_{\bbe}(t))\in \cR^{(p)}\bigg\}.
  \end{equation}
 We now show that $\dis(\cR^{(p)})\rightarrow 0$ almost surely as $p\rightarrow \infty$, arguing by contradiction. If this convergence were not true, then there would be an $\epsilon>0$ such that with positive probability, there exists a sequence of integers $p_n$ with corresponding indices $i_{p_n},j_{p_n}$ and $s_{p_n},t_{p_n}$ for which $(v_{i_{p_n}}^{(p_n)},\mathbf{p}_{\bbe}(s_{p_n})),(v_{j_{p_n}}^{(p_n)},\mathbf{p}_{\bbe}(t_{p_n}))\in \cR^{(p_n)}$ and
  \begin{equation}
    \label{eq:100}
    \left|\frac{\theta_\ttF}{\sqrt{p}}d_{M^{(p_n)}}(v^{(p_n)}_{i_{p_n}},v^{(p_n)}_{j_{p_n}})-d_{\bbe}(s_{p_n},t_{p_n})\right|>\epsilon.
  \end{equation}
  Passing to a subsequence if necessary, we may assume that \begin{equation} \lim_{n\to\infty}\frac{i_{p_n}}{2p_n}=\lim_{n\rightarrow \infty}s_{p_n}=s \quad \text{and} \quad \lim_{n\to\infty}\frac{j_{p_n}}{2p_n}=\lim_{n\rightarrow \infty}t_{p_n}=t\end{equation} for some constants $0\leq s<t\leq 1$. 
  
  For $0 \leq i < j \leq 2p$, let $\mathrm m^{(p)}(i,j)\in [\![0,2p]\!]$ be the smallest index so that $v_{\mathrm m^{(p)}(i,j)}^{(p)}$ is the most recent common ancestor of $v_i^{(p)}$ and $v_j^{(p)}$ in $\cT^{(p)}$. Let $f_i^{(p)}\in \fF(\partial M^{(p)})$ be the face bounded by $\pi($children of $v_i^{(p)}$). We claim that 
  \begin{equation}
      \left| d_{M^{(p)}}(\pi(v_i^{(p)}),\pi(v_j^{(p)})) - \bigg( \mathrm{H}_M^{(p)}\bigg(\frac{i}{2p}\bigg) + \mathrm{H}_M^{(p)}\bigg(\frac{j}{2p}\bigg) - 2\mathrm{H}_M^{(p)}\bigg( \frac{\mathrm{m}^{(p)}(i,j)}{2p} \bigg) \bigg) \right| \leq 2\per(f^{(p)}_{\mathrm{m}^{(p)}(i,j)}).
  \end{equation}
  This is because if $w_i$ and $w_j$ are the children of $\mathrm{m}^{(p)}(i,j)$ which are the ancestors of $v_i^{(p)}$ and $v_j^{(p)}$, respectively, then a geodesic from $\pi(v_i^{(p)})$ to $\pi(v_j^{(p)})$ in $M^{(p)}$ must pass through $\pi(w_i)$ and $\pi(w_j)$. Hence, the difference in the left-hand side is equal to 
  \begin{equation}|d_{M^{(p)}}(\pi(w_i),\pi(w_j)) - d_{M^{(p)}}(\pi(w_i), \pi(v_{\mathrm m^{(p)}(i,j)}^{(p)})) - d_{M^{(p)}}(\pi(w_j),\pi(v_{\mathrm m^{(p)}(i,j)}^{(p)}))|.\end{equation}
  Since geodesics for the all three pairs must stay in $M^{(p)}|_{f^{(p)}_{\mathrm{m}^{(p)}(i,j)}}$, each of them are bounded above by the maximal distance between two boundary points of this face, which is at most $\per(f^{(p)}_{\mathrm{m}^{(p)}(i,j)})$.
  
  Recalling Proposition~\ref{prop:4}, since $p^{-1/2} \partial M^{(p)}$ converges in distribution to a multiple of the CRT with respect to the Gromov--Hausdorff distance, we have $p^{-1/2} \max_{i\in[\![0,2p]\!]} \per(f^{(p)}_i) \to 0$ almost surely as $p\to\infty$ (see \cite[Equation (15)]{KR20}).  
    Moreover, if we let $b^{(p)}(i):= 2i - \mathrm{H}_\cT^{(p)}(i/(2p)) \in [\![0,4p]\!]$ denote the first time that we encounter $v_i^{(p)}$ in the enumeration of vertices of $\cT^{(p)}$ in contour order, then     
    \begin{equation}
        \mathrm{H}_{\cT}^{(p)}\bigg( \frac{\mathrm{m}^{(p)}(i,j)}{2p} \bigg) = \mathrm{C}_\cT^{(p)}\bigg( \frac{b^{(p)}(\mathrm{m}^{(p)}(i,j))}{4p} \bigg) = \inf_{\frac{b^{(p)}(i)}{4p} \leq u \leq \frac{b^{(p)}(j)}{4p}} \mathrm{C}_\cT^{(p)}(u).
    \end{equation}
  Lemma~\ref{lem:31} implies $(\frac{b^{(p)}(2pt)}{4p})_{0\leq t\leq 1} \to (t)_{0\leq t\leq 1}$ almost surely in $C([0,1])$ and 
    \begin{equation}
        \sup_{0\leq t\leq 1} \left|\frac{b^{(p)}(2pt)}{4p} - t\right| + \sup_{0\leq t\leq 1} \left|\frac{\theta_\ttF}{\sqrt p}\mathrm{H}_M^{(p)}(t) - \frac{\sigma}{2\sqrt{2p}} \mathrm{H}_\cT^{(p)}(t) \right| \to 0
    \end{equation}
    almost surely as $p\to\infty$. Hence,
    \begin{equation}
        \lim_{n\to\infty} \frac{ \theta_\ttF }{\sqrt{p_n}}d_{M^{(p_n)}}\big(\pi(v_{i_{p_n}}^{(p_n)}), \pi(v_{j_{p_n}}^{(p_n)})\big) = \bbe_s + \bbe_t - 2\inf_{s\leq u\leq t} \bbe_u = d_{\bbe}(s,t) = \lim_{n\to\infty} d_{\bbe}(s_{p_n},t_{p_n})
    \end{equation}
    almost surely. This is a contradiction, thus proving $\dis(\mathcal R^{(p)}) \to 0$ as $p\to\infty$.
\end{proof}

\bibliographystyle{alpha}
{\small
\bibliography{stars}}

\end{document}